\documentclass{amsart}

\usepackage{amssymb,amsmath,amsthm}
\usepackage{enumerate}
\usepackage{eucal}
\usepackage{mathabx}
\usepackage{xcolor}
\usepackage{hyperref}
\usepackage{orcidlink}

\newtheorem{Theorem}{Theorem}[section]
\newtheorem{Proposition}[Theorem]{Proposition}
\newtheorem{Lemma}[Theorem]{Lemma}
\newtheorem{Corollary}[Theorem]{Corollary}
\newtheorem{TheoremI}{Theorem} 

\newtheorem{Claim}{Claim}

\theoremstyle{definition}
\newtheorem{Definition}[Theorem]{Definition}
\newtheorem{Fact}[Theorem]{Fact}

\newtheorem{Example}[Theorem]{Example}
\newtheorem{Remark}[Theorem]{Remark}

\usepackage{etoolbox}
\AtBeginEnvironment{proof}{\setcounter{Claim}{0}}

\newcommand{\cupdot}{\mathbin{\mathaccent\cdot\cup}}

\def\Ind#1#2{#1\setbox0=\hbox{$#1x$}\kern\wd0\hbox to 0pt{\hss$#1\mid$\hss}
\lower.9\ht0\hbox to 0pt{\hss$#1\smile$\hss}\kern\wd0}

\def\ind{\mathop{\mathpalette\Ind{}}}

\def\notind#1#2{#1\setbox0=\hbox{$#1x$}\kern\wd0
\hbox to 0pt{\mathchardef\nn=12854\hss$#1\nn$\kern1.4\wd0\hss}
\hbox to 0pt{\hss$#1\mid$\hss}\lower.9\ht0 \hbox to 0pt{\hss$#1\smile$\hss}\kern\wd0}

\def\dep{\mathop{\mathpalette\notind{}}}

\DeclareMathOperator{\wor}{\perp\hspace*{-0.4em}^{\mathit w}}
\DeclareMathOperator{\nwor}{\not\perp\hspace*{-0.4em}^{\mathit w}}
\DeclareMathOperator{\fwor}{\perp\hspace*{-0.4em}^{\mathit f}}
\DeclareMathOperator{\nfor}{\not\perp\hspace*{-0.4em}^{\mathit f}}
\DeclareMathOperator{\Mon}{\mathfrak C}
\DeclareMathOperator{\Aut}{Aut}
\DeclareMathOperator{\tp}{tp}

\DeclareMathOperator{\Ker}{Ker}
\DeclareMathOperator{\Th}{Th}
\DeclareMathOperator{\id}{id}

\usepackage[margin=30mm]{geometry}

\title{Weakly o-minimal types}
\author[S.\ Moconja]{Slavko Moconja\ \orcidlink{0000-0003-4095-8830}}
\address[S.\ Moconja]{University of Belgrade, Faculty of mathematics, Belgrade, Serbia}
\email{slavko@matf.bg.ac.rs}

\author[P.\ Tanovi\'c]{Predrag Tanovi\'c\ \orcidlink{0000-0003-0307-7508}}
\address[P.\ Tanovi\'c]{Mathematical Institute of the Serbian Academy of Sciences and Arts, Belgrade, Serbia}
\email{tane@mi.sanu.ac.rs}
\thanks{The authors are supported by the Science Fund of the Republic of Serbia,
grant 7750027--SMART}

\begin{document} 

\begin{abstract} 
We introduce and study weak o-minimality in the context of complete types in an arbitrary first-order theory. A type $p\in S(A)$ is weakly o-minimal if for some relatively $A$-definable linear order, $<$, on $p(\Mon)$   every relatively $L_{\Mon}$-definable subset of $p(\Mon)$ has finitely many convex components in $(p(\Mon),<)$. We establish many nice properties of weakly o-minimal types. For example, we prove that weakly o-minimal types are dp-minimal and share several properties of weight-one types in stable theories, and that a version of monotonicity theorem holds for relatively definable functions on the locus of a weakly o-minimal type. 
\end{abstract}

\maketitle

\section{Introduction}

In seminal papers \cite{PS0,PS} by Pillay and Steinhorn and \cite{KPS} by Knight, Pillay, and Steinhorn, the notions of o-minimal structures and theories were introduced, and a substantial theory of definable sets within the o-minimal framework was developed. In particular, they proved the Monotonicity Theorem and the Cell Decomposition Theorem, which are fundamental tools used in the analysis of definable sets in o-minimal structures. 
Extensive research has been conducted on o-minimal structures in the decades that followed,
and applications
in various areas outside logic, even outside mathematics, have been discovered. One direction of the research included generalization of the concept of o-minimality;
this involved relaxing the o-minimality assumption and developing a theory of definable sets that resembles, as closely as possible, that for o-minimal structures. 

Let $\mathcal M=(M,<,\dots)$ be an infinite linearly ordered first-order structure. $\mathcal M$ is {\em o-minimal} if every definable\footnote{Throughout the paper ``definable'' means ``definable with parameters''.} subset of $M$ is a finite union of points and open intervals (with endpoints in $M\cup\{\pm\infty\}$).  There are several generalizations of o-minimality, many of which can be found in Fujita's article \cite{Fujita}. Here, we mention only some of them:
\begin{itemize}
\item (Dickmann \cite{Dick}) $\mathcal M$ is {\em weakly o-minimal} if every definable subset of $M$ is a finite union of convex sets;
\item (Belegradek, Stolboushkin and Taitslin \cite{BST}) $\mathcal M$ is {\em quasi-o-minimal} if every definable subset of $M$ is a Boolean combination of $\emptyset$-definable sets, points, and open intervals;
\item (Kuda\u\i bergenov \cite{Kud}) $\mathcal M$ is {\em weakly quasi-o-minimal} if every definable subset of $M$ is a Boolean combination of $\emptyset$-definable sets and convex sets.
\end{itemize}
All of the above definitions refer to the distinguished linear order $<$; in fact, they depend on $<$ in the sense that a structure can be (weakly, quasi- or weakly quasi-) o-minimal with respect to $<$, but not such with respect to some other $\emptyset$-definable linear order. 
A complete first-order theory $T$ is o-minimal (weakly o-minimal, quasi-o-minimal, weakly quasi-o-minimal) with respect to a distinguished $\emptyset$-definable linear order if all models (equivalently, some $\aleph_0$-saturated model) of $T$ are such. 
It is well known that the o-minimality of the structure $\mathcal M$ (with respect to $<$) is always carried over to the theory $\Th(\mathcal M)$ (\cite{PS3}); however, this fails for weak, quasi-, and weak quasi-o-minimality.  
Let us also mention that the weak quasi-o-minimality of the theory does not depend on the choice of order (\cite[Theorem 1]{MTwmon}); in general, this fails for o-minimal, weakly o-minimal and quasi-o-minimal theories.

The Monotonicity Theorem for an o-minimal structure $(M,<,\dots)$, \cite[Theorem 4.2]{PS} states that for every definable function $f:M\to M$ there is a finite definable partition of $M$ into points and open intervals such that $f$ is either constant of strictly monotone on each member of the partition. This fails if the structure is not o-minimal; an explanation for this and a more detailed discussion of analogues of monotonicity and cell decomposition outside the o-minimal context can be found in Goodrick's article \cite{Goodrick2}.
Among the weak analogues is a ``local" version of monotonicity proved in the weakly o-minimal context by Macpherson, Marker, and Steinhorn in \cite{Mac}. Roughly speaking, it states that for every definable function $f:M\to\overline M$ there is a finite convex partition of $M$ such that $f$ is locally constant or locally strictly monotone on each of the convex parts of the partition.

\smallskip
In this paper, we start a systematic study of weak o-minimality transferred to the locus of a complete type in an arbitrary complete (possibly multi sorted) first-order theory $T$; we introduce weakly o-minimal types. The motivation for that comes from our recent work on geometric properties of complete 1-types in weakly quasi-o-minimal theories (\cite{MTwmon}) (all of which are weakly o-minimal) and on general model theoretic properties of (stationarily ordered, or) so-types (\cite{MT}); the latter include all weakly o-minimal types.
Our intention here is to present a self-contained comprehensive exposition of all the fundamental properties of weakly o-minimal types, encompassing both their geometric and general model-theoretical features. 
All relevant results from \cite{MT} and \cite{MTwmon} are adapted to the context of weakly o-minimal types, and novel results are established. For example, Theorem 1(i) is proved in \cite{MTwmon} for the case of complete 1-types in a weakly quasi-o-minimal theory; nevertheless, Theorems 1(ii), 2, and 3 are novel even in that context. Several results from Sections 4 and 5 related to forking independence and nonorthogonality of so-types appear in some form in \cite{MT}, yet the proofs presented here are conceptually new; the main novelty is the notion of the left and right genericity of a realization of an so-type over a set of parameters, which we introduce in Section 4.    
 
  This paper will serve as an introduction to our
forthcoming papers on Vaught's and Martin's conjectures for weakly quasi-o-minimal theories. However, we believe that the results presented here will prove useful elsewhere.

Let $\Mon$ be a monster model of $T$, let $A$ be a small subset of $\Mon$, and let $p\in S(A)$. We will consider orders $(p(\Mon),<)$, where $<$ is a relatively $A$-definable linear order on $p(\Mon)$ and say that $(p,<)$ is a weakly o-minimal pair over $A$ if every relatively definable subset of $p(\Mon)$ is the union of a finite number of convex sets; a type $p$ is weakly o-minimal if there exists such a pair.

Every weakly o-minimal pair $(p,<)$ (type $p$) is an so-pair (so-type), that is,  
for every relatively definable set $D\subseteq p(\Mon)$, (exactly) one of the sets $D$ and $p(\Mon)\smallsetminus D$ is left-eventual, and (exactly) one of them is right-eventual in $(p(\Mon),<)$. However, the weak o-minimality of the types is quite different from the concept of local weak o-minimality of densely ordered structures studied by Fujita in \cite{Fujita}.  

Weakly o-minimal types have several nice properties. For example, the weak o-minimality of a type does not depend on the choice of order, and a theory $T$ is weakly quasi-o-minimal (with respect to some $L$-definable linear order) if and only if every complete type $p\in S_1(T)$ is weakly o-minimal.  
Among the geometric properties of weakly o-minimal types, the most interesting one is a version of the Monotonicity Theorem formulated in Theorem \ref{Theorem1} below which, roughly speaking, says that every relatively definable function on the locus of a weakly o-minimal type is ``weakly monotone"; this is explained as follows: Given a linear order $(D,<)$ and an equivalence relation with convex classes $E$, we can define another linear order $<_E$ by reversing the order $<$ within each class, but leaving the classes originally ordered:
  $$x<_Ey \Leftrightarrow (E(x,y)\land y<x)\vee (\lnot  E(x,y)\land x<y).$$
 For an $\subseteq$-increasing sequence of convex equivalence relations $\vec E=(E_1,\dots, E_n)$ we can iterate this construction and define $<_{\vec E}:=(\ldots (<_{E_1})_{E_2}\dots)_{E_n}$. If $(D',<')$ is another linear order and $f:D\to D'$ is an increasing function, then we say that $f$ is $(<,<')$-increasing; similarly for $(<,<')$-decreasing and $(<,<')$-monotone. If $f$ is $(<_{\vec E},<')$-monotone for some sequence of convex equivalence relations $\vec E$, then we may think of $f$ as a weakly $(<,<')$-monotone function.
 
\begin{TheoremI}\label{Theorem1}(Weak monotonicity).
Suppose that $(p,<_p)$ is a weakly o-minimal pair over $A$, $(D,<)$ an $A$-definable linear order, and $f:p(\Mon)\to D$ a relatively $A$-definable non-constant function.  
\begin{enumerate}[(i)]
\item  There exists a strictly increasing sequence of relatively $A$-definable convex equivalence relations $\vec E=(E_0,\dots,E_n)$ on $p(\Mon)$ such that $f$ is $((<_p)_{\vec E},<)$-increasing. 
\item  There exists an increasing sequence of $A$-definable  convex equivalence relations $\vec F=(F_0,\dots,F_n)$ on $(D,<)$ such that $f$ is $(<_p,<_{\vec F})$-increasing. 
 \end{enumerate}
\end{TheoremI}

Under the assumptions of Theorem \ref{Theorem1}, 
as a consequence
of the $((<_p)_{\vec E},<)$-monotonicity of $f$, 
we obtain the following:

\begin{TheoremI}\phantomsection\label{Theorem2}Suppose that $(p,<_p)$ is a weakly o-minimal pair over $A$, $(D,<)$ an $A$-definable linear order, and $f:p(\Mon)\to D$ a relatively $A$-definable non-constant function. 
\begin{enumerate}[(i)]
    \item (Local monotonicity).
There exists a convex relatively $A$-definable equivalence $E$ on $p(\Mon)$, such that $E\neq \id_{p(\Mon)}$ and the restriction of $f$ to each $E$-class is constant or strictly $(<_p,<)$-monotone.  

\item (Upper monotonicity). 
There exists a convex relatively $A$-definable equivalence $E$ on $p(\Mon)$, such that $E\neq p(\Mon)^2$ and one of the following two conditions holds for all $x_1,x_2$ realizing $p$:
\begin{center}$ [x_1]_E<_p[x_2]_E\Rightarrow  f(x_1)< f(x_2)$ \ \ \ or \ \ \  $[x_1]_E<_p[x_2]_E\Rightarrow  f(x_1)> f(x_2)$.
\end{center}
\end{enumerate}
\end{TheoremI}

Assume for a moment that the underlying theory is weakly quasi-o-minimal and $f:\Mon\to D$. Then all complete 1-types are weakly o-minimal, and Theorem \ref{Theorem1} applies to each of them. Using a routine compactness argument, one obtains a finite definable partition of $\Mon$ such that the restriction of $f$ to each member of the partition is weakly monotone; that is exactly the content of our aforementioned result \cite[Theorem 2]{MTwmon}. Similarly, we can derive an interesting definable form of local monotonicity (and a less interesting form of upper monotonicity, which we omit). In the following theorem, we state specific versions adjusted to the context of weakly o-minimal theories; note that part (ii) is a version of \cite[Theorem 3.3]{Mac}.

\begin{TheoremI}\label{Theorem3}
Suppose that $\Th(\Mon,<,\dots)$ is weakly o-minimal, $(D,\triangleleft)$ is an $A$-definable linear order, and $f:\Mon\to D$ is an $A$-definable function. 
\begin{enumerate}[(i)]
\item There exists a finite convex $A$-definable partition $\mathcal C$ of $\Mon$ and an increasing sequence of $A$-definable convex equivalence relations $\vec E$ on $\Mon$ such that $f$ is $(<_{\vec E},\triangleleft)$-increasing on each member of the partition. 

\item There exists a finite convex $A$-definable partition $\mathcal C$ of $\Mon$ and a convex $A$-definable equivalence relation $E$ with at most finitely many singleton classes on $\Mon$ such that $E=\bigcup _{C\in\mathcal C}E_{\restriction C}$ and for all $C\in\mathcal C$ the restriction $f_{\restriction[a]_E}$ is constant or strictly $(<,\triangleleft)$-monotone uniformly for all $a\in C$.
\end{enumerate}
\end{TheoremI}

Among the general model theoretical properties of so-types and weakly o-minimal types that we will study are the forking independence and the weak and forking nonorthogonality ($\nwor$ and $\nfor$) of types with the same domain. We will reprove (from \cite{MT}) that both $\nwor$ and $\nfor$ are equivalence relations, and that 
$x\dep_A y$ is an equivalence relation on the set of realizations of so-types over $A$. The proofs presented here are significantly simpler and more intuitive than the original ones.
They are based on the concept of the left and right $\mathbf p$-genericity of a tuple $a \models p$ relative to a set of parameters $B$, where $\mathbf p = (p, <_p)$ is an so-pair over $A$: $a\models p$ is {\em right (left) $\mathbf p$-generic over $B$} if and only if the locus of $\tp(a/AB)$ is a final (initial) part of $(p(\Mon),<_p)$; by $B\triangleleft^{\mathbf p}a$ we denote that $a$ is right $\mathbf p$-generic over $B$.  
We prove that left and right $\mathbf p$-generic elements exist (over any small parameter set $B$), and that $a\ind_A B$ holds (meaning $\tp(a/AB)$ does not fork over $A$) if and only if $a$ is left or right $\mathbf p$-generic over $B$. 
The relation $\triangleleft^{\mathbf p}$, viewed as a binary relation on $p(\Mon)$, is particularly well behaved. For example, $a\models p$ is left $\mathbf p$-generic over $b\models p$ if and only if $b$ is right $\mathbf p$-generic over $a$.

\begin{TheoremI}\label{Theorem4}
Let $\mathbf p=(p,<_p)$ be an so-pair over $A$. Assume that $p$ is non-algebraic and let $\mathcal D_p=\{(x,y)\in p(\Mon)^2\mid x\dep_Ay\}$.  The structure 
$(p(\Mon), \triangleleft^{\mathbf p},<_p,\mathcal D_p)$ has the following properties:
\begin{enumerate}[(i)] 
    \item  $(p(\Mon), \triangleleft^{\mathbf{p}})$ is a strict partial order  and $(p(\Mon), <_p)$ is a linear extension of $(p(\Mon), \triangleleft^{\mathbf{p}})$; 
    \item The relations $\mathcal D_{p}$ and  the $\triangleleft^{\mathbf p}$-incomparability are the same, $<_p$-convex equivalence relation on $p(\Mon)$;  \item The quotient $(p(\Mon) / \mathcal{D}_p, <_p)$ is a dense linear order.
\end{enumerate}   
\end{TheoremI}

Two so-pairs over $A$,  $\mathbf p=(p,<_p)$ and $\mathbf q=(q,<_q)$, are weakly nonorthogonal if $p\nwor q$; they are {\em directly nonorthogonal}, denoted by $\delta_A(\mathbf p,\mathbf q)$, if $p\nwor q$ and for all $a\models p$ and $b\models q$: $a$ is left $\mathbf p$-generic over $b$ if and only if $b$ is right $\mathbf q$-generic over $a$. Intuitively, $\delta_A(p,q)$ means that orders $<_p$ and $<_q$ have ``the same direction''; this will be justified by proving that $\delta_A$ is an equivalence relation that divides each $\nwor$ class into two classes such that pairs $\mathbf p=(p,<)$ and $\mathbf p^*=(p,>)$ are in distinct $\delta_A$ classes. Then the $\mathbf p$-genericity transfers to
$\mathcal F$-genericity, where $\mathcal F$ is a $\delta_A$-class in the following way: Let $\mathcal{F}(\Mon)$ be the set of realizations of all types that appear in $\mathcal F$. Then $a\in \mathcal{F}(\Mon)$ is right $\mathcal F$-generic over $B$, denoted by $B\triangleleft^{\mathcal F} a$, if $B\triangleleft^{\mathbf p} a$ holds for some (we show: equivalently, all) $\mathbf p=(\tp(a/A),<)\in\mathcal F$. The relation $\triangleleft^{\mathcal F}$ has many nice properties, collected in Theorem \ref{Thm_triangle_mathcal F}; for example, transitivity holds: $B\triangleleft^{\mathcal F}a\land a\triangleleft^{\mathcal F}b$ implies $B\triangleleft^{\mathcal F} b$. The following generalizes Theorem \ref{Theorem4}.  

\begin{TheoremI}\label{Theorem5}
Let $\mathcal F$ be a $\delta_A$-class of non-algebraic so-pairs over $A$. Let $\mathcal{F}(\Mon)$ be the set of realizations of all types that appear in $\mathcal F$ and let $\mathcal D_{\mathcal F}=\{(x,y)\in \mathcal{F}(\Mon)^2\mid x\dep_A y\}$. Then there is an $A$-invariant linear order $<^{\mathcal F}$ on $\mathcal F(\Mon)$ such that the structure
$(\mathcal F(\Mon),\triangleleft^{\mathcal F}, <^{\mathcal F},\mathcal D_{\mathcal F})$ has the following properties:  
\begin{enumerate}[(i)]
\item $(\mathcal{F}(\Mon), \triangleleft^{\mathcal F})$ is a strict partial order and $(\mathcal{F}(\Mon), <^{\mathcal F})$ its linear extension; 

\item Relations $\mathcal D_{\mathcal F}$ and  the $\triangleleft^{\mathcal F}$-incomparability are the same, $<^{\mathcal F}$-convex equivalence relation on $\mathcal{F}(\Mon)$.  

\item The quotient $(\mathcal{F}(\Mon)/\mathcal D_{\mathcal F},<^{\mathcal F})$ is a dense linear order.
\item For each $p\in S(A)$ represented in $\mathcal F$, the order $<_p=(<^{\mathcal F})_{\restriction p(\Mon)}$ is relatively $A$-definable and the pair $(p,<_p)$ is directly non-orthogonal to pairs from $\mathcal F$.    $(p(\Mon), \triangleleft^{\mathbf p},<_p,\mathcal D_p)$ is a substructure of $(\mathcal F(\Mon),\triangleleft^{\mathcal F}, <^{\mathcal F},\mathcal D_{\mathcal F})$.  
\end{enumerate}
Moreover, if $(\mathbf p_i=(p_i,<_i)\mid i\in I)$ is a family of pairs from $\mathcal F$, with the types $(p_i\mid i\in I)$ mutually distinct, then the order $<_{\mathcal F}$ can be chosen to extend each $<_i$ for $i\in I$. 
\end{TheoremI}

In general, the order $<^{\mathcal F}\supseteq \triangleleft^{\mathcal F}$ in Theorem \ref{Theorem5} cannot be chosen as the restriction of a definable one, even if the underlying theory is o-minimal. However, for the purpose of studying forking dependence as a binary relation, the order $<^{\mathcal F}$ is more adequate then definable ones.

As weakly o-minimal types are so-types, the conclusion of Theorem \ref{Theorem5} applies to $\mathcal F$, a $\delta_A$-class of weakly o-minimal pairs.
In that context, we prove several additional nice properties of $\triangleleft^{\mathcal F}$; for example, the convexity property: \ $B\triangleleft^{\mathcal F}a<^{\mathcal F}b<^{\mathcal F}c\land a\dep_{B}c$ implies  $b\dep_{B} a \land b\dep_{B}c$. Since complete extensions of weakly o-minimal types are also weakly o-minimal, it makes sense to consider nonorthogonality of nonforking extensions of pairs from $\mathcal F$ over a larger domain $B\supseteq A$. In particular, if the extensions are taken in the same direction, we will prove that each of $\nwor$, $\nfor $, and $\delta$, is maintained.

\smallskip
The paper is organized as follows.  Section \ref{Section1} contains the preliminaries. We set the terminology related to model theory and linear orders and outline the fundamental properties of orders $<_{\vec E}$. Relatively definable subsets are studied in more detail as they play an important role in the subsequent sections. Additionally, we revisit the definition and elementary properties of so-types from \cite{MT}. 
In Section \ref{Section2}, weakly o-minimal types are introduced and some basic properties are established.
Section \ref{Section3} provides a characterization of all relatively definable linear orders on the locus of a weakly o-minimal type, followed by several results, including Theorems \ref{Theorem1}--\ref{Theorem3}. Among other facts, we prove that the weak o-minimality of $p\in S(A)$ does not depend on the choice of order: If $p\in S(A)$ is weakly o-minimal, then $(p,<)$ is a weakly o-minimal pair for all relatively $A$-definable orders $<$. We also prove that a complete 1-sorted theory $T$ is weakly quasi o-minimal (with respect to some  $\emptyset$-definable linear order) if and only if every type $p\in S_1(T)$ is weakly o-minimal.  
In Section \ref{Section4}, we introduce and study $\mathbf{p}$-genericity and prove Theorem \ref{Theorem4}. In Section \ref{Section5},
we study direct nonorthogonality of so-pairs and prove Theorem \ref{Theorem5}. For the case of weakly o-minimal pairs, we prove several additional nice properties of $\triangleleft^{\mathcal F}$. 
In Section \ref{Section6}, we prove that weakly o-minimal types are dp-minimal. We also provide some analysis of indiscernible sequences of realizations of weakly o-minimal types; in particular, 
we show that they have a bit stronger property than the
distality (as defined by Simon in \cite{Simon2}).

\section{Preliminaries}\label{Section1}

We use standard concepts and notation from model theory. 
We work in $\Mon$, a large, saturated ({\em monster}) model of a complete, first-order (possibly multi-sorted) theory $T$ in a first-order language $L$. 
The singletons and tuples of elements from $\Mon$ are denoted by $a,b,c,\dots$; $|a|$ denotes the length of the tuple $a$.  The letters $A,B,A', B',\dots$ are reserved for small subsets (of cardinality $<|\Mon|$) of the monster, while $C,D,C',D',\dots$ are used to denote arbitrary sets of tuples. 
By an $L_C$-formula $\phi(x)$ we mean a formula whose parameters are from $C$; by a formula we mean an $L_{\Mon}$-formula.
The set of all the realizations of $\phi(x)$ in $\Mon$ is usually denoted by $\phi(\Mon)$, but sometimes, when we want to emphasize $|x|=n$, we also denote it by $\phi(\Mon^{n})$. 
Sets of this form are said to be {\em $C$-definable}; a set is {\em definable} if it is $C$-definable for some parameter set $C$. 
A {\em partial type} $p(x)$ is any small finitely consistent set of $L_{\Mon}$-formulae that is closed under conjunctions; $p(\Mon)$ (or $p(\Mon^{|x|})$) denotes the set of all realizations of $p(x)$. A subset $D\subset \Mon^n$ is {\em type-definable} (over $A$) if $D=p(\Mon)$ for some partial type $p(x)$ (over $A$).

The space of all {\em complete} $n$-types over the parameters $C$  is denoted by $S_n(C)$; the basic clopen subsets are of the form $[\phi]=\{p\in S_n(C)\mid \phi(x)\in p\}$, where $\phi(x)$ is a $L_C$-formula and $|x|=n$. $S_x(C)$ denotes the set of types in variable $x$. $S(C):=\bigcup _{n\in\mathbb N}S_n(C)$. In particular, $S(\Mon)$ is the set of all {\em global} finitary types. 
A global type $\mathfrak p(x)$ is {\em $A$-invariant} if $(\phi(x,b_1)\leftrightarrow \phi(x,b_2))\in\mathfrak p$ for all $L_A$-formulae $\phi(x,y)$ and all tuples $b_1,b_2$ of length $|y|$ that satisfy $b_1\equiv b_2\,(A)$; the type $\mathfrak p$ is {\em invariant} if it is $A$-invariant for some small set of parameters $A$.  For an $A$-invariant global type $\mathfrak p$ and a linear order $(I,<)$, a sequence of tuples $(a_i\mid i\in I)$ is a {\em Morley sequence} in $\mathfrak p$ over $A$ if $a_i\models \mathfrak p_{\restriction A\,a_{<i}}$ holds for all $i\in I$. Note that we allow Morley sequences to have an arbitrary (even finite) order-type.
Dividing and forking have the usual meaning, and by $a\ind_A B$ we denote that $\tp(a/AB)$ does not fork over $A$. 
The types $p,q\in S(A)$ are {\em weakly orthogonal}, denoted by $p\wor q$, if $p(x)\cup q(y)$ determines a complete type over $A$; they are {\em forking orthogonal}, denoted by $p\fwor q$, if $a\ind_A b$ holds for all $a\models p$ and $b\models q$. 

\subsection{Linear orders} Notation related to linear orders is mainly standard. Let $(X,<)$ be a linear order, and let $D\subseteq X$.

\begin{enumerate}[(1)]
\item[$\bullet$] $D$ is {\em convex}  if $a,b\in D$ and $a<c<b$ imply $c\in D$.
\item[$\bullet$] $D$ is an {\em initial part} if $a\in D$ and $b<a$ imply $b\in D$;  $D$ is a {\em left-eventual part} if it contains a non-empty initial part. {\em The final parts} and {\em right-eventual} parts are defined dually.
\item[$\bullet$] A subset $C\subseteq D$ is a {\em convex component} of $D$ if $C$ is a maximal convex subset of $D$.
\item[$\bullet$] The set of all convex components forms a partition of $D$; therefore, the meaning of {\em $D$ has a finite number of convex components} is clear. 
\item[$\bullet$] For non-empty $Y,Y'\subseteq X$ we write $Y<Y'$ if $y<y'$ holds for all $y\in Y$ and $y'\in Y'$; we write $x<Y$ and $Y<x$ instead of $\{x\}<Y$ and $Y<\{x\}$. 
\item[$\bullet$] $x\in X$ is an {\em upper (lower) bound} of $D$ if $D<x$ ($x<D$). 
\item[$\bullet$] $D$ is {\em upper (lower) bounded} if an upper (lower) bound exists; $D$ is {\em bounded} if it is both upper and lower bounded. Note that the set of all upper (lower) bounds of $D$ is a final (initial) part. 
\item[$\bullet$] $\sup D_1\leqslant \sup D_2$ \ (where $D_1,D_2\subseteq X$) denotes that any upper bound of $D_2$ is an upper bound of $D_1$ too; $\sup D_1<\sup D_2$, $\inf D_1\leqslant \inf D_2$ and $\inf D_1<\inf D_2$ have analogous meanings. 
\item[$\bullet$] An equivalence relation $E\subseteq X\times X$ is {\em convex} if all $E$-classes $[x]_E$ ($x\in X$) are convex subsets of $X$; in that case, the quotient set $X/E$ is naturally linearly ordered by $<$.
\end{enumerate}
Note that all of the above definitions should be read as {\em with respect to $(X,<)$}. Further in the text, whenever the meaning of the order is not clear from the context, we will emphasize it in some way; for example, we say that {\em $D$ is convex in $(X,<)$} or that {\em $D$ is a $<$-convex subset of $X$}, etc.

If $(X,<_X)$ and $(Y,<_Y)$ are linear orders and $f:X\to Y$, we say that $f$ is {\em $(<_X,<_Y)$-increasing} if $x<_Xx'$ implies $f(x)\leqslant_Yf(x')$ for all $x,x'\in X$. 
In that case, the kernel relation $\Ker(f)$, defined by $f(x)=f(x')$, is a convex equivalence relation on $(X,\leqslant_X)$, and the mapping defined by $[x]_{\Ker(f)}\mapsto f(x)$ is an order isomorphism between $(X/\Ker(f),<_X)$ and $(f(X),<_Y)$. Also, $f$ is {\em strictly $(<_X,<_Y)$-increasing} if $x<_Xx'$ implies $f(x)<_Yf(x')$ for all $x,x'\in X$.

In the remainder of the subsection, we recall the construction of orders $<_{\vec E}$ from \cite{MTwmon} and state their basic properties. 

\begin{Definition} 
Let $(X,<)$ be a linear order and $E$ a convex equivalence relation on $X$. Define:
$$x<_Ey\ \mbox{ iff }\ (E(x,y)\land y<x) \vee(\neg E(x,y)\land x<y).$$
\end{Definition}

It is easy to see that $(X,<_E)$ is a linear order; the order $<_E$ reverses the order $<$ within each $E$-class, but the classes in the quotient order remain originally ordered. In particular, $E$ is a $<_E$-convex equivalence relation, and $(X/E,<)=(X/E,<_E)$ holds. Furthermore, let $E'$ be another convex equivalence relation on $(X,<)$ that is $\subseteq$-comparable (either finer or coarser) to $E$. It is easy to see that $E'$ is a $<_E$-convex equivalence, so the order $(<_E)_{E'}$ is well defined. Similarly, the order $(<_{E'})_E$ is well-defined; it is not hard to see that $(<_E)_{E'}=(<_{E'})_E$ holds.

\begin{Definition}\label{Definition <E} 
Let $(X,<)$ be a linear order and $\vec E=(E_1,\dots,E_{n})$ a sequence of convex equivalence relations on $X$, such that any two of them are $\subseteq$-comparable. Define:
     $$<_{\vec E}:=(\dots(<_{E_1})_{E_2}\dots)_{E_{n}}.$$ 
\end{Definition}

\begin{Remark}\phantomsection\label{Remark_order_<_vecE}
\begin{enumerate}[(a)]
\item If the order $(X,<)$ and the sequence $\vec E$ in the previous definition are definable, then the resulting order $<_{\vec E}$ is definable with the same parameters. Similarly, if $X$ is type-definable over $A$ and $<$ and $\vec E$ are relatively $A$-definable, then $<_{\vec E}$ is relatively $A$-definable.

\item We have already remarked that $<_{(E_1,E_2)}=<_{(E_2,E_1)}$ holds for any pair of $\subseteq$-comparable convex equivalences.  By induction, if $\vec E$ is a sequence of convex equivalences such that any two of them are $\subseteq$-comparable, it is easy to prove that the order $<_{\vec E}$ does not depend on the order of elements of $\vec E$: $<_{\vec E}=<_{\pi(\vec E)}$ holds for any permutation $\pi(\vec E)$ of $\vec E$. 

\item It is very easy to see that $(<_E)_E=<$ always holds and is only slightly harder to verify $(<_{\vec E})_{\vec E}=<$. Indeed, we have  $(<_{\vec E})_{\vec E}=<_{(E_1,E_2\dots, E_n,E_1,\dots,E_n)}=  
<_{(E_1,E_1,\dots, E_n,E_n)}=<$;  here, the second equality holds by (b).
As a consequence, we have the following. If $\triangleleft$ is another linear order on $X$,  then  $<=\triangleleft_{\vec E}$ and  $<_{\vec E}=\triangleleft$ are equivalent. 

\item It is rather straightforward to verify that $<_{\vec E}\neq <_{\vec E'}$ holds for any pair of distinct strictly increasing sequences of convex equivalence relations $\vec E$ and $\vec E'$ that do not contain the identity relation.

\item Let $D\subseteq X$ be $<$-convex. Then $D$ properly intersects at most two $E_1$-classes (the endpoints of $D/E_1$ in the quotient order $X/E_1$). Thus, $D$ has at most three $<_{E_1}$-convex components. Each of these components has at most three $(<_{E_1})_{E_2}$-components, etc. Thus, $D$ can have a maximum of $3^n$ $<_{\vec E}$-convex components. Taking into account (c), it follows that a subset $D\subseteq X$ has finitely many $<$-convex components if and only if $D$ has finitely many $<_{\vec E}$-convex components. 
\end{enumerate}
\end{Remark}

\subsection{Relative definability}

\begin{Definition}
Let $p(x)$ be a partial type over $A$.  
A set $X\subseteq p(\Mon)$ is {\em relatively $B$-definable within $p(\Mon)$} if $X=D\cap p(\Mon)$ holds for some $B$-definable set $D\subseteq \Mon^{|x|}$. In that case, any formula $\phi(x)$ that defines $D$ is called a {\em relative definition of $X$ within $p(\Mon)$}, and we also say that $X$ is relatively defined by $\phi(x)$  within $p(\Mon)$.
\end{Definition}

Clearly, the family of relatively $B$-definable subsets of a type-definable set is closed for finite Boolean combinations. Also,
if the subset $P\subseteq \Mon^n$ is type-definable over $A$, then so is any finite power of $P$. Therefore, the relative definability of the relations on $P$ is well defined. For example, if a relation $R\subseteq P^2$ is relatively defined by a formula $\phi(x,y)$ and if $(P,R)$ is a linear order, then we say that {\em $\phi$ relatively defines a linear order on $P$}. Similarly, if $P\subseteq \Mon^n$  and $Q\subseteq \Mon^m$ are type-definable sets, then so is the set $P\times Q$ and the relative definability of (graphs of) functions $f:P\to Q$ is well defined.

Several interesting properties of relatively definable relations on $p(\Mon)$ can be transferred to a definable neighborhood $\theta(\Mon) \supseteq p(\Mon)$ (where $\theta(x)\in p$).  For example, assume that $\phi(x,y)$ relatively defines a pre-order on $p(\Mon)$. Let \ $\psi(x,y,z):= \phi(x,x)\land (\phi(x,y)\land\phi(y,z)\rightarrow \phi(x,z))$. Then $p(x)\cup p(y)\cup p(z)\vdash \psi(x,y,z)$, so by compactness there exists a $\theta(x)\in p(x)$ such that $\{\theta(x),\theta(y),\theta(z)\}\vdash \psi(x,y,z)$.  
Therefore, $\phi(x,y)$ relatively defines a pre-order on $\theta(\Mon)$ and $(p(\Mon),\phi(p(\Mon)))$ is a suborder. 

The key point in the above argument is that the theory of pre-orders is universally axiomatizable, so that the property ``$\phi(x,y)$ relatively defines a pre-order on $p(\Mon)$'' can be expressed by a sentence  saying that ``$\psi(x,y,z)$ holds for all $x,y,z$ realizing $p$''. 
Formally, this is expressed by the following $L_{\infty,\omega}$-sentence: $(\forall x,y,z)\left(\bigwedge_{\theta\in p}(\theta(x)\land \theta(y)\land \theta(z))\rightarrow \psi(x,y,z)\right)$, which will be informally denoted by $(\forall x,y,z\models p)\ \psi(x,y,z)$. More generally, we will consider $L_{\infty,\omega}$-sentences denoted informally by $(\forall x_1\models p_1)\dots(\forall x_n\models p_n)\ \psi(x_1,\dots,x_n)$, where $p_1(x_1),\dots,p_n(x_n)$ are partial types and $\psi(x_1,\dots,x_n)$ an $L_{\Mon}$-formula, and call them {\em $\tp$-universal sentences}; 
the properties of relations (and their defining formulae) expressed by these sentences are called {\em $\tp$-universal properties}. For example, ``$\phi(x,y)$ relatively defines a pre-order on $p(\Mon)$'' and ``$\leqslant$ is a relatively definable pre-order on $p(\Mon)$'' are $\tp$-universal properties. The following is a version of the compactness that will be applied further in the text when dealing with $\tp$-universal properties. 

\begin{Fact}\label{Fact_L_P_sentence}
Suppose that $p_1(x_1),\dots,p_n(x_n)$ are partial types and $\phi(x_1,\dots,x_n)$ is an $L_{\Mon}$-formula such that $\Mon\models (\forall x_1\models p_1)\dots(\forall x_n\models p_n)\ \phi(x_1,\dots,x_n)$. Then there are   $\theta_i(x_i)\in p_i$ ($1\leqslant i\leqslant n$) such that: 
$$\Mon\models (\forall x_1\dots  x_n)\left(\bigwedge_{1\leqslant i\leqslant n}\theta_i'(x_i)\rightarrow \phi(x_1,\dots,x_n)\right)$$
for all formulae $\theta_i'(x_i)\in p_i$ such that $\theta_i'(\Mon)\subseteq \theta_i(\Mon)$ ($1\leqslant i\leqslant n$).  
\end{Fact} 

\begin{Remark}\phantomsection\label{Remk_conjunct_tpuniversal}
A finite conjunction of tp-universal sentences is equivalent to a tp-universal sentence. For example, $(\forall x\models p)(\forall y\models q)\phi(x,y)\land (\forall x\models r)(\forall y\models p)\psi(x,y)$ is equivalent to $(\forall x,t\models p)(\forall y\models q)(\forall z\models r)(\phi(x,y)\land \psi(z,t))$.
\end{Remark}

One typical application of Fact \ref{Fact_L_P_sentence} is the following. Let $\mathcal P= (p(\Mon);R_1,\dots,R_n)$, where $p(x)$ is a partial type over $A$ and where each $R_i$ is a relatively $A$-definable finitary relation on $p(\Mon)$; let $\phi_i$ relatively define $R_i$. 
Suppose that some interesting property of $\mathcal P$ can be expressed by a tp-universal sentence
$(\forall x_1,\dots,x_m\models p)\ \psi(x_1,\dots,x_m)$, where the formula $\psi\in L_A$ is built from $\phi_1,\dots,\phi_n$ (viewed as atomic);
then Fact \ref{Fact_L_P_sentence} produces $A$-definable
superstructures $(\theta(\Mon);\phi_1(\theta(\Mon)),\dots,\phi_n(\theta(\Mon)))$ of $\mathcal P$ with the same property; we will call them {\em definable extensions} of $\mathcal P$. 
In all future applications, we will fix the sequence of all relevant formulas before applying Fact \ref{Fact_L_P_sentence}.

 \begin{Example}\label{Examples of tp universal}
\begin{enumerate}[(a)]
\item Let $\mathcal P=(p(\Mon);<)$ be a relatively $A$-definable linear order; we will always assume that $<$ is defined by the formula $x<y$. Clearly, that $\mathcal P$ is a linear order is expressible by a $\tp$-universal sentence (built from $x<y$), so by Fact \ref{Fact_L_P_sentence} there is $\theta(x)\in p$ such that $x<y$ defines a linear order, also denoted by $<$, on $\theta(\Mon)$; $(\theta(\Mon);<)$  is a definable extension of $\mathcal P$.

\item Consider $\mathcal P=(p(\Mon);<,E)$, where $<$ is a linear order and $E$ is a convex equivalence relation; we will always implicitly assume that $E$ is relatively defined by $E(x,y)$. Each of the following properties: ``$x<y$ defines a linear order on $p(\Mon)$", ``$E(x,y)$ defines an equivalence relation on $p(\Mon)$" and ``$E$-classes are $<$-convex subsets of $p(\Mon)$" is a tp-universal property (the latter is expressed by $(\forall x,y,z\models p)(E(x,y)\land x<z<y\rightarrow E(x,z))$). By Remark \ref{Remk_conjunct_tpuniversal}, the conjunction of these properties is also tp-universal, so by Fact \ref{Fact_L_P_sentence}, there exists an $A$-definable extension $(\theta(\Mon);<,\hat E)$ of $\mathcal P$ such that $<$ is a linear order and $\hat E$ is a convex equivalence relation on $\theta(\Mon)$. 
 
\item Let $\mathcal P=(p(\Mon);<,E_1,\dots,E_n)$ be a relatively $A$-definable linear order with a $\subseteq$-increasing sequence $(E_1,\dots,E_n)$ of relatively $A$-definable $<$-convex equivalence relations (each $E_i$ defined by an $L_A$-formula $E_i(x,x')$). Again, we can describe this by a $\tp$-universal sentence, so an $A$-definable extension $(\theta(\Mon),<,\hat E_1,\dots,\hat E_n)$ of $\mathcal P$ can be found such that $(\hat E_1,\dots,\hat E_n)$ is a $\subseteq$-increasing sequence of $<$-convex equivalence relations on $\theta(\Mon)$. 
\end{enumerate}
\end{Example}

Whenever $f : p(\Mon) \to q(\Mon)$ is a relatively definable function (graph), it will always be implicitly assumed that (the graph of) $f$ is relatively defined by the formula $f(x,y)$ (or $f(x) = y$).
In general, we cannot express by a tp-universal sentence that $f\subseteq p(\Mon)\times q(\Mon)$ is a function, unless $q(\Mon)$ is definable. However, one part of the conclusion of \ref{Fact_L_P_sentence} always holds: $f$ has a definable extension, which is relatively defined by $f(x,y)$.

\begin{Fact}\label{Fact rel def kernel and inverse image} 
Let $p(x)$ and $q(y)$ be partial types over $A$ and $f:p(\Mon)\to q(\Mon)$ a relatively $A$-definable function. Then:
\begin{enumerate}[(i)]
\item There are $\theta_p(x)\in p$ and $\theta_q(y)\in q$ and a definable extension of $f$, $\hat f:\theta_p(\Mon)\to\theta_q(\Mon)$, relatively defined by $f(x,y)$.
\item The image, $f(p(\Mon))\subseteq q(\Mon)$, is type definable over $A$. The inverse image of a relatively $A$-definable subset of $q(\Mon)$ is a relatively $A$-definable subset of $p(\Mon)$.
\item The kernel relation of $f$, $\Ker f$, defined by $f(x)=f(x')$, is a relatively $A$-definable equivalence relation on $p(\Mon)$.
\end{enumerate}
\end{Fact}
\begin{proof}
(i) The sentence $(\forall x,x'\models p)(\forall y\models q)(f(x,y)\land f(x',y)\rightarrow x=x')$ expresses that $f(x,y)$ relatively defines a partial function. By Fact \ref{Fact_L_P_sentence} there is a formula $\theta_q(y)\in q$ such that
$f(x,y)$ relatively defines a partial function from $p(\Mon)$ to $\theta_q(\Mon)$; clearly, this function is total. Now,   ``$f(x,y)$ relatively defines a function $p(\Mon)\to\theta_q(\Mon)$" is expressed by 
$(\forall x\models p)(\exists_1 y)(\theta_q(y)\land f(x,y))$. By Fact \ref{Fact_L_P_sentence} there is a $\theta_p(x)\in p$ such that $f(x,y)$ relatively defines a function $\hat f:\theta_p(\Mon)\to\theta_q(\Mon)$. 
This proves (i).

(ii) Fix $\hat f:\theta_p(\Mon)\to\theta_q(\Mon)$, an $A$-definable extension of $f$, given by (i). It is easy to see that the type $\{(\exists x)(\hat f(x)=y\land \psi(x))\mid \psi(x)\in p(x)\}$ defines $f(p(\Mon))$, and that if $D\subseteq q(\Mon)$ is relatively defined by $\phi(y)$, then the set $f^{-1}(D)$ is relatively defined by $(\exists y)(\hat f(x)=y\land \theta_p(x)\land \theta_q(y)\land \phi(y))$. 

(iii) Clearly, $\hat f(x)=\hat f(x')$ relatively defines $\Ker f$.  
\end{proof}

\begin{Remark}
(a) Relatively definable functions with domain $p(\Mon)$ can be viewed as "germs" of (partial) definable functions at $p(\Mon)$, where two definable functions $f,g$ are said to have the same germ at $p$, $f\sim_p g$, if there is a definable $D\supseteq p(\Mon)$ such that $f(x)=g(x)$ for all $x\in D$. This is an equivalence relation, whose classes naturally correspond to relatively definable functions defined on $p(\Mon)$.

\smallskip
(b) When dealing with tp-universal properties that involve a relatively definable function $f:p(\Mon)\to q(\Mon)$, we will proceed similarly to the proof of Fact \ref{Fact rel def kernel and inverse image}(i). First, we will collect all relevant tp-universal properties of relatively definable relations on $q(\Mon)$ and apply Fact \ref{Fact_L_P_sentence} to find an appropriate $\theta_q (y)\in q$ such that, in addition, the formula $f(x,y)$ relatively defines a function from $p(\Mon)$ to $\theta_q(\Mon)$. Then we proceed with this function and other properties involving $p(x)$. This is illustrated in the next example. 
\end{Remark}
  
\begin{Example}
Suppose that $p(x)$ and $q(y)$ are partial types over $A$, $<_p$ and $<_q$ are relatively $A$-definable orders on $p(\Mon)$ and $q(\Mon)$, respectively, and $f:p(\Mon)\to q(\Mon)$ is a relatively $A$-definable strictly $(<_p,<_q)$-increasing function. We find $\theta_p(x)\in p$ and $\theta_q(y)\in q$, such that the $A$-definable structure determined by $\theta_p(x)$, $\theta_q(y)$, $x<_px'$, $y<_qy'$ and $f(x,y)$ has the properties listed above, as follows. First, apply Fact \ref{Fact_L_P_sentence} to:

-- $y<_q y'$ relatively defines a linear order on $q(\Mon)$, and 

-- $f(x,y)$ relatively defines a partial function from $p(\Mon)$ into $q(\Mon)$. 

\noindent
Let $\theta_q(y)\in q$ be such that $y<_q y'$ defines a linear order on $\theta_q(\Mon)$ and $f(x,y)$ relatively defines a function from $p(\Mon)$ into $\theta_q(\Mon)$. The desired formula $\theta_p(x)\in p$ is obtained by applying Fact \ref{Fact_L_P_sentence} to:

-- $x<_px'$ relatively defines a linear order on $p(\Mon)$, and

-- $f(x,y)$ relatively defines a strictly $(<_p,<_q)$-increasing function from $p(\Mon)$ into $\theta_q(\Mon)$. 
\end{Example}

\subsection{Stationarily ordered types}

 In this subsection, we recall basic facts about stationarily ordered, or so-types, from \cite{MT}.

\begin{Definition}\label{Definition_so} A complete type $p\in S(A)$ is {\em stationarily ordered} (or {\em so-type} for short),  if there exists a relatively $A$-definable linear order $<$ on $p(\Mon)$ such that for every relatively definable set $D\subseteq p(\Mon)$, (exactly) one of the sets $D$ and $p(\Mon)\smallsetminus D$ is left-eventual, and (exactly) one of them is right-eventual in $(p(\Mon),<)$. In that case we say that $(p,<)$ is an {\em so-pair} over $A$. 
\end{Definition}

In general, so-types are not NIP, see Example 3.2 in \cite{MT}. However, as we will see further in the paper,  they share several nice properties with dp-minimal types,  even with weight-one types in stable theories.

\begin{Definition} For an so-pair $\mathbf p=(p,<)$, define the left ($\mathbf p_l$) and the right ($\mathbf p_r$) globalization of $\mathbf p$: 
    
$\mathbf p_l(x):=\{\phi(x)\in L_{\Mon}\mid \phi(\Mon)\cap p(\Mon) \mbox{ is left-eventual in } (p(\Mon),<)\}$ \ and

$\mathbf p_r(x):=\{\phi(x)\in L_{\Mon}\mid \phi(\Mon)\cap p(\Mon) \mbox{ is right-eventual in } (p(\Mon),<)\}$.
\end{Definition}

We collect the basic properties of the defined globalizations. These were proved in {\cite[Remark 3.5, Lemma 3.6]{MT}}. However, to maintain the completeness of the presentation, we provide a proof here.

\begin{Fact}\label{Fact_so_basic}
Let $\mathbf p=(p,<)$ be an so-pair over $A$. 
\begin{enumerate}[(i)] 
    \item Both $\mathbf p_l$ and $\mathbf p_r$ are complete global types that extend $p$.
    \item Both $\mathbf p_l$ and $\mathbf p_r$ are $A$-invariant types; in particular, they are nonforking extensions of $p$. Moreover, $\mathbf p_l$ and $\mathbf p_r$ are the only $A$-invariant globalizations of $p$.
    
    \item For all $B\supseteq A$ the locus $\mathbf (\mathbf p_ r)_{\restriction B}(\Mon)$ is a final part of $(p(\Mon),<)$, and the locus $(\mathbf p_l)_{\restriction B}(\Mon)$ is an initial part of $(p(\Mon),<)$. 
    \item For all $a,b\models p$: $a\models(\mathbf p_ l)_{\restriction Ab}\Leftrightarrow b\models(\mathbf p_r)_{\restriction Aa}$. In other words, $(a,b)$ is a Morley sequence in $\mathbf p_r$ over $A$ iff $(b,a)$ is a Morley sequence in $\mathbf p_l$ over $A$, or $(\mathbf p_r^2)_{\restriction A}(x,y)=(\mathbf p_l^2)_{\restriction A}(y,x)$. 
	\item For any relatively $A$-definable linear order $\triangleleft$ on $p(\Mon)$, $(p,\triangleleft)$ is also an so-pair over $A$. 
\end{enumerate}
\end{Fact}
\begin{proof}
(i) is easy. For (ii), note that the property being a left-eventual (right-eventual) subset of $p(\Mon)$ is invariant under automorphisms from $\Aut_A(\Mon)$, so both $\mathbf p_l$ and $\mathbf p_r$ are $A$-invariant. To see that they are the only two $A$-invariant extensions, suppose that $\mathfrak p\supset p$ is $A$-invariant. By $A$-invariance, $(a<x)\in\mathfrak p$ for all $a\models p$ or $(x<a)\in\mathfrak p$ for all $a\models p$. In the first case, it is easy to see that any $\phi(x)\in\mathfrak p$ has an arbitrarily large realization in $p(\Mon)$, so $\phi(\Mon)\cap p(\Mon)$ is a right-eventual subset of $(p(\Mon),<)$, that is, $\phi(x)\in \mathbf p_r$; $\mathfrak p=\mathbf p_r$ follows. Similarly, $(x<a)\in\mathfrak p$ implies $\mathfrak p=\mathbf p_l$. 

(iii) Obviously, the locus $(\mathbf{p}_{r})_{\restriction B}(\Mon)$ is right-eventual in $(p(\Mon),<)$, so $\{x \in p(\Mon) \mid a < x\} \subseteq (\mathbf{p}_{r})_{\restriction B}(\Mon)$ holds for some (any) $a \models (\mathbf{p}_{r})_{\restriction B}$; this implies that $(\mathbf{p}_{r})_{\restriction B}(\Mon)$ is a final part $(p(\Mon),<)$. Likewise, $(\mathbf{p}_{l})_{\restriction B}(\Mon)$ is an initial part of $(p(\Mon),<)$.

(iv) Suppose $b\models (\mathbf p_{r})_{\restriction Aa}$ and let $c\models p$ be such that $b\models(\mathbf p_{l})_{\restriction Ac}$; in particular, we have $a<b<c$. By (iii), $c>b$ and $b\models(\mathbf p_{r})_{\restriction Aa}$ imply $c\models (\mathbf p_{r})_{\restriction Aa}$, while $a<b$ and $b\models(\mathbf p_{l})_{\restriction Ac}$ imply $a\models(\mathbf p_{l})_{\restriction Ac}$. Thus, $ab\equiv ac\equiv bc\ (A)$. Now $b\models(\mathbf p_{l})_{\restriction Ac}$ implies $a\models(\mathbf p_{l})_{\restriction Ab}$ as $\mathbf p_l$ is $A$-invariant.
The other implication is similar.

(v) Let $x\triangleleft y$ be an $L_A$-formula that defines $\triangleleft$. Clearly, for $a\models p$, either $x\triangleleft a$ or $a\triangleleft x$ belongs to $\mathbf p_l$. Without loss of generality, suppose that the former is the case. By the $A$-invariance of $\mathbf p_l$, $(x\triangleleft a)\in\mathbf p_l$ for all $a\models p$. By (iv) and the $A$-invariance of $\mathbf p_r$, $(a\triangleleft x)\in\mathbf p_r$ for all $a\models p$. It suffices to prove that $(\mathbf p_{l})_{\restriction B}(\Mon)$ is an initial part of $(p(\Mon),\triangleleft)$ for all $B\supseteq A$; 
indeed, this implies that every formula from $\mathbf p_l$ is left-eventual in $(p(\Mon),\triangleleft)$, and an analogous argument shows that every formula from $\mathbf p_r$ is right-eventual in $(p(\Mon),\triangleleft)$, so we conclude that $(p,\triangleleft)$ is an so-pair over $A$.

So, let $B\supseteq A$, $a\models (\mathbf p_{l})_{\restriction B}$ and $b\models p$, $b\triangleleft a$. Since $(a\triangleleft x)\in\mathbf p_r$,  $b\not\models(\mathbf p_{r})_{\restriction Aa}$ follows. By saturation find $c\models (\mathbf p_{l})_{\restriction B}$ such that $a\models(\mathbf p_{l})_{\restriction Bc}$; in particular, $a\models(\mathbf p_{l})_{\restriction Ac}$, so $c\models(\mathbf p_{r})_{\restriction Aa}$ by (iv). Since $b\not\models(\mathbf p_{r})_{\restriction Aa}$, $c\models(\mathbf p_{r})_{\restriction Aa}$ and $(\mathbf p_{r})_{\restriction Aa}(\Mon)$ is a final part of $(p(\Mon),<)$ by (iii), we conclude $b<c$. Since $c\models(\mathbf p_{l})_{\restriction B}$ and $(\mathbf p_{l})_{\restriction B}(\Mon)$ is an initial part of $(p(\Mon),<)$ by (iii), $b<c$ implies $b\models(\mathbf p_{l})_{\restriction B}$, and we are done.
\end{proof}

Part (iii) of the previous fact suggests that the type $\mathbf (\mathbf p_ r)_{\restriction B}$ is the right $\mathbf p$-generic extension and $(\mathbf p_l)_{\restriction B}$ the left $\mathbf p$-generic extension of $p$. 

\begin{Definition}\label{Definition same orientation}
Let $p\in S(A)$ be an so-type and let $\triangleleft$ and $<$ be two relatively $A$-definable linear orders on $p(\Mon)$. We say that the order $\triangleleft$ has the same orientation as $<$, denoted by $\triangleleft\approx_p<$, if for some (all) $a\models p$ the formula $a\triangleleft x$ relatively defines a right-eventual part of $(p(\Mon),<)$. 
\end{Definition}

\begin{Lemma}\label{Lemma_same_orient}
Suppose that $p\in S(A)$ is a so-type and let $O_p$ be the set of all relatively $A$-definable linear orders on $p(\Mon)$. The relation $\approx_p$ is an equivalence relation on $O_p$ with exactly two classes, such that each class consists of the reverses of orders from the other class. 
\end{Lemma}
\begin{proof} Suppose that  $\triangleleft,<\in O_p$.  According to Fact \ref{Fact_so_basic}(v), both $\mathfrak p=(p,\triangleleft)$ and $\mathbf p=(p,<)$ are so-pairs over $A$. From Fact \ref{Fact_so_basic}(ii), we get $\{\mathfrak p_l,\mathfrak p_r\}=\{\mathbf p_r,\mathbf p_l\}$; meaning either $\mathfrak p_r=\mathbf p_r \land \mathfrak p_l=\mathbf p_l$ or $\mathfrak p_r=\mathbf p_l \land \mathfrak p_l=\mathbf p_r$ holds. It is easy to see that the first case corresponds to $\triangleleft \approx_p <$, while the second corresponds to $\mathfrak p \approx_p \mathbf p^*$; from this, the desired conclusions follow.
\end{proof}

\section{Weakly o-minimal types}\label{Section2}

The notion of a weakly o-minimal type, as defined below, was first observed by Belegradek, Peterzil and Wagner in \cite[p.1130]{Belegradek}, where they remark that every complete 1-type in a quasi-o-minimal theory is weakly o-minimal. In this section, we introduce weakly o-minimal orders and types and prove a few basic facts. For example,  
we prove that relatively definable equivalence relations on the locus of a weakly o-minimal type are convex and pairwise $\subseteq$-comparable. We also show that the weak o-minimality of a type (over $A$) is preserved under relatively $A$-definable mappings.

\begin{Definition}
Let $P$ be a type-definable set and $<$ a relatively definable linear order on $P$.  We will say that the order $(P,<)$ is {\em weakly o-minimal} if every relatively definable subset of $P$ has finitely many convex components.
 \end{Definition}

\begin{Definition}
A complete type $p(x)\in S(A)$ is {\em weakly o-minimal} if there exists a relatively $A$-definable linear order $<$ such that $(p(\Mon),<)$ is a weakly o-minimal order. In that case, we say that $(p,<)$  is a {\em weakly o-minimal pair} over $A$. 
\end{Definition}

\begin{Remark}\phantomsection\label{Remark_weak_ominimal_firstrmk}
\begin{enumerate}[(a)]
\item  If $(P,<)$ is a weakly o-minimal order and $Q\subseteq P$ is a type-definable subset, then the suborder $(Q,<)$ is also weakly o-minimal.  

\item Every complete type that extends a weakly o-minimal type is also weakly o-minimal. In fact, if $(p,<)$ is a weakly o-minimal pair over $A$, $B\supseteq A$ and the type $q\in S(B)$ extends $p$, then the pair $(q,<)$ is weakly o-minimal over $B$ because $(q(\Mon),<)$ is a suborder of $(p(\Mon),<)$.
\item If the theory $T$ is weakly o-minimal with respect to $<$, then $(P,<)$ is a weakly o-minimal order for every type-definable set $P\subset \Mon$; in particular, the pair $(p,<)$ is weakly o-minimal for every complete 1-type $p$. In fact, the latter holds even if $T$ is weakly quasi-o-minimal with respect to $<$.    
\item The weak o-minimality of a type is preserved by passing from $T$ to $T^{eq}$. More precisely, if $(p,<)$ is a weakly o-minimal pair over $A\subseteq \Mon$, then $p$, viewed as a $T^{eq}$-type of a real sort, is also weakly o-minimal (as witnessed by $<$).

\item If $(p(\Mon),<)$ is a weakly o-minimal order, then so is $(p(\Mon),<_{\vec E})$ for any sequence $\vec E$ of pairwise $\subseteq$-comparable, relatively definable, $<$-convex equivalence relations on $p(\Mon)$; this is a consequence of Remark \ref{Remark_order_<_vecE}(e). 
Similarly, if $(p,<)$ is a weakly o-minimal pair, then so is $(p,<_{\vec E})$ for any sequence $\vec E$ of pairwise $\subseteq$-comparable, relatively $A$-definable $<$-convex equivalence relations. 
\item It is easy to see that every weakly o-minimal type is an so-type. In fact, every weakly o-minimal pair over $A$, say $\mathbf p=(p,<)$, is an so-pair over $A$; therefore, $\mathbf p_r$ and $\mathbf p_l$, the right and left globalizations of $\mathbf p$, are well defined.
\item The main advantage of weakly o-minimal types compared to so-types is that weak o-minimality transfers to complete extensions. 
\end{enumerate}
\end{Remark}

In Corollary \ref{Corollary_wom_is_ind_of_order}, we will prove that a weakly o-minimal type $p\in S(A)$ is weakly o-minimal with respect to {\em any} relatively $A$-definable order on $p(\Mon)$; that is, the order $(p(\Mon),\triangleleft)$ is weakly o-minimal for any relatively $A$-definable order $\triangleleft$ on $p(\Mon)$. In the next example, we show
that this may fail if the relative definition of the order $\triangleleft$ requires parameters outside $A$. 
 
 \begin{Example}
Consider the structure $\mathcal M=(\mathbb R,<,S)$ where $S(x)=x+1$. $\mathcal M$ is o-minimal as a reduct of the ordered group of reals, and the theory $T=\Th(\mathcal{M})$ eliminates quantifiers. Since any translation $x\mapsto x+r$ is an automorphism of $\mathcal M$, there is a unique complete type $p\in S_1(\emptyset)$;  $(p,<)$ is a weakly o-minimal pair over $\emptyset$. The only $\emptyset$-definable linear orders on $\mathbb R$ are: $<$ and its reverse $>$; this follows by elimination of quantifiers. Hence $p$ is weakly o-minimal with respect to all $\emptyset$-definable orders.
Let $\triangleleft$ be defined by:

-- \ for $x\in [0,1)$ and $y\in [1,2)$: \   
 $x\triangleleft y$ \ iff \ $x+1\leqslant y$, \ and \ $y\triangleleft x$ \ iff $x+1>y$;
 
-- \  for all other pairs $(x,y)$ define $x\triangleleft y$ iff $x<y$. 
 
\noindent 
It is not hard to see that $\triangleleft$ is  a $\{0\}$-definable linear order on $\mathbb R$. The order $(\mathbb R,\triangleleft)$ is not weakly o-minimal, since the formula $x<\frac{1}{2}$ alternates on the sequence $$\dots\triangleleft \frac{1}{5}\triangleleft 1+\frac{1}{5}\triangleleft \frac{1}{4}\triangleleft 1+\frac{1}{4}\triangleleft \frac{1}{3}\triangleleft 1+\frac{1}{3}.$$
Therefore, the type $p\in S_1(\emptyset)$ is weakly o-minimal and $(p(\Mon),\prec)$ is a weakly o-minimal order for all $\emptyset$-definable orders on $p(\Mon)$, but the order $(p(\Mon),\triangleleft)$ is not weakly o-minimal. 
\end{Example}

A function $f:X\to Y$ is said to be $(<_X,<_Y)$-increasing, where $<_X$ is a linear order on $X$  and $<_Y$ a linear order on $Y$, if $x<_Xx'$ implies $f(x)\leqslant_Y f(x')$ for all $x,x'\in X$. 

\begin{Lemma}\phantomsection \label{Lemma_wom_orders_hoomomorphism_interdefinability}
\begin{enumerate}[(i)]
\item Suppose that $(P,<)$ is a weakly o-minimal order, $D$ is a definable set, and $f:P\to D$ is a relatively definable function with a convex kernel. 
For $y,y'\in f(P)$ define: $y<_f y'$ iff $f^{-1}(\{y\})<f^{-1}(\{y'\})$. 
Then $(f(P),<_f)$ is a weakly o-minimal order and $f:P\to f(P)$ is $(<,<_f)$-increasing.
\item The weak o-minimality of types is preserved under interdefinability, that is, if $p,q\in S(A)$ are interdefinable, then $p$ is weakly o-minimal iff $q$ is such. 
\end{enumerate} 
\end{Lemma}
\begin{proof}
(i) Let $P$ be type-defined by $p(x)$ and let $D=\psi(\Mon)$. Suppose that $x<y$ and $f(x,y)$ relatively define $<$ and $f$ respectively. 
The kernel relation, $\Ker f$, is relatively defined by $(\exists y)(\psi(y)\land f(x,y)\land f(x',y))$ on $P$. We have the following:
 
-- \ $x<x'$ defines a linear order on $p(\Mon)$ and 
$f(x,y)$ defines a function $p(\Mon)\to\psi(\Mon)$;

-- \ $(\exists y)(\psi(y)\land f(x,y)\land f(x',y))$ defines a convex equivalence relation on $(p(\Mon),<)$).
 
\noindent 
This is expressible by a $\tp$-universal sentence, 
so by Fact \ref{Fact_L_P_sentence} there exist a definable extension $(\theta(\Mon),<)$ of $(p(\Mon),<)$ and a definable extension $\hat f:\theta(\Mon)\to \psi(\Mon)$ of $f$ such that $\Ker\hat f$ is a $<$-convex equivalence relation on $\theta(\Mon)$.
By the convexity of $\Ker \hat f$, it is easy to see that $\hat f^{-1}(\{y\})<\hat f^{-1}(\{y'\})$ defines a linear order on $\hat f(\theta(\Mon))$ that agrees with $<_f$ on $f(P)$. By Fact \ref{Fact rel def kernel and inverse image}, the set $f(P)$ is type-definable, so $<_f$ is relatively definable on $f(P)$. It is easy to see that the function $f:P\to f(P)$ is $(<,<_f)$-increasing.

Now we prove that $(f(P),<_f)$ is a weakly o-minimal order. Let $D'$ be a relatively definable subset of $f(P)$; we need to show that $D'$ has finitely many convex components in $(f(P),<_f)$. 
By Fact \ref{Fact rel def kernel and inverse image}, the inverse image $f^{-1}(D')$ is a relatively definable subset of $P$; the weak o-minimality of $(P,<)$ implies that $f^{-1}(D')$ has finitely many $<$-convex components, say $n$. Since $f:P\to f(P)$ is $(<,<_f)$-increasing and surjective, it follows that the set $D'=f(f^{-1}(D'))$ has exactly $n$ convex components in $(f(P),<_f)$. Thus, $(f(P),<_f)$ is weakly o-minimal.

(ii) follows easily from (i).
\end{proof}

\begin{Lemma}\label{Lemma_reldef_of_convex_components}
Let $(p,<)$ be a weakly o-minimal pair over $A$ and let $B\supseteq A$. 
\begin{enumerate}[(i)]
\item  If  $D$ is a relatively $B$-definable subset of $p(\Mon)$, then  every  convex component of $D$, as well as each of the sets  $\{x\in p(\Mon)\mid x<D\}$ and $\{x\in p(\Mon)\mid D<x\}$, is relatively $B$-definable.
\item  If $q\in S(B)$ is an extension of $p$, then $q(x)$ is determined by the subtype of all $L_B$-formulae that relatively define a convex subset. In particular, $q(\Mon)$ is a convex subset of $p(\Mon)$ and $(q,<)$ a weakly o-minimal pair over $B$. 
\end{enumerate}
\end{Lemma}
\begin{proof}
(i) Let $D=\phi(p(\Mon))$. Denote by $\mathcal C_0$ the finite convex partition of $p(\Mon)$ consisting of all the convex components of the sets  $D$ and  $D^c=\lnot\phi(p(\Mon))$. Let $n=|\mathcal C_0|$. Then all the sets mentioned in the conclusion of the lemma are members of $\mathcal C_0$; we will prove that every member of $\mathcal C_0$ is relatively $B$-definable. Note that ``$x<y$ defines a linear order on $p(\Mon)$'' and ``$\phi(x)$ determines a convex partition of $p(\Mon)$ with at most $n$ elements'' are $\tp$-universal properties; the latter is expressed by:
$$ \models (\forall x_0,\ldots, x_n\models p)\left(x_0<x_1<\ldots< x_n\rightarrow\bigvee_{i<n}(\phi(x_i)\leftrightarrow \phi(x_{i+1}))\right).$$
By Fact \ref{Fact_L_P_sentence} there is a formula $\theta(x)\in p$ such that $x<y$ defines a linear order on $\theta(\Mon)$ and $\phi(x)$ induces a convex partition $\mathcal C_1$ of $(\theta(\Mon),<)$ with $\leqslant n$ members. Then each component from $\mathcal C_0$ is the intersection of the corresponding component from $\mathcal C_1$ with $p(\Mon)$. Since the members of $\mathcal C_1$ are $B$-definable, the desired conclusion follows.

(ii) Fix $q\in S(B)$ that extends $p$.  
For each  $\phi(x)\in q$ let  $D_{\phi}=\phi(\Mon)\cap p(\Mon)$, then $q(\Mon)=\bigcap_{\phi(x)\in q}D_{\phi}$. For each $\phi\in q$ the set $D_{\phi}$ has finitely many convex components on $p(\Mon)$ and, by part (i), each of them is relatively $B$-definable. Notice that the $L_B$-formulae that relatively define the components can be chosen pairwise inconsistent, in which case exactly one of them, say $\theta_{\phi}(x)$,  belongs to $q(x)$; denote by $C_{\phi}$ the component relatively defined by $\theta_{\phi}(x)$. Then $q(\Mon)=\bigcap_{\phi(x)\in q}C_{\phi}$ and $\{\theta_{\phi}(x)\mid \phi\in q\}\vdash q(x)$. Finally, since each $C_{\phi}$ is a convex subset of $p(\Mon)$, so is the intersection $\bigcap_{\phi(x)\in q}C_{\phi}=q(\Mon)$. 
\end{proof}

\begin{Corollary}\label{Cor_extesnions_ofwom_totalyordered}
Let $\mathbf p= (p,<)$ be a weakly o-minimal pair over $A$  and let $B\supseteq A$. 
Then the set $S_p(B)=\{q\in S_n(B)\mid p\subseteq q\}$ is linearly ordered by $<$;  $(\mathbf p_{l})_{\restriction B}=\min S_p(B)$ and $(\mathbf p_{r})_{\restriction B}=\max S_p(B)$.
\end{Corollary}
\begin{proof}
By Lemma \ref{Lemma_reldef_of_convex_components}(ii),  $\{q(\Mon)\mid q\in S_p(B)\}$ is a convex partition of $(p(\Mon),<)$, so it is naturally linearly ordered by $<$. The second assertion follows from Fact \ref{Fact_so_basic}(iii).
\end{proof}

In Lemma \ref{Lemma_reldef_of_convex_components}(ii), we proved that every complete type extending a weakly o-minimal type $p$ has a convex locus in $(p(\Mon),<)$ (for any $<$ witnessing the weak o-minimality of $p$). Now we show that this property characterizes weakly o-minimal types.

\begin{Lemma}\label{Lemma all extension convex imply wom}
Let $p\in S(A)$ and $<$ be a relatively $A$-definable linear order on $p(\Mon)$. If for all $B\supseteq A$ and all $q\in S(B)$ that extend $p$ the locus $q(\Mon)$ is convex in $(p(\Mon),<)$, then $(p,<)$ is a weakly o-minimal pair.
\end{Lemma} 
\begin{proof}
Denote $\lambda=2^{\aleph_0+|T|+|A|}$.
By way of contradiction, suppose that $\phi(x,b)$ relatively defines a subset of $p(\Mon)$ that has infinitely many convex components in $(p(\Mon),<)$. Let $y=(x_\alpha)_{\alpha<\lambda^+}$. By compactness, the set $\pi(y)=\{x_\alpha<x_\beta\mid \alpha<\beta<\lambda^+\}\cup\{\lnot (\phi(x_\alpha,b)\leftrightarrow\phi(x_{\alpha+1},b))\mid \alpha<\lambda^+\}$ is satisfiable; let $a=(a_\alpha)_{\alpha<\lambda^+}$ realize $\pi(y)$. Since there are at most $\lambda$ extensions of $p$ over $Ab$, there are $\alpha<\beta<\lambda^+$ such that $a_\alpha\equiv a_\beta\ (Ab)$. Note that $a_\alpha\nequiv a_{\alpha+1}\ (Ab)$, so $\beta>\alpha+1$. Since also $a_\alpha<a_{\alpha+1}<a_\beta$, this contradicts the assumption that the locus of $\tp(a_\alpha/Ab)$ is convex within $(p(\Mon),<)$.
\end{proof}

\begin{Definition}
For a complete type $p\in S(A)$, define $\mathcal E_p$ as the set of all relatively $A$-definable equivalence relations on $p(\Mon)$; $\mathbf 1_p\in\mathcal E_p$ is the complete relation $p(\Mon)^2$.    
\end{Definition}   

\begin{Proposition}\label{Prop_E_p_is convex_linear}
Let $(p,<)$ be a weakly o-minimal pair over $A$. Then:
\begin{enumerate}[(i)]
\item  Every relation from $\mathcal E_p$ is $<$-convex;
\item  $(\mathcal E_p,\subseteq)$ is a linear order. 
\end{enumerate}
\end{Proposition}
\begin{proof}
(i) Let $E\in\mathcal E_p$. By way of contradiction, suppose that $E$ is not convex and choose $a_1<b_1<a_2$ realizing $p$ such that $\models E(a_1,a_2)\land \lnot E(a_1,b_1)$. Let $f\in \Aut_A(\Mon)$ map $a_1$ to $a_2$. Define $f(a_n)=a_{n+1}$ and $f(b_n)=b_{n+1}$ for $n\geq 1$. Then $\models E(a_n,a_{n+1})\land \lnot E(a_n,b_n)$ holds for all $n\geq 1$. Therefore,  members of the sequence $a_1<b_1<a_2<b_2<a_3<\dots$ alternately satisfy the formula  $E(a_1,x)$; that contradicts weak o-minimality of $(p,<)$.

(ii) Let $E_1,E_2\in\mathcal E_p$ and $a\models p$. It suffices to prove $[a]_{E_1}\subseteq [a]_{E_2}$ or $[a]_{E_2}\subseteq[a]_{E_1}$. Suppose, for the sake of contradiction, that $b\in [a]_{E_1}\smallsetminus [a]_{E_2}$ and $c\in [a]_{E_2}\smallsetminus [a]_{E_1}$. By (a), both $[a]_{E_1}$ and $[a]_{E_2}$ are convex, so either $b<[a]_{E_2}$ and $[a]_{E_1}<c$, or $c<[a]_{E_1}$ and $[a]_{E_2}<b$. Without loss, suppose that the former holds. Take $f\in\Aut_A(\Mon)$ such that $f(b)=a$, and let $f(a)=a'$; clearly, $a<a'$ as $b<a$, and $a'\in[a]_{E_1}$ as $\models E_1(b,a)$. Thus, $a'<c$ as $[a]_{E_1}<c$. Since $a<a'<c$, $a,c\in[a]_{E_2}$, and since $[a]_{E_2}$ is convex, we obtain $\models E_2(a,a')$. Thus, $\models E_2(f^{-1}(a),f^{-1}(a'))$ holds, that is, $\models E_2(b,a)$; a contradiction.
\end{proof}

\begin{Remark}
Assume that $p\in S(T)$ is a weakly o-minimal type and that $\mathcal E_p$ is finite; note that this holds for all $p\in S(T)$ if $T$ is $\aleph_0$-categorical. Choose an embedding $\delta$ of $(\mathcal E_p,\subset)$ into the ordered unit interval such that $\delta(id_p)=0$ and $\delta(\mathbf 1_p)=1$. 
Then $d(x,y)=\min\{\delta(E)\mid E\in \mathcal E_p\mbox{ and } \models E(x,y)\}$ defines an ultrametric on $p(\Mon)$.
The importance of this ultrametric space was first observed by
Herwig et al. in \cite{HM} for types in a $\aleph_0$-categorical, weakly o-minimal theory. In that case, $d(x,y)=r$ is $A$-definable for all $r\in[0,1]$, and they have shown that every complete 2-type over $A$ extending $p(x)\cup p(y)\cup \{x<y\}$ is isolated by formula $d(x,y)=r$. In other words, the binary structure on $p(\Mon)$ is completely determined by the ultrametric and the order $<$.

Outside the $\aleph_0$-categorical context, assuming that $T$ is countable and $p\in S(T)$ is weakly o-minimal, it is still possible to properly define an embedding $\delta$ of $(\mathcal E_p,\subset)$ into the unit interval such that $d(x,y)$ defined as above is an ultrametric on $p(\Mon)$;
in our forthcoming paper, we will prove that the binary structure on $p(\Mon)$ is completely determined by this ultrametric and $<$ if $T$ is weakly quasi-o-minimal theory with few countable models and $p\in S_1(T)$.  
\end{Remark}

\begin{Corollary}\label{Cor_reldef_functions_have_convex_kernel}
If $(p,<)$ is a weakly o-minimal pair over $A$, then $\Ker f\in \mathcal E_p$ is a convex equivalence relation for any relatively $A$-definable function $f$ from $p(\Mon)$ into an $A$-definable set.  
\end{Corollary}
\begin{proof}
By Fact \ref{Fact rel def kernel and inverse image} the kernel $\Ker f$ is relatively $A$-definable, so it is convex by Proposition \ref{Prop_E_p_is convex_linear}(i).    
\end{proof}
 
\begin{Proposition}\label{Prop_pwom_implies_dclq_wom}
Let $(p,<)$ be a weakly o-minimal pair over $A$ and let $q\in S(A)$. Suppose that $f:p(\Mon)\to q(\Mon)$ is a relatively $A$-definable function. Then: 
\begin{enumerate}[(i)]
\item $(q,<_f)$ is a weakly o-minimal pair over $A$, where $<_f$ is defined by \ $y<_fy'$ \ iff \ $f^{-1}(\{y\})<f^{-1}(\{y'\})$;
in particular, $q$ is a weakly o-minimal type.
\item $f:(\mathcal E_p,\subseteq)\to (\mathcal E_q,\subseteq)$ is an order-epimorphism whose restriction $\{E\in\mathcal E_p\mid \Ker f\subseteq E\}\to\mathcal E_q$ is an order-isomorphism; 
\item For all $E\in\mathcal E_p$, $[x]_E\mapsto [f(x)]_{f(E)}$  defines a function $f_E:p(\Mon)/E\to q(\Mon)/f(E)$.
\end{enumerate}
\end{Proposition}
\begin{proof}
(i) By Corollary \ref{Cor_reldef_functions_have_convex_kernel}, the kernel relation, $\Ker f$, is convex on $(p(\Mon),<)$. Hence, the weakly o-minimal order $(p(\Mon),<)$ and the function $f$ satisfy the assumptions of Lemma \ref{Lemma_wom_orders_hoomomorphism_interdefinability}(i); we conclude that $(f(p(\Mon)),<_f)$ is a weakly o-minimal order. As $f(p(\Mon))=q(\Mon)$, the pair $(q,<_f)$ is weakly o-minimal over $A$. 

(ii) Here, the main task is to prove that $f$ determines a function from $\mathcal E_p$ to $\mathcal E_q$. 
Fix $E\in\mathcal E_p$ and we will prove $f(E)\in \mathcal E_q$. As $\Ker f\in\mathcal E_p$, by Proposition \ref{Prop_E_p_is convex_linear}(ii), we have two possibilities: $E\subseteq \Ker f$ and $\Ker f\subseteq E$. 
Clearly, $E\subseteq \Ker f$ implies $f(E)= \id_{q(\Mon)}\in \mathcal E_q$ and we are done, so from now on assume $\Ker f\subseteq E$. 
Suppose that $f$ is relatively defined by the $L_A$-formula $f(x,y)$.
By Fact \ref{Fact_L_P_sentence} there is an $A$-definable set $D\supseteq q(\Mon)$, defined by $\psi(y)$ say, such that $f(x,y)$ relatively defines a function $p(\Mon)\to D$. Then the formula $(\exists y)(\psi(y)\land f(x,y)\land f(x',y))$ relatively defines the kernel $\Ker f$, so we have the following properties:
\begin{enumerate}[\hspace{10pt} (1)]
\item   $(p(\Mon),<,E)$ is a linear order with a convex equivalence relation; 

\item $f(x,y)$ relatively defines a function $p(\Mon)\to D$;

\item $(\exists y)(\psi(y)\land f(x,y)\land f(x',y))$ relatively defines a convex equivalence on $(p(\Mon),<)$;

\item  $\models(\forall x,x'\models p)( (\exists y)(\psi(y)\land f(x,y)\land f(x',y))\rightarrow E(x,x'))$. 
\end{enumerate}
Here, (3) says that $\Ker f$ is convex on $(p(\Mon),<)$ and (4) says $\Ker f\subseteq E$. 
Clearly, (1)--(4) are tp-universal properties, so by Fact \ref{Fact_L_P_sentence} there exists an $A$-definable set $D_p\supseteq p(\Mon)$  such that letting $\hat E$ and $\hat f$ be as usual,
we have: \ $(D_p,<,\hat E)$ is a linear order with a convex equivalence relation; \ 
 $\hat f: D_p\to D$; \ 
 $\Ker \hat f$ is a convex equivalence on $(D_p,<)$ and $\Ker\hat f\subseteq \hat E$. Furthermore, by replacing $D$ with $\hat f(D_p)$ the kernel relation does not change, so we can also assume that $\hat f$ is surjective.  
Note that $\Ker\hat f\subseteq\hat E$ implies that $\hat f(\hat E)$ is an equivalence relation on $D_q$; clearly, $\hat f(\hat E)$ is $A$-definable. Then $\hat f(\hat E)_{\restriction q(\Mon)}=f(E)$ is a relatively $A$-definable equivalence relation on $q(\Mon)$, so $f(E)\in\mathcal E_q$. 
Therefore, $f$ maps $\mathcal E_p$ to $\mathcal E_q$.  

To show that $f$ is surjective, let $F\in\mathcal E_q$ be relatively defined by $F(y,y')$.
Then $E'=\{(x,x')\in p(\Mon)^2\mid \models F(f(x),f(x'))\}$ is an equivalence relation on $p(\Mon)$ that is relatively defined by $(\exists y,y'\in D)(f(x,y)\land f(x',y')\land F(y,y'))$, so $E'\in \mathcal E_p$. 
As $f(E')=F$ is clearly true, the function $f:\mathcal E_p\to \mathcal E_q$ is surjective; it follows that $f$ is an order-epimorphism. 

Finally, note that $\Ker f\subseteq E_1\subset E_2$ implies $f(E_1)\subset f(E_2)$, so $f: \{E\in\mathcal E_p\mid \Ker f\subseteq E\}\to\mathcal E_q$ is an order-isomorphism. 
This completes the proof of (ii). 
(iii) follows from (ii)
\end{proof}

\begin{Remark}
A consequence of the previous proposition is that, roughly speaking, the quotient $(p/E,<)$ of a weakly o-minimal pair $(p,<)$ over $A$ by a relatively $A$-definable (convex) equivalence relation $E$ is also a weakly o-minimal pair. Formally, $p/E$ is not a $\Mon^{eq}$-type, so instead of $p/E$ we will work with $p/\hat E\in S^{eq}(A)$, where $\hat E$ is a definable (convex) extension of $E$.  
The pair $(p/\hat E,<)$ is weakly o-minimal by the previous proposition, and $p/E$ is interdefinable with $p/\hat E$ in the sense that for each $a\models p$, the hyperimaginary $[a]_E$ and the imaginary $[a]_{\hat E}\in \Mon^{eq}$ are interdefinable (meaning $\Aut_{A[a]_E}(\Mon)=\Aut_{A[a]_{\hat E}}(\Mon)$). 
\end{Remark}

The next lemma is technical. There, we state basic properties of the $(<_{\vec E},\triangleleft)$-increasing functions that will be used in the proof of the monotonicity theorems. 

\begin{Lemma}\label{Cor_f_of_convex_is_convex}
Suppose that $(p,<)$ and $(q, \triangleleft)$ are weakly o-minimal  pairs over $A$,  $\vec E=(E_1,\ldots, E_n)\in \mathcal (\mathcal E_p\smallsetminus\{\mathbf 1_p\})^{n}$ an $\subseteq$-increasing sequence, and $f: p(\Mon) \to  q(\Mon)$ a non-constant, relatively $A$-definable function such that $\Ker(f)\subseteq E_1$.   Then:
\begin{enumerate}[(i)]
    \item  $f$ is $(<_{\vec E},\triangleleft)$-increasing if and only if it is $(<,\triangleleft_{f(\vec E)})$-increasing;
    \item  If $f$ is $(<_{\vec E},\triangleleft)$-increasing, then:
    \begin{enumerate}[(a)]
        \item $f_{E_n}: p(\Mon)/E_n \to  q(\Mon)/f(E_n)$ \ is strictly  $(<,\triangleleft)$-increasing;
        \item $f_{E_k}:p(\Mon)/E_k\to q(\Mon)/f(E_k)$ \ is strictly  $(<_{(E_{k+1},\ldots,E_n)},\triangleleft)$-increasing for all $k\leqslant n$.
    \end{enumerate}
 \end{enumerate}
\end{Lemma}
\begin{proof} Since $\Ker f\subseteq E_1$ holds, we can apply Proposition \ref{Prop_pwom_implies_dclq_wom}(ii) and conclude that $(f(E_1),\dots,f(E_n))$ is an increasing sequence of convex equivalences on $q(\Mon)$. In particular:
\begin{equation}\tag{$*$}
(a,b)\in E_k\ \Leftrightarrow \ (f(a),f(b))\in f(E_k)\ \mbox{ holds for all }k\leqslant n.
\end{equation}

(i) We will prove the case $n=1$; the general case follows from this one and Remark \ref{Remark_order_<_vecE}(b) by easy induction. To prove $(\Rightarrow)$, suppose that $f$ is $(<_{E_1} \triangleleft)$-increasing, and let $a<b$ be realizations of $p$. We consider two cases.
If $(a,b)\in E_1$, then $b<_{E_1}a$, so $f(b)\trianglelefteq f(a)$ as $f$ is $(<_{E_1},\triangleleft)$-increasing. Since also $(f(a),f(b))\in f(E_1)$ by $(*)$, we conclude $f(a)\trianglelefteq_{f(E_1)} f(b)$. 
On the other hand, if $(a,b)\notin E_1$, then $a<_{E_1}b$, so $f(a)\trianglelefteq f(b)$ as $f$ is $(<_{E_1},\triangleleft)$-increasing. Since also $(f(a),f(b))\notin f(E_1)$ by $(*)$, we conclude $f(a)\trianglelefteq_{f(E_1)} f(b)$.    
Therefore,  for all $a,b\models p$, $a<b$  implies  $f(a)\trianglelefteq_{f(E_1)} f(b)$,  so $f$ is $(<,\triangleleft_{f(E_1)})$-increasing. This completes the proof of $(\Rightarrow)$.
To prove $(\Leftarrow)$, note that $f$ being $(<,\triangleleft_{f(E_1)})$-increasing is the same as being $((<_{E_1})_{E_1},\triangleleft_{f(E_1)})$-increasing as $(<_{E_1})_{E_1}=<$. So, by $(\Rightarrow)$, $f$ is $(<_{E_1},(\triangleleft_{f(E_1)})_{f(E_1)})$-increasing, that is, it is $(<_{E_1},\triangleleft)$-increasing as $(\triangleleft_{f(E_1)})_{f(E_1)}=\triangleleft$. 

(ii) Suppose that $f$ is $(<_{\vec E},\triangleleft)$-increasing. To prove part (a), assume that $a,b\models p$ and $[a]_{E_n}<[b]_{E_n}$.  Since $E_n$ is maximal in $\vec E$, the orders $<$ and $<_{\vec E}$ agree on the $E_n$-classes, so $f([a]_{E_n})\trianglelefteq f([b]_{E_n})$ follows as $f$ is $(<_{\vec E},\triangleleft)$-increasing. 
Furthermore, by $(*)$,  $[a]_{E_n}<[b]_{E_n}$ implies  $(f(a),f(b))\notin f(E_n)$, so the classes $[f(a)]_{f(E_n)}$ and $[f(b)]_{f(E_n)}$ are distinct; $[f(a)]_{f(E_n)}\triangleleft [f(b)]_{f(E_n)}$ follows. 
Therefore, $f_{E_n}$ is strictly $(<,\triangleleft)$-increasing, which proves part (a). 
To prove (b), consider $(p(\Mon),<_{(E_{k+1},\ldots,E_n)})$ instead of $(p(\Mon),<)$. Then $f$ is $((<_{(E_{k+1},\ldots,E_n)})_{(E_1,\ldots, E_k)},\triangleleft)$-increasing by Remark \ref{Remark_order_<_vecE}(b), so $f_{E_k}$ is strictly $(<_{(E_{k+1},\ldots,E_n)},\triangleleft)$-increasing by part (a).    
\end{proof}

\section{Relatively definable orders and monotonicity theorems}\label{Section3}

In this section, we characterize relatively definable linear orders on the locus of a weakly o-minimal type and, as corollaries, obtain monotonicity theorems.  The most technically demanding part is Theorem \ref{Theorem_characterize_wom_orders}, where we prove that any pair of relatively definable orders $\triangleleft$ and $<$ on the locus of a weakly o-minimal type $p$ satisfies $\triangleleft=<_{\vec E}$ for some increasing sequence of relations $\vec E\in(\mathcal E_p)^{<\omega}$. 
This result has already been proved in \cite[Proposition 5.2]{MTwmon} in the context of weakly quasi-o-minimal theories. Although the proof there {\em ad verbum} goes through in the context of weakly o-minimal types, we present a simpler argument here.

\begin{Lemma}\label{Lemma_wom_order_property}
Let $(p,<)$ be a weakly o-minimal pair over $A$ and $\phi(x,y)$ an $L_A$-formula. Let $C(a)=\phi(\Mon,a)$. Suppose that for some (any) $a\models p$, $C(a)$ is an initial part of $\{x\in p(\Mon)\mid a<x\}$. Then there do {\it not} exist $a_1<b_0<a_0$ realizing $p$ with $C(b_0)<C(a_0)$ and $C(b_0)\cup C(a_0)\subseteq C(a_1)$.  
\end{Lemma}
\begin{proof}
By way of contradiction, assume that $a_1<b_0<a_0$ are realizations of $p$ such that $C(b_0)<C(a_0)$ and $C(b_0)\cup C(a_0)\subseteq C(a_1)$; in particular, $a_1<C(b_0)<a_0$. Let $f\in \Aut_A(\Mon)$ map $a_0$ to $a_1$, and for $n\geqslant 0$ define $a_{n+1}=f(a_n)$ and $b_{n+1}=f(b_n)$; by induction we have  $a_{n+1}<b_n<a_n$,  $C(a_n)\subseteq C(a_{n+1})$ and $a_{n+1}<C(b_n)<a_n$.  
Since $a_0\in C(a_0)\subseteq C(a_1)\subseteq C(a_n)\subseteq\dots$, we conclude $\models \phi(a_0,a_n)$ as $\phi(x,a_n)$ relatively defines $C(a_n)$. 
Also, $a_{n+1}<C(b_n)<a_n$ implies $a_0>C(b_0)>C(b_1)>\dots$, so $a_0\notin C(b_n)$, that is, $\models\lnot \phi(a_0,b_n)$. Therefore, the members of the sequence $\dots< a_2<b_1<a_1<b_0<a_0$ alternately satisfy the formula $\phi(a_0,x)$, contradicting the fact that $(p,<)$ is a weakly o-minimal pair.  
\end{proof}

\begin{Theorem}\label{Theorem_characterize_wom_orders}
Suppose that $\mathbf p=(p,<)$ is a weakly o-minimal pair over $A$ and $\triangleleft$ a relatively $A$-definable linear order on $p(\Mon)$. Then there exists a unique strictly increasing sequence of equivalence relations $\vec E= (E_0,E_1,...E_{n})\in \mathcal E_p^{n+1}$ such that $E_0=\id_{p(\Mon)}$ and $\triangleleft=<_{\vec E}$.
Moreover, $<$ and $\triangleleft$ have the same orientation if and only if $E_n\neq \mathbf 1_p$.  
\end{Theorem}
\begin{proof}
For $a\models p$, the subsets of $(a,+\infty)_{\mathbf p}=\{x\in p(\Mon)\mid a<x\}$ relatively defined by the formulae $a\triangleleft x$ and $x\triangleleft a$ have finitely many convex components in $p(\Mon,<)$. These components are relatively $Aa$-definable by Lemma \ref{Lemma_reldef_of_convex_components}(i), form a $<$-convex partition $C_0(a)<C_1(a)<\dots<C_n(a)$ of $(a,+\infty)_{\mathbf p}$ and alternately satisfy $x\triangleleft a$ and $a\triangleleft x$; note that we obtain the same partition if we work with the reverse $\triangleleft^*$ instead of $\triangleleft$. We proceed by induction on $n$.

If $n=0$, then for any $a,b\models p$, $a<b$ implies $b\in C_0(a)$, so $a\triangleleft b$ or $b\triangleleft a$ is valid, depending on whether $C_0(a)$ is determined by $a\triangleleft x$ or $x\triangleleft a$. Thus, $<\subseteq \triangleleft$ or $<\subseteq\triangleleft^*$, that is, either $<=\triangleleft$ or $<=\triangleleft^*$ holds by linearity. So taking $\vec E=(\id_{p(\Mon)})$ or $\vec E=(\id_{p(\Mon)},\mathbf 1_p)$ completes the proof.

Assume that $n\geqslant 1$. For $k=1,\dots,n$ let $E_k$ be the equivalence relation in $p(\Mon)$ given by $\sup C_{k-1}(x)=\sup C_{k-1}(y)$ (the sets $C_{k-1}(x)$ and $C_{k-1}(y)$ have the same set of strict upper bounds in $(p(\Mon),<)$). We will prove that $E_k\in\mathcal E_p$ and that they are $\triangleleft$-convex. Moreover, it will turn out that $(\id_{p(\Mon)},E_1,\dots,E_n)$ or $(\id_{p(\Mon)},E_1,\dots,E_n,\mathbf 1_p)$ is the desired sequence.

For a while, we will assume that the elements of $C_n(a)$ satisfy $a\triangleleft x$, that is, that the orders $<$ and $\triangleleft$ have the same orientation, so the elements of $C_{n-1}(a)$ satisfy $x\triangleleft a$. 
Define $C(a)=\bigcup_{i<n}C_i(a)$; then $\inf C(a)=a$, and $x\in C(y)$ is a relatively $A$-definable relation on $p(\Mon)$. 
Also note that $\max C(a)$ does not exist, for if $a'=\max C(a)$ and $a''=\max C(a')$, then $a'\in C_{n-1}(a)$ and $a''\in C_n(a)\cap C_{n-1}(a')$, so $a'\triangleleft a$, $a\triangleleft a''$ and $a''\triangleleft a'$; this is impossible.
Finally, since $C_{n-1}(a)$ is a final part of $C(a)$, we have $E_n(x,y)$ iff $\sup C(x)=\sup C(y)$. 

\begin{Claim}
$b\in C_{n-1}(a)$ implies $C(b)\subseteq C(a)$.
\end{Claim}

\noindent{\em Proof.}  
Assume $b\in C_{n-1}(a)$; then $a<b$ and $b\triangleleft a$. The set $C(a)$ is convex and contains $b=\inf C(b)$ so, to prove $C(b)\subseteq C(a)$, it suffices to show that some final part of $C(b)$ is contained in $C(a)$; we will prove $C_{n-1}(b)\subseteq C(a)$.
For any $x\in C_{n-1}(b)$  we have $b<x$ and $x\triangleleft b$. Combining with $a<b$ and $b\triangleleft a$ we derive $a<x$ and $x\triangleleft a$, so $x\in (a,+\infty)_{\mathbf p}$ and $x\notin C_n(a)$, and thus $x\in C(a)$. This proves $C_{n-1}(b)\subseteq C(a)$  and completes the proof of the claim. \hfill$\blacksquare$

\begin{Claim}
If $C(a)\cap C(a_1)\neq 0$, then $E_n(a,a_1)$ holds.
\end{Claim}

\noindent {\it Proof.} 
Suppose not. Without loss, let $\sup C(a)<\sup C(a_1)$. 
Since $C_{n-1}(a_1)$ is a final part of $C(a_1)$, there exists an $a_0\in C_{n-1}(a_1)$ such that $C(a)<a_0$; clearly, $C(a)<C(a_0)$. By Claim 1, $a_0\in C_{n-1}(a_1)$ implies $C(a_0)\subseteq C(a_1)$.  
Additionally, note that $\sup C(a)<\sup C(a_1)$ implies that the non-empty set $C(a)\cap C(a_1)$ is a final part of $C(a)$. Since $C_{n-1}(a)$ is also a final part of $C(a)$, we conclude $C_{n-1}(a)\cap C(a_1)\neq \emptyset$. 
Choose $b_0\in C_{n-1}(a)\cap C(a_1)$; $b_0\in C(a_1)$ in particular implies $a_1<b_0$.
By Claim 1, $b_0\in C_{n-1}(a)$ implies $C(b_0)\subseteq C(a)$, which together with $C(a)<a_0$ implies $C(b_0)<C(a_0)$ and $b_0<a_0$. 
Now, $C(b_0)<C(a_0)$, $b_0\in C(a_1)$ and $C(a_0)\subseteq C(a_1)$ imply $C(b_0)\subseteq C(a_1)$. 
Therefore, we have $a_1<b_0<a_0$  such that $C(b_0)<C(a_0)$ and $C(a_0)\cup C(b_0)\subseteq C(a_1)$; this is impossible by Lemma \ref{Lemma_wom_order_property}. \hfill$\blacksquare$

\begin{Claim}
$E_n(x,y)$ iff $x\in C(y)\vee y\in C(x)\vee x=y$, and $E_n\in\mathcal E_p$ is a $\triangleleft$-convex equivalence relation.
\end{Claim}

\noindent {\it Proof.}  
By Claim 2 we see that $E_n(x,y)$ if and only if $C(x)\cap C(y)\neq 0$. 
Since $\max C(x)$ does not exist, $C(x)\cap C(y)\neq 0$ is easily seen to be equivalent to $x\in C(y)\vee y\in C(x)\vee x=y$, and since $x\in C(y)$ is a relatively $A$-definable relation, the relation $E_n$ is also relatively $A$-definable. 
It remains to show that $E_n$ is $\triangleleft$-convex, so assume that $x\triangleleft z\triangleleft y$ and $E_n(x,y)$ hold. If $C(z)\cap (C(x)\cup C(y))\neq 0$, then by Claim 2 we have $z\in [x]_{E_n}=[y]_{E_n}$ and we are done. So suppose $C(z)\cap (C(x)\cup C(y))= 0$. Since these are convex sets and $C(x)\cap C(y)\neq 0$, we see that $C(z)<C(x)\cup C(y)$ or $C(x)\cup C(y)<C(z)$ hold. 
To rule out the first option, note that $C(z)<C(x)$ implies $x\in C_n(z)$, which contradicts $x\triangleleft z$. The second option is impossible, as $C(y)<C(z)$ implies $z\in C_n(y)$, which contradicts $z\triangleleft y$.  \hfill$\blacksquare$

\smallskip
As a consequence of Claim 3 we see that orders $<$ and $\triangleleft$ agree on $p/E_n$: $[x]_{E_n}<[y]_{E_n}$ iff $[x]_{E_n}\triangleleft [y]_{E_n}$. 
Now, consider the convex decomposition of $(a,+\infty)_{\mathbf p}$ with respect to the formulae $x\triangleleft_{E_n}a$ and $a\triangleleft_{E_n} x$. 
Since the order $\triangleleft_{E_n}$ reverses the order $\triangleleft$ within each $E_n$-class and maintains the order of $E_n$-classes, it follows that the corresponding decomposition is $C_0(a)<\dots<C_{n-2}(a)<C_{n-1}(a)\cup C_n(a)$ and that the elements of the final component $C_{n-1}(a)\cup C_n(a)$ satisfy the formula $a\triangleleft_{E_n} x$. 
By the induction hypothesis, the relations $(E_1,\dots,E_{n-1})$ are relatively $A$-definable and $\triangleleft_{E_n}$-convex; note that each of them refines $E_n$, so they are all $\triangleleft$-convex. By the induction hypothesis and Remark \ref{Remark_order_<_vecE}(b,c) we also have $<=(\triangleleft_{E_n})_{(\id_{p(\Mon)},E_1,\dots,E_{n-1})}=\triangleleft_{(\id_{p(\Mon)},E_1,\dots,E_{n-1},E_n)}$, which is equivalent to $\triangleleft=<_{(\id_{p(\Mon)},E_1,\dots,E_n)}$.

So far, assuming that the elements of $C_n(a)$ satisfy $a\triangleleft x$, that is, the orders $<$ and $\triangleleft$ have the same orientation, we proved $\triangleleft=<_{(\id_{p(\Mon)},E_1,\dots,E_n)}$. Now, if the elements of $C_n(a)$ satisfy $x\triangleleft a$, then the orders $<$ and $\triangleleft^*$ have the same orientation, so $<= \triangleleft_{(\id_{p(\Mon)},E_1,\dots,E_n, \mathbf 1_p)}$ easily follows. To finalize the proof of the theorem, it remains to notice that the uniqueness of $\vec E$ follows from Remark \ref{Remark_order_<_vecE}(d).
\end{proof}

\subsection{Monotonicity theorems and other corollaries of Theorem \ref{Theorem_characterize_wom_orders}} \
 
\begin{Corollary}\label{Corollary_wom_is_ind_of_order} 
If $(p,<)$ is a weakly o-minimal pair over $A$, then so is the pair $(p,\triangleleft)$ for every relatively $A$-definable linear order $\triangleleft$ on $p(\Mon)$.
\end{Corollary}
\begin{proof}
 Suppose that $(p,<)$ is a weakly o-minimal pair over $A$ and $\triangleleft$ is a relatively $A$-definable linear order on $p(\Mon)$. By Theorem \ref{Theorem_characterize_wom_orders} there is a sequence $\vec E\in \mathcal E_p^{<\omega}$ such that $\triangleleft=<_{\vec E}$. By Remark \ref{Remark_weak_ominimal_firstrmk}(e), $(p,\triangleleft)$ is a weakly o-minimal pair over $A$. 
\end{proof}

\begin{Corollary}\label{Corollary_wom_pre-orders}
Suppose that $\mathbf p=(p,<)$ is a weakly o-minimal pair over $A$ and  $\preccurlyeq$ 
 a relatively $A$-definable total pre-order on $p(\Mon)$. Then there exists a unique strictly increasing sequence of equivalence relations $\vec E= (E_0,\dots,E_{n})\in \mathcal E_p^{n+1}$ such that $E_0$ is defined by $x\preccurlyeq y \land y\preccurlyeq x$ and for all $x,y\in p(\Mon)$: 
$$x\preccurlyeq y \ \ \mbox{ if and only if } \ \ E_0(x,y)\vee x<_{\vec E}y.$$
\end{Corollary} 
\begin{proof} 
Note that $E_0\in\mathcal E_p$, so it is convex by Proposition \ref{Prop_E_p_is convex_linear}(i).
By Fact \ref{Fact_L_P_sentence} there is an $A$-definable set $D\supseteq p(\Mon)$ such that: $x<y$ defines a linear order on $D$, $x\preccurlyeq y$ defines a total pre-order on $D$, and $x\preccurlyeq y \land y\preccurlyeq x$ defines a convex equivalence relation, $\hat E_0$, on $(D,<)$. 
The canonical projection $\pi:p(\Mon)\to p/\hat E_0(\Mon)$ is relatively $A$-definable, so by Proposition \ref{Prop_pwom_implies_dclq_wom}(i) the pair $(p/\hat E_0,<)$ is weakly o-minimal.
Define: $[x]_{\hat E_0}\triangleleft [y]_{\hat E_0}$ iff $\lnot \hat E_0(x,y)\land x\preccurlyeq y$; it is easy to see that $\triangleleft$ is a relatively $A$-definable linear order on $p/\hat E_0(\Mon)$, so we can apply Theorem \ref{Theorem_characterize_wom_orders}. Let $\vec E'=(\id_{p/\hat E_0(\Mon)},E_1',\dots,E_n')\in\mathcal E_{p/\hat E_0}^{n+1}$ be a strictly increasing sequence such that $\triangleleft=<_{\vec E'}$. By Proposition \ref{Prop_pwom_implies_dclq_wom}(ii) there is a strictly increasing sequence $\vec E=(E_0,E_1,\dots,E_n)\in\mathcal E_p^{n+1}$ such that $\pi(\vec E)=\vec E'$. For all $x,y\in p(\Mon)$ satisfying $\lnot \hat E_0(x,y)$ we have: 
$$x\preccurlyeq y\ \Leftrightarrow\ [x]_{\hat E_0}\triangleleft [y]_{\hat E_0}\ \Leftrightarrow\ [x]_{\hat E_0}<_{\vec E'}[y]_{\hat E_0}\ \Leftrightarrow\ x<_{\vec E}y,$$ where the last equivalence easily holds by induction on $n$.
\end{proof}

\begin{proof}[{\it Proof of Theorem \ref{Theorem1}.}]
(i) The conclusion follows easily by Corollary \ref{Corollary_wom_pre-orders}, after noting that $f(x)\leqslant f(y)$ relatively defines a total pre-order on $p(\Mon)$.

(ii) Let $q=f(p)$ and let $\vec E$ satisfy the conclusion of (i). Then $f:p(\Mon)\to q(\Mon)$ is $((<_p)_{\vec E},<)$-increasing.  By Proposition \ref{Prop_pwom_implies_dclq_wom}(ii),  $f(\vec E)=(f(E_0), \ldots,f(E_n))$ is an increasing sequence of convex, relatively $A$-definable equivalence relations on $q(\Mon)$, and $f$ is $(<_p,<_{f(\bar E)})$-increasing by Lemma \ref{Cor_f_of_convex_is_convex}(i). Now,  the sequence $\vec F$ satisfying the conclusion in (ii) can be found by routine compactness.
\end{proof}

In the following two theorems we deduce Theorem \ref{Theorem2}.

\begin{Theorem}[Local monotonicity]\label{Theorem Local Mono}
Suppose that $\mathbf p=(p,<_p)$ is a weakly o-minimal pair over $A$, $(D,<)$ is a $A$-definable linear order, and $f:p(\Mon)\to D$ is a relatively $A$-definable non-constant function.  Then there exists $E\in\mathcal E_p\smallsetminus\{\id_{p(\Mon)}\}$ such that the restriction of $f$ to each $E$-class is either constant or strictly $(<_p,<)$-monotone.   
\end{Theorem}
\begin{proof}
If $\Ker f\neq\id_{p(\Mon)}$ then $E=\Ker f$ satisfies the conclusion of the theorem, as $f$ is constant on each $E$-class. If $f$ is strictly $(<_p,<)$-monotone, then $E=\mathbf 1_p$ satisfies the conclusion.  The remaining case is where $\Ker f=\id_{p(\Mon)}$ holds and $f$ is not strictly $(<_p,<)$-monotone. Let $\vec E=(E_0,\dots,E_n)\in \mathcal E_p^{n+1}$ be given by Theorem \ref{Theorem1}(i). In particular, $E_0=\Ker f=\id_{p(\Mon)}$, and $f$ not strictly $(<_p,<)$-monotone implies $n\geqslant 1$. Then we easily see that $E=E_1$ satisfies the conclusion of the theorem. 
\end{proof}

\begin{Theorem}[Upper monotonicity]\label{Theorem Upper Mono}
Suppose that $\mathbf p=(p,<_p)$ is a weakly o-minimal pair over $A$, $(D,<)$ is a $A$-definable linear order and $f:p(\Mon)\to D$ is a relatively $A$-definable non-constant function.
\begin{enumerate}[(i)]
\item There exists a $E\in\mathcal E_p\smallsetminus \{\mathbf 1_p\}$ such that one of the following two conditions holds for all $x_1,x_2$ realizing $p$:
\begin{center}$ [x_1]_E<_p[x_2]_E\Rightarrow  f(x_1)< f(x_2)$ \ \ \ or \ \ \  $[x_1]_E<_p[x_2]_E\Rightarrow  f(x_1)> f(x_2)$.
\end{center}
\item If $q=f(p)$, then there exists a $E\in\mathcal E_p\smallsetminus \{\mathbf 1_p\}$ such that the function $f_E:p(\Mon)/E\to q(\Mon)/f(E)$, defined by $f_E([x]_E)=[f(x)]_{f(E)}$, is strictly $(<_p,<)$-monotone.
\end{enumerate}
\end{Theorem}
\begin{proof}
(i) Let $\vec E=(E_0,\dots,E_n)\in \mathcal E_p^{n+1}$ be an increasing sequence given by Theorem \ref{Theorem1}(i). 
If $E_n\neq \mathbf 1_p$ then $<_p$ agrees with $(<_p)_{(E_0,\ldots,E_{n})}$ on $p(\Mon)/E_n$, so for $E=E_n$ the first option of (i) holds. Otherwise, $>_p$ agrees with $(<_p)_{(E_0,\ldots,E_{n-1})}$ on $p(\Mon)/E_{n-1}$, so for $E=E_{n-1}$ the second option of (i) holds.

(ii) Let $E\in\mathcal E_p\smallsetminus \{\mathbf 1_p\}$ satisfy the conclusion of (i). Then the function $f_E$ is well defined by Proposition \ref{Prop_pwom_implies_dclq_wom}(iii), and strictly $(<_p,<)$-monotone by (i).  
\end{proof}

Now we turn to the context of weakly o-minimal theories and prove Theorem \ref{Theorem3}.

\begin{proof}[{\it Proof of Theorem \ref{Theorem3}.}]
(i) First, note that the pair $(p,<)$ is weakly o-minimal for every $p\in S_1(A)$. For each $p\in S_1(A)$ we will find a formula $\theta_p\in p$ and a sequence $\vec E_p\in\mathcal E_p^{<\omega}$ such that 
\setcounter{equation}{0}
\begin{equation}\mbox{$\theta_p(\Mon)$ is a $<$-convex subset of $\Mon$ \ \ and  \ \  $f_{\restriction \theta_p(\Mon)}$ is $(<_{\vec E_p},\triangleleft)$-increasing,}
\end{equation}
in the following way. 
If $f$ is constant on $p(\Mon)$, then by compactness there is a $\theta_p(x)\in p$ such that $f$ is constant on $\theta_p(\Mon)$; by the weak o-minimality of the theory we may suppose that $\theta_p(\Mon)$ is convex. Set $\vec E_p=\id_{p(\Mon)}$ and note that $f_{\restriction\theta_p(\Mon)}$ is $(<_{\vec E_p},\triangleleft)$-increasing, so condition (1) is satisfied in this case.
The other case is where $f$ is non-constant on $p(\Mon)$. Then by Theorem \ref{Theorem1}(i) there is an increasing sequence $\vec E_p'\in\mathcal E_p^{<\omega}$ such that $f_{\restriction p(\Mon)}$ is $(<_{\vec E_p'},\triangleleft)$-increasing. Note that the following are tp-universal properties:

-- \ ${\vec E_p'}$ is an increasing sequence of $<$-convex equivalence relations on $p(\Mon)$;

-- \ $f:p(\Mon)\to D$ is a $(<_{\vec E_p'},\triangleleft)$-increasing function.  

\noindent 
By Fact \ref{Fact_L_P_sentence} there is a $\theta_p(x)\in p$ and an increasing sequence of $A$-definable $<$-convex equivalence relations $\vec E_p$ on $\theta_p(\Mon)$ such that $\vec E_{p\restriction p(\Mon)}=\vec E_p'$ and $f_{\restriction \theta_p(\Mon)}$ is $(<_{\vec E_p},\triangleleft)$-increasing; again, we can assume that $\theta_p(\Mon)$ is a $<$-convex subset of $\Mon$, so condition (1) is satisfied.

Since $\{[\theta_p(x)]\mid p\in S_1(A)\}$ is an open cover of $S_1(A)$, by compactness, we can find a finite subcover $\{[\theta_{p_1}(x)],\dots,[\theta_{p_n}(x)]\}$. By a simple modification, we can assume that $\theta_{p_i}(x)$'s are mutually contradictory. 
It remains to construct the sequence $\vec E$. First, note that we may assume that all $\vec E_{p_i}$'s are of the same length. Indeed, if $m$ is the maximal length, then every shorter sequence $\vec E_{p_i}$ can be expanded by adding an appropriate number of $\id_{\theta_{p_i}(\Mon)}$ at the beginning; this does not change $<_{\vec E_{p_i}}$. So, let $\vec E_{p_i}=(E_{p_i,1},\dots,E_{p_i,m})$. Now, define $\vec E=(E_1,\dots,E_m)$ in the obvious way: set $E_j$ equal to $E_{p_i,j}$ on the part $\theta_{p_i}(\Mon)$, leaving the elements of different parts unrelated. Clearly, $<_{\vec E}$ equals $<_{\vec E_{p_i}}$ on $\theta_{p_i}(\Mon)$, so the conclusion follows.

\smallskip
(ii) We need the following observation: if $F$ is a convex definable equivalence relation on $\Mon$, then the set $\{a\in \Mon\mid [a]_F\mbox{ is infinite}\}$ is definable, which follows from the fact that weak o-minimality of $T$ guarantees that $F$ has only finitely many classes with finitely many but more than one element.  

For each $p\in S_1(A)$ we find a $\theta_p(x)\in p$ and an $A$-definable convex equivalence relation $E_p$ on $\theta_p(\Mon)$, such that $\theta_p(\Mon)$ is convex and:
\begin{enumerate}[\hspace{10pt} (1)]
    \item If $f_{\restriction p(\Mon)}$ is constant, then $f_{\restriction\theta_p(\Mon)}$ is also constant and $E_p=\theta_p(\Mon)^2$;
    \item If $f_{\restriction p(\Mon)}$ is non-constant, then each $E_p$-class is infinite and $f_{\restriction[a]_{E_p}}$ is constant/strictly $(<,\triangleleft)$-increasing/strictly $(<,\triangleleft)$-decreasing uniformly for all $a\in \theta_p(\Mon)$. 
\end{enumerate}
(1) is fulfilled as in the proof of (i). For (2), assume that $f$ is non-constant. By Theorem \ref{Theorem Local Mono}, there exists a $E_p\in \mathcal E_p\smallsetminus\{\id_{p(\Mon)}\}$ such that $f_{\restriction[a]_{E_p}}$ is constant/strictly $(<,\triangleleft)$-increasing/strictly $(<,\triangleleft)$-decreasing on $[a]_{E_p}$ uniformly for all $a\in p(\Mon)$; for simplicity, assume that $f_{\restriction[a]_{E_p}}$ is strictly $(<,\triangleleft)$-increasing. Note that $E_p\neq \id_{p(\Mon)}$ implies that each $E_p$-class is infinite. 
As in the proof of part (i), we find $\theta_p(x)\in p$ and a convex $A$-definable equivalence relation on $\theta_p(\Mon)$, also denoted by $E_p$, such that the restriction of $f$ to each $E_p$-class is strictly $(<,\triangleleft)$-increasing. By the above observation, we can further shrink $\theta_p(\Mon)$ so that condition (2) is satisfied. 

As in the proof of (i), choose a finite subcover $\{[\theta_{p_1}(x)],\dots,[\theta_{p_n}(x)]\}$ of $\{[\theta_{p}(x)]\mid p\in S_1(A)\}$. Let $\mathcal C=\{\theta_{p_i}(\Mon)\mid \ i\leq n\}$. For each $C=\theta_{p_i}(\Mon)\in \mathcal C$ denote $E_C:=E_{p_i}$. 
Thus, we have a finite $A$-definable convex cover $\mathcal C$ of $\Mon$, and for each $C\in \mathcal C$ a convex $A$-definable definable equivalence relation $E_C$ on $C$, such that at least one of  the following two conditions is satisfied:
\begin{enumerate}[\hspace{10pt} (1)]
\item [(3)]  $f_{\restriction C}$ is constant and $E_C=C^2$;
\item [(4)]  Each $E_C$-class is infinite and  $f_{[a]_{E_C}}$ is constant/strictly $(<,\triangleleft)$-increasing/strictly $(<,\triangleleft)$-decreasing uniformly for all $a\in C$. 
\end{enumerate}
Refine $\mathcal C$ in an obvious way to become a convex partition of $\Mon$; attach to each member of the partition the restriction of an appropriately chosen $E_C$. Note that after this modification, we have at most finitely many ``new" finite $E_C$-classes (parts of previously infinite classes); each of those classes is $A$-definable, so we can split each of them into single-element classes and form a new $A$-definable convex partition, each of whose members satisfies at least one of conditions (3) and (4).
It is easy to see that the partition $\mathcal C$ and the relation $E=\bigcup _{C\in\mathcal C}E_C$ satisfy the conclusion of (ii).
\end{proof}

\subsection{Weak quasi-o-minimality}\

The following proposition describes the weak quasi-o-minimality of a theory as a ``local'' property of its complete $1$-types. 
This extends \cite[Theorem 1(ii)]{MTwmon}, in which we showed that the weak quasi-o-minimality of $T$ does not depend on the particular choice of the linear order.

\begin{Proposition}\label{Proposition qwom iff all 1-types are wom}
A complete first-order theory $T$ with infinite models is weakly quasi-o-minimal if and only if every type $p\in S_1(T)$ is weakly o-minimal. 
\end{Proposition}
\begin{proof}
It is easy to see that if $T$ is weakly quasi-o-minimal with respect to $<$, then $(p,<)$ is a weakly o-minimal pair over $\emptyset$ for every $p\in S_1(T)$. For the converse, assume that every $p\in S_1(T)$ is weakly o-minimal. In particular, every $p\in S_1(T)$ is linearly ordered (there is a relatively $\emptyset$-definable linear order on $p(\Mon)$), so by routine compactness we may find an $\emptyset$-definable linear order on whole $\Mon$. We prove that $T$ is weakly quasi-o-minimal with respect to $<$.

Let $D\subseteq\Mon$ be definable. By Corollary \ref{Corollary_wom_is_ind_of_order}, $(p,<)$ is a weakly o-minimal pair for every $p\in S_1(T)$, so $D\cap p(\Mon)$ has finitely many convex components, say $n_p$, in $(p(\Mon),<)$.  By compactness, as in the proof of Lemma \ref{Lemma_reldef_of_convex_components}(i), we find $\theta_p(x)\in p$ such that $D\cap \theta_p(\Mon)$ has $n_p$ convex components in $(\theta_p(\Mon),<)$, say $D_{p,i}$, $1\leqslant i\leqslant n_p$: $D\cap \theta_p(\Mon)=\bigcup_{i=1}^{n_p} D_{p,i}$. Note that each of $D_{p,i}$ is definable.
Setting $D_{p,i}^{conv}$ as the convex hull of $D_{p,i}$ in $(\Mon,<)$, we have $D_{p,i}= D_{p,i}^{conv}\cap\theta_p(\Mon)$, so $D\cap\theta_p(\Mon)= \bigcup_{i=1}^{n_p}D_{p,i}^{conv}\cap\theta_p(\Mon)$. 
Since $\{\theta_p(x)\mid p\in S_1(T)\}$ covers $S_1(T)$, by compactness we find a finite sub-cover $\{\theta_{p_j}(x)\mid 1\leqslant j\leqslant m\}$. We have:
$$D=D\cap\Mon= D\cap \bigcup_{j=1}^m\theta_{p_j}(\Mon)=\bigcup_{j=1}^m D\cap \theta_{p_j}(\Mon)= \bigcup_{j=1}^m\bigcup_{i=1}^{n_{p_j}}D_{p_j,i}^{conv}\cap\theta_{p_j}(\Mon),$$ 
which is a Boolean combination of convex and $\emptyset$-definable sets, and we are done.
\end{proof}

The proposition motivates the following definition. 

\begin{Definition}
A partial type $\pi(x)$ is {\em weakly quasi-o-minimal over $A$} if $\pi(x)$ is over $A$ and every $p\in S_x(A)$ extending $\pi(x)$ is weakly o-minimal; in that case we say that the set $\pi(\Mon)$ is 
{\em weakly quasi-o-minimal over $A$}.
\end{Definition}

Note that if $\pi(x)$ is weakly quasi-o-minimal over $A$, then $\pi(x)$ is weakly quasi-o-minimal over any $B\supseteq A$ as complete extensions of weakly o-minimal types are weakly o-minimal by Lemma \ref{Lemma_reldef_of_convex_components}(ii).

\begin{Proposition}
Let $P$ be type-definable over $A$. Then $P$ is weakly quasi-o-minimal over $A$ if and only if there exists a relatively $A$-definable linear order $<$ on $P$ such that every relatively definable subset of $P$ is a Boolean combination of $<$-convex and relatively $A$-definable sets. In that case, the latter is true for any relatively $A$-definable linear order $<$ on $P$.
\end{Proposition}
\begin{proof}
The proof of Proposition \ref{Proposition qwom iff all 1-types are wom} with obvious modifications goes through. 
\end{proof}

Let $(X,<)$ be a linear order and $\mathcal C=(C_1,\dots,C_n)$ a partition of $X$. By $<_{\mathcal C}$ we denote the order obtained by keeping the original order within each component and defining $C_1<_{\mathcal C}\dots<_{\mathcal C}C_n$. Combining the arguments from previous proofs, the following theorem can be routinely derived. 

\begin{Theorem}\label{Theorem_wqom_WMONO}
Suppose that a type-definable set $P$ is weakly quasi-o-minimal over $A$, $<$ is a relatively $A$-definable linear order on $P$, $(D,\triangleleft)$ is an $A$-definable linear order, and $f:P\to D$ is a relatively $A$-definable function. Then there are $A$-definable extensions $(\hat P,<)$ of $(P,<)$ and $\hat f:\hat P\to D$, an $A$-definable partition
$\mathcal C$ of $\hat P$ and an increasing sequence of $A$-definable $<_{\mathcal C}$-convex equivalence relations $\vec E$ on $\hat P$  such that $\hat f$ is $((<_{\mathcal C})_{\vec E},\triangleleft)$-increasing on each member of $\mathcal C$.
\end{Theorem}

\section{Nonorthogonality of so-types}\phantomsection\label{Section4}
In this section, we study forking independence in the context of so-types. We introduce the notion of $\mathbf p$-genericity, and prove Theorem \ref{Theorem4}.

\begin{Lemma}\label{Lemma_so_forking_equals_bounded}
Let $\mathbf p=(p,<)$ be an so-pair over $A$ and let $\phi(x,b)$ be any formula. Then the following conditions are equivalent:

(1)  $p(x)\cup \{\phi(x,b)\}$ divides over $A$; 

(2) $p(x)\cup \{\phi(x,b)\}$ forks over $A$;

(3) $\phi(x,b)$ relatively defines a bounded subset of $(p(\Mon),<)$. 
\end{Lemma}
\begin{proof}
(1)$\Rightarrow$(2) is immediate. To prove (2)$\Rightarrow$(3)  assume that $p(x)\cup\{\phi(x,b)\}$ forks over $A$. By Fact \ref{Fact_so_basic}(ii),  $\mathbf p_l$ and $\mathbf p_r$ are nonforking extensions of $p$, so $\phi(x,b)\notin  \mathbf p_l(x)\cup\mathbf p_r(x)$. Then $\phi(x,b)\notin  \mathbf p_r(x)$ implies that $p(\Mon)\cap\phi(\Mon,b)$ is upper bounded, while $\phi(x,b)\notin  \mathbf p_l(x)$ implies that $p(\Mon)\cap\phi(\Mon,b)$ is lower bounded. Therefore, $p(\Mon)\cap\phi(\Mon,b)$ is bounded in $(p(\Mon),<)$.

(3)$\Rightarrow$(1) Assume that $p(\Mon)\cap\phi(\Mon,b)$ is bounded. Let $b_0=b$, and let $a_0,a_1\models p$ be such that $a_0<p(\Mon)\cap\phi(\Mon,b_0)<a_1$. By compactness there is $\theta(x)\in p$ such that $a_0<\theta(\Mon)\cap\phi(\Mon,b_0)<a_1$. 
Let $f\in\Aut_A(\Mon)$ be such that $f(a_0)=a_1$. Define $a_{n+1}:=f(a_n)$ and $b_{n+1}:=f(b_n)$; each $b_n\models\tp(b/A)$. By induction, we see that $a_n<\theta(\Mon)\cap\phi(\Mon,b_n)<a_{n+1}$. So $\{\theta(\Mon)\cap \phi(\Mon,b_n)\mid n<\omega\}$ is $2$-inconsistent, which shows that $p(x)\cup\{\phi(x,b)\}$ divides over $A$.
\end{proof}

Immediately from the lemma, we have the following corollary.
\begin{Corollary}\label{Corolary_of_forking_eq_bounded}
Let $\mathbf p=(p,<)$ be an so-pair over $A$ and let $B\supseteq A$. 
\begin{enumerate}[(i)]
\item The type $p$ has exactly two global nonforking extensions: $\mathbf p_r$ and $\mathbf p_l$.

\item The only nonforking extensions of $p$ in $S(B)$ are $(\mathbf p_{r})_{\restriction B}$ and $(\mathbf p_{l})_{\restriction B}$.

\item The following holds for all $q\in S(B)$ that extend $p$: $q$ forks over $A$ if and only if the locus $q(\Mon)$ is bounded in $(p(\Mon),<)$. 
\end{enumerate}
\end{Corollary}

Notice that condition ``$p(x)\cup\{\phi(x)\}$ forks over $A$'' from Lemma \ref{Lemma_so_forking_equals_bounded} does not refer to any particular relatively definable order $<$ on the locus of the so-type $p\in S(A)$, so the equivalent condition, $\phi(p(\Mon))$ is bounded in $(p(\Mon),<)$,  holds for all relatively $A$-definable orders on $p(\Mon)$. 

\begin{Definition}
Let $p(x)\in S(A)$ be an so-type. We will say that $\phi(x)$ is a {\em $p$-bounded formula}  if $p(x)\cup\{\phi(x)\}$ forks over $A$. 
\end{Definition} 
 
Therefore, a formula $\phi(x)$ is $p$-bounded if and only if the corresponding relatively defined
subset of $p(\Mon)$ is bounded in $(p(\Mon),<)$ for every relatively $A$-definable linear order $<$.

\begin{Definition}
Let $\mathbf p=(p,<)$ be an so-pair over $A$ and $B$ a small set. Define:
\begin{center}
 $\mathcal L_{\mathbf p}(B):=(\mathbf p_{l})_{\restriction AB}(\Mon)$; \ \ \ $\mathcal R_{\mathbf p}(B):=(\mathbf p_{r})_{\restriction AB}(\Mon)$; \ \ \   $\mathcal D_p(B):=\{a\in p(\Mon)\mid a\dep_A B\}$.
 \end{center}
\end{Definition}

Let $\mathbf p=(p,<)$ be an so-pair. Then the pair $\mathbf p^*=(p,>)$ is also an so-pair over $A$, called the reverse of $\mathbf p$. 
It is easy to see that $\mathcal L_{\mathbf p}(B)=\mathcal R_{\mathbf p^*}(B)$ and $\mathcal R_{\mathbf p}(B)=\mathcal R_{\mathbf p}(B)$ hold for all $B$.

Recall that $p,q\in S(A)$ are {\em forking orthogonal}, $p\fwor q$, if $a\ind_A b$ holds for all $a\models p$ and $b\models q$.

\begin{Lemma}\label{Lemma_R_p_basic} Let $\mathbf p=(p,<)$ be an so-pair over $A$.
\begin{enumerate}[(i)]
\item $\mathcal L_{\mathbf p}(B)$ is an initial and $\mathcal R_{\mathbf p}(B)$ is a final part of $(p(\Mon),<)$.

\item $a\ind_A B$ \ if and only if  \  $a\in \mathcal L_{\mathbf p}(B)\cup\mathcal R_{\mathbf p}(B)$.

\item $\mathcal D_p(B)$ is the union of all $p$-bounded, relatively $AB$-definable subsets of $p(\Mon)$;\\ $\mathcal D_p(B)=p(\Mon)\smallsetminus (\mathcal L_{\mathbf p}(B)\cup \mathcal R_{\mathbf p}(B))$ is a convex, possibly empty,   bounded subset of $(p(\Mon),<)$.

\item  There are three possible cases:
     \begin{enumerate}[1$^\circ$]
     \item   $p\nfor \tp(B/A)$.  Then $\mathcal L_{\mathbf p}(B)< \mathcal D_p(B)< \mathcal R_{\mathbf p}(B)$  is a convex partition of $p(\Mon)$;
     \item   $p\fwor \tp(B/A)$ and $p\nwor \tp(B/A)$. Then $\mathcal D_p(B)=\emptyset$ and $\mathcal L_{\mathbf p}(B)< \mathcal R_{\mathbf p}(B)$  is a convex partition of $p(\Mon)$.
     \item   $p\wor \tp(B/A)$.  Then $\mathcal L_{\mathbf p}(B)=\mathcal R_{\mathbf p}(B)=p(\Mon)$ and $\mathcal D_p(B)=\emptyset$.
     \end{enumerate}
\item If $p\nfor \tp(B/A)$ and if $(p(\Mon),<)$ is a suborder of an $A$-definable order $(D,<)$, then the sets $\{x\in D\mid \mathcal D_p(B)<x\}$ and $\{x\in D\mid x<\mathcal D_p(B)\}$ are type definable over $AB$.
\end{enumerate}
\end{Lemma}
\begin{proof}  
(i) is Fact \ref{Fact_so_basic}(iii), (ii)--(iv) follow by (i) and Corollary \ref{Corolary_of_forking_eq_bounded}.

(v) By (iii), we know that the set $\mathcal D_p(B)$ is the union of all relatively $AB$-definable subsets $C\subseteq p(\Mon)$ that are $p$-bounded; that is, $a_1<C<a_2$ holds for some $a_1,a_2\in p(\Mon)$. Let $\theta(x)$ relatively define $C$ in $p(\Mon)$. Then $\{\theta(x)\}\cup p(x)\vdash a_1<x<a_2$, so by compactness there is a formula $\theta'(x)\in p$ such that 
$a_1<(\theta\land \theta')(\Mon)<a_2$. Note that the $L_{AB}$-formula $\theta_C(x)=\theta(x)\land \theta'(x)$ relatively defines $C$. Let $\Pi(x)$ be the type consisting of $x\in D$ and all formulae $\theta_C(\Mon)<x$. Clearly, $\Pi(x)$ defines $\{x\in D\mid \mathcal D_p(B)<x\}$.
\end{proof}

By part (i) of the previous lemma, the set $\mathcal R_{\mathbf p}(B)$ is a final part of $(p(\Mon),<)$, so its elements can be thought of as being realizations of $p$ which are ``as far to the $<$-right (from the point of view) of $B$ as possible''. This is formalized in the next definition.
 
\begin{Definition}
Let $\mathbf p=(p,<)$ be an so-pair over $A$, $B$ be a small set of parameters,  and $a\in p(\Mon)$. 
\begin{itemize}
    \item[--] {\em $a$ is right $\mathbf p$-generic over $B$}, denoted by $B\triangleleft^{\mathbf p} a$, if $a\in \mathcal R_{\mathbf p}(B)$ ($a\models (\mathbf p_r)_{\restriction AB}$); 
    \item[--]  $a$ is {\em left $\mathbf p$-generic over $B$} if $a\in\mathcal L_{\mathbf p}(B)$ ($a\models(\mathbf p_l)_{\restriction AB}$).
    \end{itemize}
\end{Definition}

\begin{Remark}\label{Remark_R_p_basic} 
(a) For all $B$, the left (right) $\mathbf{p}$-generic element over $B$ exists. 

(b) To avoid possible confusion, we {\em do not} choose a special notation for left $\mathbf p$ -genericity. However, we can express that $a$ is left $\mathbf p$-generic over $B$ by \ $B\triangleleft^{\mathbf p^*}a$, where $\mathbf p^*=(p,>)$ is the reverse of $\mathbf p$.
In general, $a\triangleleft^{\mathbf p}B$ is meaningless, unless $B\models p$. 

(c) When we write \ $B\triangleleft^{\mathbf p}a\triangleleft^{\mathbf q}b$,  where $\mathbf p,\mathbf q$ are so-pairs over $A$, we mean the conjunction: $B\triangleleft^{\mathbf p}a$ and $a\triangleleft^{\mathbf q}b$; here, $a\models p$ and $b\models q$ are implicit. In general, $B\triangleleft^{\mathbf p}a\triangleleft^{\mathbf q}b$ does not imply $B\triangleleft^{\mathbf q}b$ (unless $p\wor q$ or $\mathbf p,\mathbf q$ are directly nonorthogonal so-pairs, as defined and explained later). 

(d) A left $\mathbf{p}$-generic element over $B$ can also be right $\mathbf{p}$-generic; this occurs precisely when $p \wor \tp(B/A)$.  
\end{Remark}

Next, we list some of the basic properties of $\mathbf p$-genericity.  

\begin{Lemma}\label{Lemma_R_p_basic2} Let $\mathbf p=(p,<_p)$ be an so-pair over $A$. Then for all $a,a'\models p$ and all $B$: 
\begin{enumerate}[ (i)]
\item  $B\triangleleft^{\mathbf p}a<_p a'$ implies $B\triangleleft^{\mathbf p} a'$; in particular, $\triangleleft^{\mathbf p}$ is a transitive binary relation on $p(\Mon)$. 
 \item  $a\triangleleft^{\mathbf p} a'\Leftrightarrow a'\triangleleft^{\mathbf p^*}a$; \ that is, $a'$ is right $\mathbf p$-generic over $a$ iff $a$ is left $\mathbf p$-generic over $a'$. 

\item    $a\ind_A B \Leftrightarrow (B\triangleleft^{\mathbf p} a\vee B\triangleleft^{\mathbf p^*}a )$ \ (that is, $a$ is left or right $\mathbf p$-generic over $B$). 

\item $a\ind_A a'\Leftrightarrow (a\triangleleft^{\mathbf p} a'\vee a'\triangleleft^{\mathbf p} a)$;

\item $a\triangleleft^{\mathbf p}a'\Leftrightarrow \mathcal D_p(a)<_p a'\Leftrightarrow a<_p\mathcal D_p(a')\Leftrightarrow (a\ind_Aa'\land a<_pa')$. 
\end{enumerate}
\end{Lemma}
\begin{proof}   
Parts (i)-(iii) are restated earlier proved facts: 

(i) is Lemma \ref{Lemma_R_p_basic}(i): \  $\mathcal R_{\mathbf p}(B)$ is a final part of $(p(\Mon),<_p)$;

(ii) is Fact \ref{Fact_so_basic}(iv): \  $a'\models (\mathbf p_r)_{\restriction aA}\Leftrightarrow a\models (\mathbf p_l)_{\restriction a'A}$; that is, $a'\in \mathcal R_{\mathbf p}(a)\Leftrightarrow a\in\mathcal L_{\mathbf p}(a')$. 

(iii) is Lemma \ref{Lemma_R_p_basic}(ii): \ $a\ind_A B$ \ iff  \  $a\in \mathcal L_{\mathbf p}(B)\cup\mathcal R_{\mathbf p}(B)$. \ 

(iv) follows easily from (ii) and (iii).  

(v) Clearly, $a\in\mathcal D_p(a)$, so the set $\mathcal D_p(a)$ is not empty, and Lemma \ref{Lemma_R_p_basic}(iv) implies that $\mathcal L_{\mathbf p}(a)<_p \mathcal D_p(a)<_p\mathcal R_{\mathbf p}(a)$ is a convex partition of $p(\Mon)$. This implies $a'\in \mathcal R_{\mathbf p}(a)\Leftrightarrow \mathcal D_p(a)< a'$, that is, 
$a\triangleleft^{\mathbf p}a'\Leftrightarrow \mathcal D_p(a)< a'$, proving the first equivalence. 
Similarly, we infer $a\in\mathcal L_{\mathbf p}(a')\Leftrightarrow a<_p\mathcal D_p(a')$, which combined with $a'\in \mathcal R_{\mathbf p}(a)\Leftrightarrow a\in\mathcal L_{\mathbf p}(a')$ from Fact \ref{Fact_so_basic}(iv), implies $a'\in\mathcal R_{\mathbf p}(a)\Leftrightarrow a<_p\mathcal D_p(a')$ and proves the second equivalence. The third one follows from (iv).
\end{proof}

As an immediate corollary of part (iv) of the lemma, we obtain the following.

\begin{Corollary}\label{Cor_pairwise_ind_sequence_is_tr_comparable}
Let $\mathbf p=(p,<)$ be an so-pair over $A$ and let $I$ be a set of realizations of $p$ that are pairwise independent over $A$. Then $I$ is totally ordered by $\triangleleft^{\mathbf p}$. In particular, Morley sequences in $\mathbf p_r$ are $\triangleleft^{\mathbf p}$-increasing and Morley sequences in $\mathbf p_l$ are $\triangleleft^{\mathbf p}$-decreasing.    
\end{Corollary}
 
In the next lemma, we show that between any pair of $\triangleleft^{\mathbf q}$-related tuples, we can $\triangleleft^{\mathbf p}$-insert a realization of $p$; further in the text, we will refer to this as the density property of $\triangleleft^{\mathbf p}$.

\begin{Lemma}\label{Lemma_existence_anddensity_ofrightgenerics} 
Let $\mathbf p=(p,<_p)$ and $\mathbf q=(q,<_q)$ be so-pairs over $A$. Assume $B\triangleleft^{\mathbf p} a$. Then there exists $b\models q$ such that $B\triangleleft^{\mathbf q} b\triangleleft^{\mathbf  p} a$. 
\end{Lemma}
\begin{proof} Choose $b'\models q$ satisfying $B\triangleleft^{\mathbf q} b'$, and then $a'\models p$ satisfying $Bb'\triangleleft^{\mathbf p} a'$; in particular, $B\triangleleft^{\mathbf p} a'$ and $b'\triangleleft^{\mathbf p} a'$ hold. Then $B\triangleleft^{\mathbf p} a'$  and $B\triangleleft^{\mathbf p} a$  imply  $\tp(a/AB)=\tp(a'/AB)=(\mathbf p_{r})_{\restriction AB}$. Let $f\in \Aut_{AB}(\Mon)$ map $a'$ to $a$. Put $b=f(b')$. Then $B\triangleleft^{\mathbf q} b'\triangleleft^{\mathbf p} a'$ implies the desired conclusion $B\triangleleft^{\mathbf q} b\triangleleft^{\mathbf p} a$.
\end{proof}

\begin{Lemma}\label{Lemma_forking_symmetry}
$a\ind_A b\Leftrightarrow b\ind_A a$ \  holds for all realizations of so-types over $A$.   
\end{Lemma}
\begin{proof}
Suppose that $p=\tp(a/A)$ and $q=\tp(b/A)$ are so-types and $b\dep_A a$; in particular, $q\nfor p$. Set $a_0:=a$ and $b_0:=b$ and note $b_0\in\mathcal D_q(a_0)$. Choose orders $<_p$ and $<_q$ such that $\mathbf p=(p,<_p)$ and $\mathbf q=(q,<_q)$ are so-pairs over $A$. 
By Lemma \ref{Lemma_R_p_basic}(iv), $\mathcal D_q(a_0)$ is a non-empty bounded subset of $q(\Mon)$, and by taking $d_0\in\mathcal L_{\mathbf q}(a_0)$ and $d_1\in\mathcal R_{\mathbf q}(a_0)$ we have $d_0<_q\mathcal D_q(a_0)<_q d_1$. Let $f\in \Aut_A(\Mon)$ be such that $f(d_0)=d_1$; set $a_{n+1}:=f(a_n)$ and $b_{n+1}:=f(b_n)$ for $n=0,1$. Then $\mathcal D_q(a_0)<_q\mathcal D_q(a_1)<_q\mathcal D_q(a_2)$ and $b_i\in\mathcal D_q(a_i)$ for $i=0,1,2$. Note that the sequence $(a_0,a_1,a_2)$ is $<_p$-monotone, so by possibly reversing the order $<_p$, we may assume that it is $<_p$-increasing.

Since $\mathcal D_q(a_0)<_qb_1\in\mathcal D_q(a_1)$ we have $a_0\nequiv a_1\ (Ab_1)$, and since $\mathcal D_q(a_1)\ni b_1<_q\mathcal D_q(a_2)$ we have $a_1\nequiv a_2\ (Ab_1)$. Consider the locus $P$ of $\tp(a_1/Ab_1)$. Since $a_0<_pa_1$, $a_1\in P$, and $a_0\notin P$, $P$ is not an initial part of $(p(\Mon),<_p)$, so $a_1$ is not left $\mathbf p$-generic over $b_1$. Similarly, $a_1<_pa_2$, $a_1\in P$ and $a_2\notin P$ imply that $P$ is not a final part of $(p(\Mon),<_p)$, so $a_1$ is also not right $\mathbf p$ generic over $b_1$; by Lemma \ref{Lemma_R_p_basic2}(iii), we derive $a_1\dep_Ab_1$. Applying $f^{-1}$ we get $a\dep_Ab$, and we are done.
\end{proof}

\begin{Lemma}\label{Lema_Dpa_subset_Dqa}
Let $\mathbf p=(p,<)$ be an so-type and let $B\supseteq A$. 
\begin{enumerate}[(i)]
\item If $q=\tp(a/B)$ is a nonforking extension of $p$, then $\mathcal D_p(a)\subseteq q(\Mon)$.  
\item For all $a,a'\models p$: \ $a\ind_A B\land a'\dep_A a$ \ implies \ $a\equiv a'\ (B)$. 
\item For all $a,a'\models p$: \ $B\triangleleft^{\mathbf p} a \land a'\dep_A a$ implies $B\triangleleft^{\mathbf p} a'$.
\item If $p$ is weakly o-minimal and $q=\tp(a/B)$ is a nonforking extension of $p$, then $\mathcal D_p(a)\subseteq \mathcal D_q(a)$;
\end{enumerate}
\end{Lemma}
\begin{proof}
(i)  Since $q$ is a nonforking extension of $p$, by Corollary \ref{Corolary_of_forking_eq_bounded}(ii) we have $q=(\mathbf p_{l})_{\restriction B}$ or $q=(\mathbf p_{r})_{\restriction B}$; by reversing the order $<$ if necessary, we can assume $q=(\mathbf p_{r})_{\restriction B}$; hence, $B\triangleleft^{\mathbf p}a$. Next, we show that the set $\mathcal D_p(a)$ is lower bounded in $(q(\Mon),<)$. By the density property (Lemma \ref{Lemma_existence_anddensity_ofrightgenerics}), there exists $a'\models p$ such that $B\triangleleft^{\mathbf p}a'\triangleleft^{\mathbf p}a$. By Lemma \ref{Lemma_R_p_basic2}(ii), $a'\triangleleft^{\mathbf p}a$ implies that $a'$ is left $\mathbf p$-generic over $a$,  $a'\in\mathcal L_{\mathbf p}(a)$, which together with $\mathcal L_{\mathbf p}(a)<\mathcal D_{p}(a)$ (Lemma \ref{Lemma_R_p_basic}(iv)) implies $a'<\mathcal D_p(a)$. Since $B\triangleleft^{\mathbf p}a'$ implies $a'\models q$, we conclude that $a'\in q(\Mon)$ is a lower bound of $\mathcal D_p(a)$ in $(p(\Mon),<)$. As $q(\Mon)$ is a final part of $p(\Mon)$ and $\mathcal D_p(a)\subseteq p(\Mon)$,  $\mathcal D_p(a)\subseteq q(\Mon)$ follows.

(ii) Suppose $a\ind_A B$ and let $q=\tp(a/B)$. By (i), we have $\mathcal D_p(a)\subseteq q(\Mon)$, so every element $a'\in\mathcal D_p(a)$ realizes $q$, that is, $a\equiv a'\ (B)$. This proves (ii), which easily implies (iii). 

(iv)  Since $p$ is weakly o-minimal, $q$ is also weakly o-minimal, so $\mathcal D_q(a)$ is well defined. From the proof of (i), we know that the set $\mathcal D_p(a)$ is lower bounded in $(q(\Mon), <)$. It is easy to see that $\mathcal D_p(a)$ is upper bounded by any $a''$ that satisfies $a\triangleleft^{\mathbf q}a''$. So, for each $b\in\mathcal D_p(a)$, the locus of $\tp(b/Ba)$ is a bounded subset of $(q(\Mon),<)$; by Corollary \ref{Corolary_of_forking_eq_bounded}(iii) $\tp(b/Ba)$ is a forking extension of $q$ and $b\in\mathcal D_q(a)$ follows by the definition of $\mathcal D_q$. Therefore, $\mathcal D_p(a)\subseteq \mathcal D_q(a)$.
\end{proof}
 
Theorem \ref{Theorem4} is an easy corollary of the following proposition.  
\begin{Proposition}\label{Proposition_corT4}
Let $\mathbf{p} = (p, <_p)$ be an so-pair over $A$. Suppose that $p$ is non-algebraic.
\begin{enumerate}[(i)]
    \item The relation $x \dep_A y$ defines a convex equivalence relation, denoted by $\mathcal{D}_p$, on $(p(\Mon),<_p)$.
    \item The structure $(p(\Mon), \triangleleft^{\mathbf{p}})$ is a strict partial order where $\triangleleft^{\mathbf{p}}$-incomparability corresponds to the relation $\mathcal{D}_p$. Moreover, \ $\mathcal D_p(x)<_p\mathcal D_p(y)\Leftrightarrow x\triangleleft^{\mathbf p}y$.
    \item The orders $\triangleleft^{\mathbf{p}}$ and $<_p$ coincide on $p(\Mon) / \mathcal{D}_p$; the quotient $(p(\Mon) / \mathcal{D}_p, <_p)$ is a dense linear order.
\end{enumerate}
\end{Proposition}
\begin{proof}
(i) The reflexivity is clear and the symmetry follows from Lemma \ref{Lemma_forking_symmetry}. For transitivity, assume $a,b,c\models p$, $a\dep_A b$, and $b\dep_A c$.  If $a\ind_A c$ were true, then, by Lemma \ref{Lema_Dpa_subset_Dqa}(ii), $a\ind_A c$ and $a\dep_A b$ would imply  $a\equiv b\ (Ac)$,  which contradicts $b\dep_A c$. Therefore, $\mathcal D_p$ is an equivalence relation. To prove that it is convex, assume $a\dep_A b$ and $a<_pd<_pb$, and we prove $a\dep_A D$. Note that $a\triangleleft^{\mathbf{p}}d$ is impossible since, by Lemma \ref{Lemma_R_p_basic2}, $a\triangleleft^{\mathbf{p}}d<_pb$ implies $a\triangleleft^{\mathbf{p}}b$, which contradicts $a\dep_Ab$. Therefore, $d$ is not right $\mathbf p$-generic over $a$ and, since $a<_pd$ implies that it is not left $\mathbf p$-generic over $a$, we conclude $a\dep_A d$; $\mathcal D_p$ is convex.          

(ii) Clearly, $\triangleleft^{\mathbf p}$ is antireflexive and the transitivity follows by Lemma \ref{Lemma_R_p_basic2}(i).
Now, let $a,b\models p$. We apply Lemma \ref{Lemma_R_p_basic2} twice: by part (iii) we have $a\dep_A b\Leftrightarrow\lnot(a\triangleleft^{\mathbf p}b\vee a\triangleleft^{\mathbf p^*}b)$, and by part (ii), $a\triangleleft^{\mathbf p^*}b\Leftrightarrow b\triangleleft^{\mathbf p}a$.   Therefore, $a\dep_A b\Leftrightarrow \lnot(a\triangleleft^{\mathbf p}b\vee b\triangleleft^{\mathbf p}a)$, that is, if $a$ and $b$ are $\triangleleft^{\mathbf p}$-incomparable. Therefore, $\triangleleft^{\mathbf p}$-incomparability agrees with $\mathcal D_p$ on $p(\Mon)$. 
To prove the third claim, assume first that
$\mathcal D_p(x)<_p\mathcal D_p(y)$. Then $x\ind_A y$ and $x<_p y$ imply $x\triangleleft^{\mathbf p}y$.  
Conversely, if $x\triangleleft^{\mathbf p} y$ then $x\notin\mathcal D_p(y)$, so $x<_py$ and the convexity of $\mathcal D_p$ imply $x<_p\mathcal D_p(y)$; $x\triangleleft^{\mathbf p} y$ follows.

(iii) The coincidence of the orders follows from (ii), and the density from Lemma \ref{Lemma_existence_anddensity_ofrightgenerics}. 
\end{proof}

\begin{Remark}  Let $(p,<_p)$ be an so-pair and let $a_1,a_2\models p$. By now, we have found several equivalent ways to express that $(a_1,a_2)$ is a Morley sequence in $\mathbf p_r$ over $A$. Here we list them in groups:
\begin{itemize}
    \item $a_2\models (\mathbf p_r)_{\restriction Aa_1}$; \ $a_2\in\mathcal R_{\mathbf p}(a_1)$; \ $a_1\triangleleft^{\mathbf p} a_2$; \ $(a_1,a_2)\models (\mathbf p_r^2)_{\restriction A}$;  
    \item $a_1\models (\mathbf p_l)_{\restriction Aa_2}$; \ $a_1\in\mathcal L_{\mathbf p}(a_2)$; \ $a_2\triangleleft^{\mathbf p^*} a_1$; \ $(a_2,a_1)\models (\mathbf p_l^2)_{\restriction A}$  
     \ (Lemma \ref{Lemma_R_p_basic2}(ii)); 
    \item $a_1<_p a_2\land a_1\ind_A a_2$  \ (Lemma \ref{Lemma_R_p_basic2}(v)); 
    \item $\mathcal D_p(a_1)<_p  a_2 $; \ $a_1<_p \mathcal D_p(a_2)$; \ $\mathcal D_p(a_1)<_p \mathcal D_p(a_2)$. The first two are from Lemma \ref{Lemma_R_p_basic2}(v); the third follows from either of the first two since, by Theorem \ref{Theorem4},  $\mathcal D_p$ is a convex equivalence relation.   
\end{itemize}
\end{Remark}

In the case of a longer sequence, there is no simple way to express that it is Morley. Clearly, Morley sequences in $\mathbf p_r$ are $\triangleleft^{\mathbf p}$-increasing but in general, the opposite is not true. 

\begin{Lemma}
Let $\mathbf p=(p,<_p)$ be an so-pair over $A$ and let $I=(a_i\mid a\in\omega)$ be a sequence of realizations of $p$. Then $I$ is a Morley sequence in $\mathbf p_r$ over $A$ if and only if one of the following mutually equivalent conditions holds:

(1) \ \ $a_0\triangleleft^{\mathbf p} a_1$, \ $a_0a_1\triangleleft^{\mathbf p} a_2$, \dots , \ $a_0\ldots a_n\triangleleft^{\mathbf p} a_{n+1}$,...., that is, $a_{<n}\triangleleft^{\mathbf p} a_n$ holds for all $n$;

(2)  \ \ $I$ is $<_p$-increasing and  \ $a_{n+1}\ind_{Aa_{<n}}a_n$ \ holds for all $n$;  

(3) \ \ $I$ is $<_p$-increasing and $a_n\ind_A a_{<n}$ holds for all $n\geq 1$.
\end{Lemma}
\begin{proof}
Clearly, (1) expresses that $a_n\models (\mathbf p_r)_{\restriction Aa_{<n}}$ for all $n\geq 1$, that is, $I$ is Morley in $\mathbf p_r$ over $A$. 
To prove (1)$\Rightarrow$(2) note that if $I$ is Morley, then $(a_{n},a_{n+1})$ is Morley in $\mathbf p_r$ over $Aa_{<n}$, so $a_{n+1}\ind_{Aa_{<n}}a_n$ follows by Lemma \ref{Lemma_R_p_basic2}(v).  
(2)$\Rightarrow$(3) follows by induction.  
(3)$\Rightarrow$(1): By Lemma \ref{Lemma_R_p_basic2}(iii), $a_n\ind_A a_{<n}$ implies $a_{<n}\triangleleft^{\mathbf p}a_n\vee a_{<n}\triangleleft^{\mathbf p^*}a_n$; here, $a_{<n}\triangleleft^{\mathbf p^*}a_n$ is impossible as $I$ increases, so $a_{<n}\triangleleft^{\mathbf p}a_n$ holds, proving (3).

\end{proof}

\section{Orientation and nonorthogonality}\phantomsection\label{Section5} 

In this section, we first introduce the direct nonorthogonality of so-pairs and prove that it is an equivalence relation; we also prove that $\nfor,\nwor$ and $x\dep_Ay$ are equivalence relations. Further, we introduce $\mathcal F$-genericity and prove Theorem \ref{Theorem5}. The main technical properties of $\triangleleft^{\mathcal F}$-genericity are proved in Theorem \ref{Thm_triangle_mathcal F}. In the case of weakly o-minimal pairs, several additional properties are proved in subsection 5.1.

Let $\mathbf p=(p,<_p)$ and $\mathbf q=(q,<_q)$ be so-pairs with $p\nwor q$. We have chosen $a\triangleleft^{\mathbf q} b$ to describe that ``$b\in q(\Mon)$ is as far $<_q$-right of $a$ as possible''. The underlying intuition would be justified if ``$b$ is far $<_q$-right of $a$'' and ``$a$ is far $<_p$-left of $b$'' would be equivalent; that is, if $b\models q$ is right $\mathbf q$-generic over $a\models p$ and $a$ is left $\mathbf p$-generic over $b$ would be equivalent.  Note that option (I) in part (i) of the following lemma says just that, which motivates the definition of direct nonorthogonality of so-pairs.

\begin{Lemma}\label{Lemma_delta_options}
Suppose that $\mathbf p=(p,<_p)$ and $\mathbf q=(q,<_q)$ are so-pairs over $A$ and $p\nwor q$. 
\begin{enumerate}[(i)]
\item Exactly one of the following two conditions holds:
\begin{enumerate}[(I)] 
\item For all $a\models p$ and $b\models q$: $a\in \mathcal L_{\mathbf p}(b)\Leftrightarrow b\in \mathcal R_{\mathbf q}(a)$  \ and  \ $a\in \mathcal R_{\mathbf p}(b)\Leftrightarrow b\in \mathcal L_{\mathbf q}(a)$. 
\item For all $a\models p$ and $b\models q$: $a\in \mathcal L_{\mathbf p}(b)\Leftrightarrow b\in \mathcal L_{\mathbf q}(a)$  \ and  \ $a\in \mathcal R_{\mathbf p}(b)\Leftrightarrow b\in \mathcal R_{\mathbf q}(a)$.
\end{enumerate}
\item  (II) is equivalent to: $a\triangleleft^{\mathbf q}b \triangleleft^{\mathbf p}a$ for some $a\models p$ and $b\models q$. 
\item   (I) is equivalent to each of the following conditions:

-- for some $a\models p$ and $b\models q$ \  $a\in \mathcal R_{\mathbf p}(b)$ and  $b\in\mathcal L_{\mathbf q}(a)$; 

-- for some $a\models p$ and $b\models q$ \  $a\in \mathcal L_{\mathbf p}(b)$ and  $b\in\mathcal R_{\mathbf q}(a)$;

-- for all $a\models p$ and $b\models q$,  $a\in \mathcal R_{\mathbf p}(b)$ implies  $b\in\mathcal L_{\mathbf q}(a)$;

-- for all $a\models p$ and $b\models q$,  $a\in \mathcal L_{\mathbf p}(b)$ implies  $b\in\mathcal R_{\mathbf q}(a)$.
\end{enumerate}
\end{Lemma}
\begin{proof}
(i) Let $a\models p$ and $b\models q$.
By Lemma \ref{Lemma_R_p_basic}(ii), we know that $a\ind_A b$ is equivalent to $a\in \mathcal L_{\mathbf p}(b)\cup \mathcal R_{\mathbf p}(b)$. By independence symmetry, proved in Lemma \ref{Lemma_forking_symmetry}, another equivalent condition is $b\in \mathcal L_{\mathbf q}(a)\cup \mathcal R_{\mathbf q}(a)$. Therefore, for all $x\models q$, we have:
\setcounter{equation}{0}
\begin{equation}\tag{$*$}
  \mbox{$a\in \mathcal L_{\mathbf p}(x)\cup \mathcal R_{\mathbf p}(x)$ \ \ if and only if \ \ $x\in \mathcal L_{\mathbf q}(a)\cup \mathcal R_{\mathbf q}(a)$. } 
\end{equation}
By $A$-invariance of $\mathbf p_l$ and $\mathbf p_r$, each of $a\in \mathcal L_{\mathbf p}(x)\cup q(x)$ and $a\in\mathcal R_{\mathbf p}(x)\cup q(x)$ determines a completion of $q(x)$ in $S(Aa)$; denote these completions by $q_L$ and $q_R$, respectively, and note that they are different because of $p\nwor q$. Since
$x\in \mathcal L_{\mathbf q}(a)$ determines the type $(\mathbf q_l)_{\restriction Aa}(x)$ and since $x\in\mathcal R_{\mathbf q}(a)$ determines $(\mathbf q_r)_{\restriction Aa}(x)$, the equivalence in ($*$)  can be expressed by $\{q_L,q_R\}=\{(\mathbf q_l)_{\restriction Aa}, (\mathbf q_r)_{\restriction Aa}\}$. Here, we have two possibilities. 
The first is $q_R(x)=(\mathbf q_l)_{\restriction Aa}(x)$ and $q_L(x)=(\mathbf q_r)_{\restriction Aa}(x)$ (for all $a\models p$); this is equivalent to (I). The second, $q_L(x)=(\mathbf q_l)_{\restriction Aa}(x)$ and $q_R(x)=(\mathbf q_r)_{\restriction Aa}(x)$, corresponds to (II).

(ii)  Let $a\models p$ and $b\models q$ satisfy $a\triangleleft^{\mathbf q}b$ ($b\in \mathcal R_{\mathbf p}(a)$).  Then, since $p\nwor q$ implies $\mathcal L_{\mathbf p}(b)\neq \mathcal R_{\mathbf p}(b)$, $b\triangleleft^{\mathbf p}a$ ($a\in \mathcal R_{\mathbf p}(b)$) is inconsistent with (I) and is therefore equivalent to (II).

(iii) Each of the four conditions contradicts (II) and is therefore equivalent to (I).

\end{proof}

\begin{Definition}
Let $\mathbf p=(p,<_p)$ and $\mathbf q=(q,<_q)$ be so-pairs over $A$ and $p\nwor q$. The pairs $\mathbf p$ and $\mathbf q$   are {\em directly nonorthogonal}, denoted by $\delta_A(\mathbf p,\mathbf q)$, if option (I) of Lemma \ref{Lemma_delta_options} holds. That is,  whenever $a\models p$ is right $\mathbf p$-generic over $b\models q$, then $b$ is left $\mathbf q$-generic over $a$, and vice versa, whenever $b$ is right $\mathbf q$-generic over $a$, then $a$ is left $\mathbf p$-generic over $b$. 
\end{Definition}

\begin{Remark}\phantomsection\label{Remark_basic_delta}
\begin{enumerate}[(a)]
\item Immediately from the definition, we find that $\delta$ is symmetric: $\delta_A(\mathbf p,\mathbf q)$ iff $\delta_A(\mathbf q,\mathbf p)$.
    \item By Lemma \ref{Lemma_delta_options}(i), $p\nwor q$ implies that exactly one of $\delta_A(\mathbf p,\mathbf q)$ and $\delta_A(\mathbf p,\mathbf q^*)$ is true. 
\end{enumerate}
\end{Remark} 
 
The direct nonorthogonality of so-pairs extends the idea of orders (on a fixed domain) possesing a specific orientation, as defined in Definition \ref{Definition same orientation}. This is justified by the next lemma.  
\begin{Lemma}\label{Lemma_delta_pp}
Let $\mathbf p=(p,<_p)$ and $\mathfrak p=(p,<)$ be so-pairs over $A$. 
 $\delta_A(\mathfrak p, \mathbf p)$ if and only if $\mathfrak p_r=\mathbf p_r$  if and only if the orders $<_p$ and $<$ have the same orientation (on $p(\Mon)$) if and only if   $B\triangleleft^{\mathbf p} a \Leftrightarrow B\triangleleft^{\mathfrak p} a$ for all $B$ and all $a\models p$.
 
\end{Lemma}
\begin{proof} As $\mathbf p_r$ and $\mathfrak p_r$ ($\mathfrak p_r$ and $\mathfrak p_l$) are the only nonforking global extensions of $\mathbf p$ ($\mathfrak p$), we have $\{\mathbf p_l,\mathbf p_r\}=\{\mathfrak p_l,\mathfrak p_r\}$. There are two possible cases: 

(I) \ $\mathbf p_l=\mathfrak p_l \land \mathbf p_r=\mathfrak p_r$;  \ In this case, the left (right) $\mathbf p$- and $\mathfrak p$-genericity are of $a$ over $B$ are equivalent. In particular, $a$ is left (right) $\mathbf p$-generic over $b\models p$ if and only if $b$ is right (left) $\mathbf p$-generic over $a$;  so $\delta_A(\mathbf p,\mathfrak p)$. Also, from the proof of Lemma \ref{Lemma_same_orient}, we have that orders $<_p$ and $<$ have the same orientation.      

(II) \ $\mathbf p_l=\mathfrak p_r \land \mathbf p_r=\mathfrak p_l$; \ $<_p$ and $<$ have opposite orientations. If $a\models p$ is left $\mathbf p$-generic over $b\models p$, $a\triangleleft^{\mathbf p} b$, then $\mathbf p_l=\mathfrak p_r$ implies that $a$ is right $\mathfrak p$-generic over $b$; then $b$ is left $\mathfrak p$-generic over $a$. By 
Lemma \ref{Lemma_delta_options}(ii), $a\triangleleft^{\mathbf p}b\triangleleft^{\mathfrak p} a$ implies $\lnot\delta_A(\mathbf p,\mathfrak p)$.  

Therefore, case (I) corresponds to $\delta_A(\mathbf p,\mathfrak p)$  and case (II) to $\lnot\delta_A(\mathbf p,\mathfrak p)$.
\end{proof}
        
\begin{Lemma}\label{Lemma_delta_equivalents}
Let $\mathbf p=(p,<_p)$ and $\mathbf q=(q,<_q)$ be so-pairs over $A$ and $p\nwor q$. Then $\delta_A(\mathbf p,\mathbf q)$ is equivalent to each of the following conditions:
\begin{enumerate}[\hspace{10pt}(1)]
\item for all $a_1<_p a_2\models p$ \   $\mathcal R_{\mathbf q}(a_1)\supseteq  \mathcal R_{\mathbf q}(a_2)$; \ that is \  $a_1<_p a_2\triangleleft^{\mathbf q}x$ implies $a_1\triangleleft^{\mathbf q}x$;

\item for all $a_1,a_2\models p$: \  $a_1\triangleleft^{\mathbf p} a_2 \Leftrightarrow \mathcal R_{\mathbf q}(a_1)\supsetneq  \mathcal R_{\mathbf q}(a_2)$;

\item There do not exist $a\models p$ and $b\models q$ such that $a\triangleleft^{\mathbf q}b \triangleleft^{\mathbf p} a$.  
\end{enumerate}
\end{Lemma}
\begin{proof}
 $\delta_A(\mathbf p,\mathbf q)\Rightarrow$(1) Assume $\delta(\mathbf p,\mathbf q)$. Let $a_1<_p a_2$ realize $p$. Suppose $b\in \mathcal R_{\mathbf q}(a_2)$. Then $a_2\in \mathcal L_{\mathbf p}(b)$ holds by condition (I) of Lemma \ref{Lemma_delta_options}(i).
Combining with $a_1<_p a_2$ and the fact that $\mathcal L_{\mathbf p}(b)$ is an initial part of $(p(\Mon),<)$, we derive $a_1\in \mathcal L_{\mathbf p}(b)$; $b\in \mathcal R_{\mathbf q}(a_1)$ follows by condition (I). Therefore,  $b\in \mathcal R_{\mathbf q}(a_2)$ implies $b\in \mathcal R_{\mathbf q}(a_1)$, proving (1).

(1)$\Rightarrow$(2) Assume (1) and let $a_1,a_2\models p$. For the left-right direction of $a_1\triangleleft^{\mathbf p} a_2 \Leftrightarrow \mathcal R_{\mathbf q}(a_1)\supsetneq  \mathcal R_{\mathbf q}(a_2)$, assume $a_1\triangleleft^{\mathbf p}a_2$. 
Choose $a'\models p$ satisfying $\mathcal R_{\mathbf q}(a_1
)\supsetneq \mathcal R_{\mathbf q}(a')$; this is possible since $p\nwor q$ implies that $\mathcal R_{\mathbf q}(a_1)$ is a proper final part of $(p(\Mon),<_p)$. Next, choose $a_2'$ satisfying \ $a'a_1\triangleleft^{\mathbf p} a_2'$. In particular, $a_1\triangleleft^{\mathbf p}a_2'$ holds and hence $a_1a_2\equiv a_1a_2'\ (A)$. Thus, we can find $a\models p$ such that $a_1a_2a\equiv a_1a_2'a'\ (A)$; then $\mathcal R_{\mathbf q}(a_1)\supsetneq \mathcal R_{\mathbf q}(a)$ and $aa_1\triangleleft^{\mathbf p} a_2$ hold. Then $a\triangleleft^{\mathbf p}a_2$, by (1), implies $\mathcal R_{\mathbf q}(a)\supseteq \mathcal R_{\mathbf q}(a_2)$, which together with $\mathcal R_{\mathbf q}(a_1)\supsetneq \mathcal R_{\mathbf q}(a)$ implies $\mathcal R_{\mathbf q}(a_1)\supsetneq \mathcal R_{\mathbf q}(a_2)$, as desired.

To prove the right-to-left direction, assume (1) and $\mathcal R_{\mathbf q}(a_1)\supsetneq \mathcal R_{\mathbf q}(a_2)$; we need to prove $a_1\triangleleft^{\mathbf p}a_2$. Let $b\in \mathcal R_{\mathbf q}(a_1)\smallsetminus \mathcal R_{\mathbf q}(a_2)$; then $a_1\triangleleft^{\mathbf q} b$ and $a_1\nequiv a_2\ (Ab)$.
To prove $a_1\triangleleft^{\mathbf p}a_2$, it suffices to rule out $a_2<_p a_1$ and $a_2\dep_A a_1$. We do this by showing that each of them leads to $a_1\equiv a_2\,(Ab)$: for $a_2<_p a_1$, note that $a_2<_pa_1\triangleleft^{\mathbf  q}b$, by (1), implies $a_2\triangleleft^{\mathbf p}b$, which together with $a_1\triangleleft^{\mathbf  q}b$ implies $a_1\equiv a_2\,(Ab)$ implies; for $a_2\dep_Aa_1$, note that $a_1\ind_A b$ and $a_1\dep_A a_2$, by Lemma \ref{Lema_Dpa_subset_Dqa}(ii), imply $a_1\equiv a_2\,(Ab)$.    

\smallskip
(2)$\Rightarrow$(3) Assume (2). Toward a contradiction, suppose $a\triangleleft^{\mathbf q}b\triangleleft^{\mathbf p}a$; in particular, $b\in \mathcal R_{\mathbf q}(a)$. Choose $a'\models p$ such that $b\notin\mathcal R_{\mathbf q}(a')$; this is possible as $p\nwor q$ is valid. Then $b\in\mathcal R_{\mathbf q}(a)\smallsetminus \mathcal R_{\mathbf q}(a')$. Since $\mathcal R_{\mathbf q}(a)$ and $\mathcal R_{\mathbf q}(a')$ are $\subseteq$-comparable (as the final parts of $(p(\Mon),<_p)$), we conclude $\mathcal R_{\mathbf q}(a)\supsetneq \mathcal R_{\mathbf q}(a')$, so $a\triangleleft^{\mathbf p}a'$ follows by (2); in particular, $a<_pa'$. As $\mathcal R_{\mathbf p}(b)$ is a final part of $(\Mon)$, $b\triangleleft^{\mathbf p}a <_pa'$ implies $b\triangleleft^{\mathbf p}b$ and thus $a\equiv a'\ (Ab)$; this is in contradiction with $b\in\mathcal R_{\mathbf q}(a)\smallsetminus\mathcal R_{\mathbf q}(a')$.

(3)$\Leftrightarrow \delta_A(\mathbf p,\mathbf q)$ follows from Lemma \ref{Lemma_delta_options}.
\end{proof}

In part (i) of the lemma below, we prove a general form of transitivity of $\triangleleft$; further in the text, we will refer to it as the transitivity property.

\begin{Lemma}\label{Lema_basic_delta_transitivity}
Let $\mathbf p=(p,<_p)$ and $\mathbf q=(q,<_q)$ be so-pairs over $A$ with $\delta_A(\mathbf p,\mathbf q)$ 
\begin{enumerate}[(i)]
\item For all $a,b$ and $B$:   $B\triangleleft^{\mathbf p}a\land a\triangleleft^{\mathbf q}b$  implies $B\triangleleft^{\mathbf q} b$.
\item For all $a\models p,b\models q$ and all $B$:\   $B\triangleleft^{\mathbf p}a\land a\dep_Ab$  implies $B\triangleleft^{\mathbf q} b$.
\end{enumerate}
\end{Lemma}
\begin{proof}  
(i)  Suppose not. Then $b\in\mathcal R_{\mathbf q}(a)\smallsetminus \mathcal R_{\mathbf q}(B)$. Since $\mathcal R_{\mathbf q}(B)$ and $\mathcal R_{\mathbf q}(a)$ are the final parts of $(q(\Mon),<_q)$ we get $\mathcal R_{\mathbf q}(B)\subsetneq \mathcal R_{\mathbf q}(a)$. Let $b'\in\mathcal R_{\mathbf q}(B)$. Choose $a'\models p$ such that $b'\in\mathcal L_{\mathbf q}(a')$. Then $\mathcal L_{\mathbf q}(a')<_q\mathcal R_{\mathbf q}(a')$ and
$b'\in\mathcal L_{\mathbf q}(a')$ together imply
$b'<_q\mathcal R_{\mathbf q}(a')$, which combined with $b'\in \mathcal R_{\mathbf q}(B)$ gives
$\mathcal R_{\mathbf q}(a')\subsetneq\mathcal R_{\mathbf q}(B)$. Therefore,  $\mathcal R_{\mathbf q}(a')\subsetneq\mathcal R_{\mathbf q}(B)\subsetneq \mathcal R_{\mathbf q}(a)$ holds and, in particular, $a\nequiv a'\ (AB)$. By Lemma \ref{Lemma_delta_equivalents}(2),   $\mathcal R_{\mathbf q}(a)\supsetneq \mathcal R_{\mathbf q}(a')$ implies $a\triangleleft^{\mathbf p} a'$, which together with
$B\triangleleft^{\mathbf p} a$ implies $B\triangleleft^{\mathbf p} a'$. Consequently, $a,a'\models (\mathbf p_{r})_{\restriction AB}$, which is in contradiction with $a\nequiv a'\,(AB)$.

(ii) Assume that $B\triangleleft^{\mathbf p}a$, $b\models q$, and $a\dep_A b$. Then $b\in\mathcal D_q(a)$. By the density property of $\triangleleft^{\mathbf q}$, there exists $b'$ such that $B\triangleleft^{\mathbf q}b'\triangleleft^{\mathbf p} a$. Now, $\delta_A(\mathbf p,\mathbf q)$ implies that option (I) of Lemma \ref{Lemma_delta_options}(i) holds, so $b'\triangleleft^{\mathbf p}a$ implies $b'\in\mathcal L_{\mathbf q}(a)$ which combined with $b\in \mathcal D_q(a)$ yields $b'<_qb$. Since $\mathcal R_{\mathbf q}(B)$ is a final part of $q(\Mon)$, $B\triangleleft^{\mathbf q}b'<_q b$ implies $B\triangleleft^{\mathbf q}b$, completing the proof of (ii).  
\end{proof}

As we have already remarked, a so-type need not be NIP.   However, so-types share the following property of stationary, weight-one types in stable theories:

\begin{Proposition}\label{Prop_pqr_nwor}
Assume that $p,q\in S(A)$ are so-types.
\begin{enumerate}[(i)]
    \item If \ $p\nwor q$,  then  \  $p\wor r \Leftrightarrow q\wor r$ \  for all $r\in S(A)$.
    \item If \ $p\nfor q$,  then   \ $p\fwor r \Leftrightarrow q\fwor r$  \  for all $r\in S(A)$.
\end{enumerate} 
\end{Proposition}
\begin{proof} Let $\mathbf p=(p,<_p)$ and $\mathbf q=(q,<_q)$ be so-pairs over $A$ such that $\delta_A(\mathbf p,\mathbf q)$. 

(i) Assuming $p\nwor q$ and $p\nwor r$ we will prove $q\nwor r$. Let $C\models r$. As $p\nwor r$ implies that $\mathcal R_{\mathbf p}(C)$ is a proper final part of $p(\Mon)$, there is an $a\models p$ with $a<_p\mathcal R_p(C)$. 
Choose $a'\models p$ and $b,b'\models q$ such that: \begin{center} $b\triangleleft^{\mathbf p} a <_p\mathcal R_p(C)$ \ \ and \ \ $C\triangleleft^{\mathbf p} a'\triangleleft^{\mathbf q}b'$;
\end{center}
This is possible because of the existence of generic elements.
We will prove $b\nequiv b' \,(AC)$; clearly, this implies $q\nwor r$. 
Note that $b$ satisfies the following (infinitary) formula: $(\exists x\models p)(x<\mathcal R_p(C) \land b\triangleleft^{\mathbf  p}x)$. Hence, if $b\equiv b'\,(C)$ were valid, there would be a $AC$-automorphism moving $b$ to $b'$ and there would be an $a''\models p$ with $a''<_p\mathcal R_p(C)$ and $b'\triangleleft^{\mathbf p} a''$. Then $C\triangleleft^{\mathbf p} a'\triangleleft^{\mathbf q}b'\triangleleft^{\mathbf p}a''$ would imply $C\triangleleft^{\mathbf p}a''$ and hence $a''\in\mathcal R_p(C)$; contradiction. Therefore, $b\nequiv b' \,(AC)$, completing the proof.

\smallskip
(ii) Assuming $p\nfor q$ and $p\nfor r$ we will prove $q\nfor r$. Let $C\models r$. Choose $a\models p$ and  $b\models q$  such that $a\dep_Ab$ and $b\dep_AC$. We will prove $a\dep_AC$, which implies the desired conclusion. If $a\ind_AC$, then after possibly reversing the orders $<_p$ and $<_q$,  we may assume $C\triangleleft^{\mathbf p}a$. Then $C\triangleleft^{\mathbf p}a\land a\dep_Ab$ holds, so Lemma \ref{Lema_basic_delta_transitivity}(ii) implies   $C\triangleleft^{\mathbf q}b$, which contradicts the initial assumption $b\dep_AC$. Therefore, $a\dep_AC$.        
\end{proof}

\begin{Proposition}\label{Prop_triangle_mathbf_pq_properties}
Let $\mathbf p=(p,<_p)$ and $\mathbf q=(q,<_q)$ be so-pairs over $A$ with $\delta_A(\mathbf p,\mathbf q)$. Then for all $B$, $a\models p$ and $b\models q$:
\begin{enumerate}[ (i)]
\item (Existence) \ There exists a $a'\models p$ such that $B\triangleleft^{\mathbf p}a'$;
\item (Density) \ If $B\triangleleft^{\mathbf q} b$, then there is a $a'\models p$ such that $B\triangleleft^{\mathbf p} a'\triangleleft^{\mathbf q} b$; 
\item (Symmetry) \ $a\ind_A b$ if and only if $b\ind_A a$; 
\item   \   
$a \triangleleft^{\mathbf q} b\Leftrightarrow b\triangleleft^{\mathbf p^*} a$ \ ($b$ is right $\mathbf q$-generic over $a$ if and only if $a$ is left $\mathbf p$-generic over $b$.) 
\item  ($\triangleleft^{\mathbf p}$-comparability) \ $a\ind_A B\Leftrightarrow (B \triangleleft^{\mathbf p} a\vee B\triangleleft^{\mathbf p^*} a)$
\ and \ 
$a\ind_A b\Leftrightarrow (a \triangleleft^{\mathbf q} b\vee b\triangleleft^{\mathbf p} a)$;
\item  (Transitivity) \ $B\triangleleft^{\mathbf p} a \triangleleft^{\mathbf q} b$ \ implies \ $B\triangleleft^{\mathbf q} b$. 
\item   \  $B\triangleleft^{\mathbf p}a\land a\dep_A b$ implies $B\triangleleft^{\mathbf q}b$; that is, $B\triangleleft^{\mathbf p}a$ implies $B\triangleleft^{\mathbf q} x$ \  for all $x\in \mathcal D_{\mathbf q}(a)$.
\item  \ $a\triangleleft^{\mathbf q} b$ \  if and only if  $x \triangleleft^{\mathbf q}y$ for all $x\in \mathcal D_{\mathbf p}(a)$ and $y\in \mathcal D_{\mathbf q}(b)$. 
\end{enumerate}
\end{Proposition}
\begin{proof}
(i)-(vii) follow easily from previous results; for example, (vi)-(vii) follow from Lemma \ref{Lema_basic_delta_transitivity}. To prove (viii), assume $a'\in \mathcal D_{p}(a)$, and $b'\in\mathcal D_{q}(b)$. Then:
$$(a\triangleleft^{\mathbf q} b) \Leftrightarrow (a\triangleleft^{\mathbf q}b')\Leftrightarrow (b'\triangleleft^{\mathbf p^*}a)\Leftrightarrow (b'\triangleleft^{\mathbf p^*}a') \Leftrightarrow (a'\triangleleft^{\mathbf q}b')\ ;$$
where the first and third equivalences are consequences of (vii), and the second and fourth follow by (iv). This proves (viii).
\end{proof}

Recall that the pairs $(p,<_p)$ and $(q,<_q)$ are weakly nonorthogonal if  $p\nwor q$.  
\begin{Theorem}\label{Theorem_nonorthogonality}
Let $\mathcal{S}_A$ denote the set of all so-types over $A$ and  $\mathcal P_A$ the set of all so-pairs over $A$. Let $\mathcal{S}_A(\Mon)$  be the set of all realizations of types from $\mathcal{S}_A$.
\begin{enumerate}[(i)]
    \item $\nwor$ is an equivalence relation on both $\mathcal{S}_A$ and $\mathcal P_A$. \
    \item  $\delta_A$ is an equivalence relation on $\mathcal P_A$; $\delta_A$ refines  $\nwor$  by splitting each class into two classes, each of them consisting of the reverses of pairs from the other class.
    \item $x\dep_A y$ is an equivalence relation on $\mathcal{S}_A(\Mon)$. 
    \item $\nfor$ is an equivalence relation on $\mathcal{S}_A$.
\end{enumerate}
\end{Theorem}
\begin{proof} 
Clearly, both $\delta_A$ and $\nwor$ are reflexive and symmetric.

(i) The transitivity of $\nwor$ is a consequence of Proposition \ref{Prop_pqr_nwor}.   

\smallskip 
(ii) To prove the transitivity of $\delta_A$, assume that $\mathbf p=(p,<_p)$, $\mathbf q=(q,<_q)$ and $\mathbf r=(r,<_r)$ are so-pairs such that $\delta_A(\mathbf p,\mathbf q)$ and $\delta_A(\mathbf q,\mathbf r)$. By part (i), $p\nwor q$ and $q\nwor r$ 
 imply $p\nwor r$. To prove $\delta_A(\mathbf p,\mathbf r)$, we will verify the equivalent condition (1) of Lemma \ref{Lemma_delta_equivalents}: Assuming $a_1<_pa_2\triangleleft^{\mathbf r}c$, we need to prove $a_1\triangleleft^{\mathbf r}c$. By the density property, there exists a $b\models q$ that satisfies $a_2\triangleleft^{\mathbf q}b\triangleleft^{\mathbf r}c$. Then $a_1<_pa_2\triangleleft^{\mathbf q}b$,  by condition (1) from Lemma \ref{Lemma_delta_equivalents}, implies $a_1\triangleleft^{\mathbf q} b$. By Lemma \ref{Lema_basic_delta_transitivity}(i), $a_1\triangleleft^{\mathbf q}b\triangleleft^{\mathbf r}c$ implies $a_1\triangleleft^{\mathbf r}c$, as desired.  Therefore, $\delta_A$ is an equivalence relation. To show that there are exactly two $\delta_A$ classes, recall Remark \ref{Remark_basic_delta}(b): for all $\mathbf p, \mathbf q$ over $A$, $p\nwor q$ implies $\delta_A(\mathbf p,\mathbf q)\vee \delta_A(\mathbf p,\mathbf q^*)$. Hence, if $p$ is non-algebraic, then the $\nwor$-class of $\mathbf p$ is split into two $\delta_A$-classes.   

\smallskip
(iii) By Lemma \ref{Lemma_forking_symmetry}, $\dep_A$ is symmetric and the reflexivity is clear. 
To prove transitivity, assume $a\dep_A b$ and $a\ind_A  c$, and we prove $b\ind_A c$.  Let $p=\tp(a/A),q=\tp(b/A)$ and $r=\tp(c/A)$; these are so-types, and note that $a\dep_A b$ implies $p\nwor q$. Thus, if $p\wor r$, then $q\wor r$ holds by (i); $b\ind_A c$ follows, and we are done.
So, assume $p\nwor r$. By (i), $p,q$, and $r$ are in the same $\nwor$-class, so we can choose 
orders $<_p,<_q$ and $<_r$ such that the corresponding so-pairs $\mathbf p,\mathbf q$ and $\mathbf r$ are in the same $\delta_A$-class. 
Then $a\ind_A c$ implies that $a$ is left- or right $\mathbf p$-generic over $c$; without loss (by reversing all three orders if necessary) assume $c\triangleleft^{\mathbf p}a$. 
Thus, we have: \ $c\triangleleft^{\mathbf p}a\land a\dep_A b$ which, by Lemma \ref{Lema_basic_delta_transitivity}(ii) implies $c\triangleleft^{\mathbf q}b$ and therefore $b\ind_Ac$, as needed.  

\smallskip
(iv) Follows from (iii). 
\end{proof}

 Let $\mathcal F$ be a set of pairwise directly non-orthogonal so-pairs over $A$. By Theorem \ref{Theorem_nonorthogonality}, $\mathcal F$ is part of a $\delta_A$-class of so-pairs over $A$.  We will use the following notation: 
$$\mathcal{F}(\Mon)=\bigcup\{p(\Mon)\mid (p,<_p)\in\mathcal F\mbox{ for some $<_p$}\}  \ \  \ \ \mbox{and} \ \ \ \ \mathcal D_{\mathcal F}=\{(a,b)\in\mathcal{F}(\Mon)\times\mathcal{F}(\Mon)\mid a\dep_A b\} ;$$    
Here, $\mathcal{F}(\Mon)$ is the set of realizations of all types (appearing in some pair) from $\mathcal F$ \ and \ $\mathcal D_{\mathcal F}$ \  is the equivalence relation of forking on $\mathcal F(\Mon)$. Also, define $\mathcal F^*$, called the reverse of $\mathcal F$, by $\mathcal F^*=\{\mathbf p^*\mid \mathbf p\in\mathcal F\}$.  Note that, by Theorem \ref{Theorem_nonorthogonality}(ii), $\mathcal F^*$ is also a part of a $\delta_A$-class.     
In particular, if $\mathcal F$ is the whole $\delta_A$-class, then the union $\mathcal F\cup\mathcal F^*$ is a $\nwor$-class of so-pairs from $\mathcal P_A$.

\begin{Definition} Let $\mathcal F$ be a part of a $\delta_A$-class of so-pairs over $A$ and let $a\in \mathcal{F}(\Mon)$. We say that {\em $a$ is right $\mathcal F$-generic over $B$}, denoted by  $B\triangleleft^{\mathcal F}a$,  if   $B\triangleleft^{\mathbf p}a$ holds for some (equivalently, any\footnote{By Lemma \ref{Lemma_delta_pp}}) pair $\mathbf p=(p,<_p)\in \mathcal F$ such that $a\models p$. 
\end{Definition}

In the next theorem, we collect the basic properties of the relation $\triangleleft^{\mathcal F}$ for so-types.  

\begin{Theorem}\label{Thm_triangle_mathcal F}
Let $\mathcal F$ be part of a $\delta_A$-class of so-pairs over $A$. Then for all small $B$, all $\mathbf p\in\mathcal F$, and all $a,b\in \mathcal{F}(\Mon)$ the following hold true: 
\begin{enumerate}[ (i)]
\item (Existence) \ There exists a $c\models p$ such that $B\triangleleft^{\mathcal F}c$;
\item (Density) \ If $B\triangleleft^{\mathcal F} b$, then there is a $c$ such that $B\triangleleft^{\mathbf p} c\triangleleft^{\mathcal F} b$; 
\item (Symmetry) \ $a\ind_A b$ if and only if $b\ind_A a$; 
\item   \   
$a \triangleleft^{\mathcal F} b\Leftrightarrow b\triangleleft^{\mathcal F^*} a$ \ ($b$ is right $\mathcal F$-generic over $a$ if and only if $a$ is left $\mathcal F$-generic over $b$.) 
\item  ($\triangleleft^{\mathcal F}$-comparability) \ $a\ind_A B\Leftrightarrow (B \triangleleft^{\mathcal F} a\vee B\triangleleft^{\mathcal F^*} a)$
\ and \ 
$a\ind_A b\Leftrightarrow (a \triangleleft^{\mathcal F} b\vee b\triangleleft^{\mathcal F} a)$;
\item  (Transitivity) \ $B\triangleleft^{\mathcal F} a \triangleleft^{\mathcal F} b$ \ implies \ $B\triangleleft^{\mathcal F} b$. 
\item   \  $B\triangleleft^{\mathcal F}a\land a\dep_A b$ implies $B\triangleleft^{\mathcal F}b$; that is, $B\triangleleft^{\mathcal F}a$ implies $B\triangleleft^{\mathcal F} x$ \  for all $x\in \mathcal D_{\mathcal F}(a)$.
\item  \ $a\triangleleft^{\mathcal F} b$ \  if and only if  $x \triangleleft^{\mathcal F}y$ for all $x\in \mathcal D_{\mathcal F}(a)$ and $y\in \mathcal D_{\mathcal F}(b)$. 
\end{enumerate}
\end{Theorem}
\begin{proof}
Follows from Proposition \ref{Prop_triangle_mathbf_pq_properties}.
\end{proof}

\begin{proof}[{Proof of Theorem \ref{Theorem5}.}] Let $\mathcal F$ be a $\delta_A$-class and let $\mathbf p=(p,<_p)\in\mathcal F$. By part Theorem \ref{Thm_triangle_mathcal F}(v), it follows that $(\mathcal{F}(\Mon), \triangleleft^{\mathcal F})$ is a strict partial order. By the definition of $\triangleleft^{\mathcal F}$, the restriction of this order to $p(\Mon)$ is the order $(p(\Mon),\triangleleft^{\mathbf p})$, which has a linear extension $(p(\Mon),<_p)$. 
\begin{itemize}
    \item Let $(p_i\mid i<\alpha)$ be an enumeration of all (pairwise distinct) types $p\in S(A)$ appearing in $\mathcal F$;
    \item For each $i<\alpha$ choose $<_i$ such that $(p_i,<_i)\in\mathcal F$. 
\end{itemize}
In order to linearize $\triangleleft^{\mathcal F}$, by Theorem  \ref{Thm_triangle_mathcal F}(viii), it suffices to linearize each class $\mathcal D_{\mathcal F}(a)$ ($a\in\mathcal F(\Mon)$). This is easy to do in an $A$-invariant way. For example:
\begin{center}
$ x<^{\mathcal F}y$ \ \ if and only if \ \ $x\triangleleft^{\mathcal F}y $ \ or \ ($x\dep_A y$,  $x\models p_i, y\models p_j$ and $i<j$).
\end{center}
It is straightforward to verify that the order $<^{\mathcal F}$ satisfies the conclusion of Theorem \ref{Theorem5}.
\end{proof}

\subsection{Orientation of extensions of weakly o-minimal pairs} \

Throughout this subsection, let $\mathcal F$ be a part of a $\delta_A$-class of weakly o-minimal pairs over $A$; as weakly o-minimal types are so-types, the conclusion of Theorem \ref{Theorem5} still holds, so let $<^{\mathcal F}$ be an $A$-invariant order on $\mathcal F(\Mon)$ such that (i)-(iii) of Theorem \ref{Theorem5} are satisfied.  
Since complete extensions of weakly o-minimal types are also weakly o-minimal, it makes sense to consider nonforking extensions of types (pairs) from $\mathcal F$ over a larger domain $B\supseteq A$. For each $\mathbf p=(p,<_p)\in\mathcal F$ define: 
\begin{center} 
$p_B=(\mathbf p_r)_{\restriction B}$, \ $\mathbf p_B=(p_B,<_p)$, \ $\mathcal F_B=\{\mathbf p_B\mid \mathbf p\in\mathcal F\}$, \ and \ $\mathcal F_B(\Mon)= \bigcup_{\mathbf p\in\mathcal F}p_B(\Mon)$.
\end{center}

\begin{Proposition}\label{Prop_delta_B_propoerties}  For all $\mathbf p,\mathbf q\in\mathcal F$ and  all $a,b,c\in\mathcal{F}(\Mon)$: 
\begin{enumerate}[(i)]
\item  For all $C$: \ $C\triangleleft^{\mathbf p_B} a  \Leftrightarrow \ BC\triangleleft^{\mathbf p}a$. \ In particular, $b\triangleleft^{\mathbf p_B}a$ implies $b\triangleleft^{\mathbf p}a$ for all $b\models q_B$;
\item $\delta_B(\mathbf p_B,\mathbf q_B)$ holds;
\item  $B\triangleleft^{\mathcal F} a \land a\dep_A b$ implies $B\triangleleft^{\mathcal F} b\land a\dep_{B}b$ \ \ ($ a\in \mathcal F_B(\Mon) \land a\dep_A b$ implies $b\in\mathcal F_B(\Mon)\land a\dep_{B}b$);
\item   $B\triangleleft^{\mathcal F} a<^{\mathcal F}b$ implies $B\triangleleft^{\mathcal F} b$ \ \ ($a\in\mathcal F_B(\Mon) \land a<^{\mathcal F}b$ implies $b\in\mathcal F_B(\Mon)$);  
\item   $B\triangleleft^{\mathcal F} a<^{\mathcal F}b\land  \, a\ind_B b$ implies $Ba\triangleleft^{\mathcal F} b$; 
\item   $B\triangleleft^{\mathcal F}a<^{\mathcal F}b<^{\mathcal F}c\land a\dep_{B}c$ implies  $b\dep_{B} a \land b\dep_{B}c$; 
\end{enumerate}
\end{Proposition}
\begin{proof} (i) Clearly: \   $C\triangleleft^{\mathbf p_B} a \ \Leftrightarrow \  a\models (\mathbf p_r)_{\restriction BC} \ \Leftrightarrow \ BC\triangleleft^{\mathbf p}a$. \ In particular, if $b\models q_B$ and  $b\triangleleft^{\mathbf p_B}a$, then  $Bb\triangleleft^{\mathbf p}a$ and hence $b\triangleleft^{\mathbf p}a$.

(ii) By Proposition \ref{Prop_pqr_nwor}(i) we have $p_B\nwor q_B$.  
To prove $\delta_B(\mathbf p_B,\mathbf q_B)$, by Lemma \ref{Lemma_delta_equivalents} it suffices to show that $a'\triangleleft^{\mathbf q_B}b'\triangleleft^{\mathbf p_B} a'$ cannot hold for some $a'\models p_B$ and $b'\models q_B$. If it did, then $a'\triangleleft^{\mathbf q_B}b'\triangleleft^{\mathbf p_B} a'$, by part (i), would lead to $a'\triangleleft^{\mathbf q}b'\triangleleft^{\mathbf p} a'$, which would contradict $\delta_A(\mathbf p,\mathbf q)$.

(iii)  Suppose that $B\triangleleft^{\mathcal F} a \land a\dep_A b$. Choose $\mathbf p=(p,<_p),\mathbf q=(q,<_q)\in\mathcal F$ such that $p=\tp(a/A)$, $q=\tp(b/A)$. By Theorem \ref{Thm_triangle_mathcal F}(vii), $B\triangleleft^{\mathcal F} a$  and  $a\dep_A b$ imply   $B\triangleleft^{\mathcal F}b$ ($b\models q_B$), proving the first part of the claim.  To prove the second, $a\dep_B b$, first note that $a\models p_B$ and $b\models q_B$, so if $a\ind_B b$ were true, then we would have $a\triangleleft^{\mathbf q_B}b$ or $b\triangleleft^{\mathbf p_B}a$. By (i), this would imply $a\triangleleft^{\mathbf q}b$ or $b\triangleleft^{\mathbf p}a$, which would contradict $a\dep_Ab$. Therefore, $a\dep_Bb$. 

(iv) Suppose that $B\triangleleft^{\mathcal F} a<^{\mathcal F} b$. By Theorem \ref{Theorem5}(i)-(ii) $a<^{\mathcal F} b$ implies $(a,b)\in\mathcal D_{\mathcal F}$ or $a\triangleleft^{\mathcal F} b$. In the first case, $B\triangleleft^{\mathcal F}b$ follows by Theorem \ref{Thm_triangle_mathcal F}(vii), and in the second by the transitivity property of $\triangleleft^{\mathcal F}$. 

(v) Suppose that $B\triangleleft^{\mathcal F}a<^{\mathcal F}b$ and $a\ind_Bb$. Denote $p=\tp(a/A)$, $q=\tp(b/A)$ and choose $\mathbf p=(p,<_p)$,  $\mathbf q=(q,<_q)\in \mathcal F$.  We have: $a\models p_B$, $b\models q_B$ (by (iv)), and  $\delta_B(\mathbf p_B,\mathbf q_B)$ (by (ii)). Now, $a\ind_B b$ implies $a\triangleleft^{\mathbf q_B} b$ or $b\triangleleft^{\mathbf p_B} a$. By (i),  $b\triangleleft^{\mathbf p_B}a$ implies   $b\triangleleft^{\mathcal F} a$, which contradicts  $a<^{\mathcal F} b$. Therefore, $a\triangleleft^{\mathbf q_B} b$ holds, so $b\models \mathbf q_{Ba}$, that is, $Ba\triangleleft^{\mathcal F} b$. 

(vi) Assume that $B\triangleleft^{\mathcal F}a<^{\mathcal F}b<^{\mathcal F}c$ and $a\dep_B c$. We will prove $a\dep_B b$; then $b\dep_Bc$ follows by the transitivity of $\dep_B$.    
If $a\dep_B b$ does not hold, then according  to (v),  $a\ind_B b$ and $B\triangleleft^{\mathcal F}a<^{\mathcal F}b$ would imply $Ba\triangleleft^{\mathcal F}b$. This together with $b\triangleleft^{\mathcal F}c$ would lead to $Ba\triangleleft^{\mathcal F}c$, contradicting $a\dep_Bc$. Thus, we must have $a\dep_B b$.
\end{proof}

\begin{Corollary}\label{Prop_4.21} Let \ $<^{\mathcal F_B}:= <^{\mathcal F}{\restriction \mathcal F_B(\Mon)}$ \ and \ $\mathcal D_{\mathcal F_B}:=\{(x,y)\in\mathcal F_B(\Mon)^2\mid x\dep_B y\}$. 
\begin{enumerate}[(i)] 
\item  $\mathcal F_B$ is a part of some $\delta_B$-class and  $\mathcal F_B(\Mon)$ is a final part of $(\mathcal{F}(\Mon),<^{\mathcal F })$; 
\item $\mathcal D_{\mathcal F_B}$ is a convex equivalence relation on $(\mathcal F_B(\Mon),<^{\mathcal F_B})$ that is coarser than   $\mathcal D_{\mathcal F}\restriction \mathcal F_B(\Mon)$; 
\item The structure $(\mathcal F_B(\Mon),\triangleleft^{\mathcal F_B},<^{\mathcal F_B},\mathcal D_{\mathcal F_B})$ satisfies the conclusion of Theorem \ref{Theorem5}. 
\end{enumerate} 
\end{Corollary}
\begin{proof} 
(i) $\mathcal F_B$ is a part of some $\delta_B$-class by Proposition \ref{Prop_delta_B_propoerties}(ii);   $\mathcal F_B(\Mon)$ is a final part of $(\mathcal{F}(\Mon),<^{\mathcal F_B })$ by Proposition \ref{Prop_delta_B_propoerties}(iii).

(ii) $\mathcal D_{\mathcal F_B}$ is convex by Proposition \ref{Prop_delta_B_propoerties}(vii); it is coarser than $\mathcal D_{\mathcal F}\restriction \mathcal F_B(\Mon)$ by Proposition \ref{Prop_delta_B_propoerties}(iii).

(iii) Straightforward.
\end{proof}

Now we drop the initial assumptions and formulate the next corollary in a general context.
\begin{Corollary}
Suppose that $\mathbf p,\mathbf q$ are weakly o-minimal pairs over $A$ with $\delta_A(\mathbf p,\mathbf q)$. Let $B\supseteq A$. Define $p_B=(\mathbf p_r)_{\restriction B}$, $\mathbf p_B=(p_B,<_p)$ and analogously define $q_B$ and $\mathbf q_B$.   
    \begin{enumerate}[(i)]
\item $\delta_B(\mathbf p_B,\mathbf q_B)$ holds;  in particular, $p_B\nwor q_B$ and $\mathbf p_r\nwor \mathbf q_r$;
\item If  $p\nfor q$, then $p_B\nfor q_B$ and, in particular, $\mathbf p_r\nfor\mathbf q_r$.  
\end{enumerate}
\end{Corollary}
\begin{proof}
(i) follows from Proposition \ref{Prop_delta_B_propoerties}(ii). To prove (ii), assume that $p\nfor q$ and let $a\models p_B$. Choose $b\models q$ with $a\dep_Ab$. Then $B\triangleleft^{\mathbf p} a$ and $a\dep_A b$, by Proposition \ref{Prop_delta_B_propoerties}(ii), imply $B\triangleleft^{\mathbf q}b$ and $a\dep_B b$. Hence $a\models p_B,b\models q_B$ and $a\dep_B b$; $p_B\nfor q_B$ follows.   
\end{proof}

\section{dp-minimality}\phantomsection\label{Section6}

In this section, we prove that weakly o-minimal types are dp-minimal and that their indiscernible sequences enjoy some nice properties. For example, in Proposition \ref{Proposition_replace_a_i_by_ai'} below, we prove that every Morley sequence $I$ of realizations of a weakly o-minimal type $p$ remains Morley after replacing every element $a\in I$ with an element $a'\in\mathcal D_p(a)$.

We will very briefly recall the notions of NIP and dp-minimality for (partial) types, without recalling the original definition of dp-minimal types, as we will not use it in this paper; a detailed exposition of NIP can be found in Simon's book \cite{Simon}. Rather, we recall two characterizations of NIP types (see \cite[Claim 2.1]{KS}), since we will use them later. We also recall an equivalent characterization of dp-minimality due to Kaplan, Onshuus, and Usvyatsov (see \cite[Proposition 2.8]{KOU}).

\begin{Definition} Let $p( x)$ be a (partial) type over $A$. The type $p(x)$ is:
\begin{enumerate}[(a)]
\item {\em NIP} if there does {\em not} exist a formula $\varphi( x, y)$, an $A$-indiscernible sequence $( b_n\mid n<\omega)$ of tuples of length $|y|$ and $ a\models p$ such that $\models\varphi(a,b_n)$ if and only if $n$ is even.
Equivalently, there is no formula $\varphi(x, y)$, an $A$-indiscernible sequence $( a_n\mid n<\omega)$ in $p$ and a tuple $ b$, such that $\models\varphi( a_n, b)$ iff $n$ is even;

\item {\em dp-minimal} if there does {\em not} exist a formula $\varphi( x, y)$, an infinite $A$-indiscernible sequence $(b_i\mid i\in I)$ of tuples of length $|y|$, and $a\models p$ such that the truth value $\varphi(a,y)$ has at least four alternations on $(b_i\mid i\in I)$, that is, there are $i_0<i_1<i_2<i_3<i_4$ such that 
$\models\varphi(a,b_{i_j})\not\leftrightarrow\varphi(c,b_{i_{j+1}})$ for all $j<4$. 
\end{enumerate}
\end{Definition}

\begin{Fact}\phantomsection\label{fact nip dp-min wom}
\begin{enumerate}[(i)]
\item Every dp-minimal type is NIP.
\item Every weakly o-minimal type is NIP.
\item If $p( x)$ over $A$ is NIP and $ a_i\models p$ for $i<n$, then $\tp( a_0,\dots, a_{n-1}/A)$ is NIP.
\end{enumerate}  
\end{Fact}
\begin{proof}
(i) is obvious by the first characterization from the definition of NIP types. (ii) If $(p,<)$ is a weakly o-minimal pair over $A$, then every $A$-indiscernible sequence of realizations of $p$ is monotone with respect to $<$, so weak o-minimality implies that the second characterization of NIP types above is satisfied. (iii) follows from \cite[Theorem 4.11]{KOU}.
\end{proof}

Goodrick in \cite[Proposition 4.9]{Goodrick2} proved that every weakly quasi-o-minimal theory is dp-minimal. We will prove that every partial type whose all completions are weakly o-minimal is dp-minimal and, as a corollary, derive Goodrick's result. In the proof, we will use the following simple, Helly-style fact.

\begin{Fact}\label{fact combinatorics for dp-min}
Let $(X,<)$ be a linear order and $\{S_n\mid n<\omega\}$ a family of non-empty subsets of $X$ such that there is $N<\omega$ such that each $S_n$ has at most $N$ convex components. Suppose that $\{S_n\mid n<\omega\}$ has the $2$-intersection property. Then there is an infinite $I\subseteq\omega$ such that $\{S_n\mid n\in I\}$ has the finite intersection property.  
\end{Fact}
\begin{proof}  
Without loss we may assume that each $S_n$ has $N$ convex components, and write $S_n=C_n^1\cupdot \dots \cupdot C_n^N$ with $C_n^1<\dots<C_n^N$ being convex. For $n<m$ choose minimal $i$ and minimal $j$ such that $C_n^{i}\cap C_m^{j}\neq\emptyset$. By Ramsey's theorem, there is an infinite $I\subseteq\omega$  such that the chosen pairs $(i,j)$ are the same for all $n<m$ in $I$. If $i=j$ then $\{C_n^i\mid n\in I\}$ is a $2$-consistent family of convex sets, so it is $k$-consistent for all $k\geqslant 2$ by Helly's theorem, and hence $\{S_n\mid n\in I\}$ is $k$-consistent for all $k\geqslant 2$, too. 

Consider the case $i<j$. If $n<m$ in $I$ then $C_m^i<C_n^i$ as $C_m^i<C_m^j$ and $j$ is minimal such that $C_n^i\cap C_m^j\neq\emptyset$. Write $I=\{n_0,n_1,n_2,\dots\}$ in increasing order. Then $C_{n_2}^i<C_{n_1}^i<C_{n_0}^i$. For each $l\geqslant 3$, the set $C_{n_l}^j$ meets both $C_{n_2}^i$ and $C_{n_0}^i$, so by convexity it completely contains $C_{n_1}^i$. Therefore, $\{S_{n_l}\mid l\geqslant 3\}$ contains $C_{n_1}^i$ at its intersection, so it is $k$-consistent for all $k\geqslant 2$. The proof in the case $i>j$ is similar.
\end{proof}

\begin{Proposition}\label{proposition wom is dp minimal}
Suppose that $p(x)$ is a partial type over $A$ whose completions over $A$ are all weakly o-minimal. Then $p(x)$ is dp-minimal. 
\end{Proposition}
\begin{proof}Suppose not. Let $\varphi(x,y)$ be a formula, $I=(b_n\mid n<\omega)$ a $A$-indiscernible sequence, and $c$ realizing $p$ such that $\varphi(c,y)$ has at least three alternations on $I$. Let $i_0<i_1<i_2<i_3$ be such that:
\begin{equation}\tag{$*$}
\models\varphi(c,b_{i_0})\land\lnot\varphi(c,b_{i_1}) \land \varphi(c,b_{i_2})\land\lnot\varphi(c,b_{i_3}).
\end{equation}
Put $q=\tp(c/A)$. By assumption, $q$ is a weakly o-minimal type, so $(q,<_q)$ is a weakly o-minimal pair over $A$ for some (any) relatively $A$-definable order $<_q$ on $q(\Mon)$.
Consider the relatively definable subsets $S_n$ of $q(\Mon)$ relatively defined by $\varphi(x,b_{2n})\land\lnot\varphi(x, b_{2n+1})$. By weak o-minimality of $q$ and indiscernibility of $I$, for some $N<\omega$ each $S_n$ has $N$ $<_q$-convex components. Furthermore, by the indiscernibility of $I$ and ($*$) we see that $\{S_n\mid n<\omega\}$ is 2-consistent. Then by Fact \ref{fact combinatorics for dp-min} there is an infinite $J\subseteq \omega$ such that $\{S_n\mid n\in J\}$ has the finite intersection property; in fact, by the indiscernibility of $I$, we may conclude that $\{S_n\mid n<\omega\}$ has the finite intersection property. By saturation we find $a\in\bigcap_{n<\omega}S_n$, but this contradicts the fact that $p$ is NIP (Fact \ref{fact nip dp-min wom}(ii)).
\end{proof}

\begin{Corollary}
Every weakly o-minimal type is dp-minimal, and every weakly quasi-o-minimal theory is dp-minimal. 
\end{Corollary}

In the rest of this section, we focus on indiscernible sequences of realizations of a weakly o-minimal type. In part (ii) of the following lemma, we show that they have a bit stronger property than the distality as defined by Simon in \cite{Simon2} (see \cite[Lemma 2.7, Corollary 2.9]{Simon2} and \cite[Definition 4.21]{EK}). 

\begin{Lemma}\label{Lemma_indiscernible_wom}
Let $(p,<)$ be a weakly o-minimal pair over $A$, let $a_0,a_1\models p$ and $B\supseteq A$. Suppose that $I$ and $J$ are sequences of realizations of $p$ such that $I+J$ is infinite and $<$-increasing. Then:
\begin{enumerate}[(i)]
    \item  If $I$ has no maximum, $I<a_0<a_1$ and $I+a_1$ is $B$-indiscernible, then $I+a_0$ is $B$-indiscernible. 
    \item If $I+J$ is $B$-indiscernible, $I$ has no maximum, $J$ has no minimum, and $I <a_0<  J$, then $I +a_0 + J$ is $B$-indiscernible.
    \item If $I+a_0+a_1+J$ is  $B$-indiscernible, $a\models p$  and $a_0<a< a_1$, then at least one of the sequences $I+a_0+a+J$ and $I+a+a_1+J$ is $B$-indiscernible. 
\end{enumerate} 
\end{Lemma}
\begin{proof}
 (i) Let $I_0$ be a finite subset of $I$. Since $I$ has no maximum, there exists $a \in I$ such that $I_0<a$. Then $a< a_0< a_1$ and the sequence $I_0+a+a_1$ is $B$-indiscernible, so $a\equiv a_1\ (BI_0)$ holds. Since $p$ is weakly o-minimal, the locus of type $\tp(a_1/BI_0)$ is a convex subset of $(p(\Mon),<)$ by Lemma \ref{Lemma_reldef_of_convex_components}(ii), so $a<a_0<a_1$ implies $a_0\equiv a_1\ (B I_0)$. In particular, the sequence $I_0+a_0$ is $B$-indiscernible.  Since this holds for all finite $I_0\subseteq I$,   $I+a_0$ is $B$-indiscernible.

(ii) Let $I_0\subseteq I$ and $J_0\subseteq J$ be finite. Choose $a_i\in I$ and $b_j\in J$ that satisfy $I_0< a_i$ and $b_j< J_0$. Then the sequence $I_0+a_i+b_j+J_0$ is $B$-indiscernible, and $a_i<a_0<b_j$.  By Lemma \ref{Lemma_reldef_of_convex_components}(ii) the locus of $\tp(a_i/BI_0J_0)=\tp(b_j/BI_0J_0)$ is convex in $(p(\Mon),<)$, so $a_0\equiv a_i\ (BI_0J_0)$, therefore the sequence $I_0+a_0+J_0$ is $B$-indiscernible. Since this holds for all finite $I_0\subseteq I$ and $J_0\subseteq J$, the sequence $I+a_0+J$ is $B$-indiscernible. 

(iii) Choose $a_{\frac{1}{2}}\models p$ such that $I+a_0+a_{\frac{1}{2}}+a_1+J$ is $B$-indiscernible; this is possible by compactness as $I+J$ is infinite. Then $a_0< a_{\frac{1}{2}}<a_{1}$, so $a$ is in one of the intervals $(a_0,a_{\frac{1}{2}}]$ and $[a_{\frac{1}{2}},a_{1})$. First, assume $a\in (a_{0},a_{\frac{1}{2}}]$. Notice that $a_{0}$ and $a_{\frac{1}{2}}$ realize the same type,  say $q$,  over $a_1BIJ$. By Lemma \ref{Lemma_reldef_of_convex_components}(ii), $q(\Mon)$ is a convex subset of $p(\Mon)$, so $a_{0}< a\leqslant a_{\frac{1}{2}}$  implies $a\models q$. In particular, $a\equiv a_{\frac{1}{2}}\,(a_1BIJ)$ implies that $I+a+a_1+J$ is $B$-indiscernible.  Similarly, $a\in[a_{\frac{1}{2}},a_1)$ implies the indiscernibility of $I+a_0+a+J$. 
\end{proof}

\begin{Corollary}\label{Cor_indiscernible_wom}
 Let $(p,<)$ be a weakly o-minimal pair over $A$ and $B\supseteq A$. Suppose that $I$ and $J$ are increasing sequences of realizations of $p$ such that $I+J$ is infinite and $B$-indiscernible. Let $a\models p$ satisfy $I<a<J$.  Then by removing at most one element except $a$ from the sequence $I+a+J$ the sequence remains $B$-indiscernible. More precisely, at least one of the following conditions holds:
\begin{enumerate}[(1)]
     \item $I+a+J$ is $B$-indiscernible;
     \item $a_0=\max I$ exists and the sequence $(I-a_0)+a+J$ is $B$-indiscernible;
     \item $a_1=\min J$ exists and the sequence $I+a+(J-a_1)$  is $B$-indiscernible.
\end{enumerate}
\end{Corollary}
\begin{proof} If $I$ has no maximum and $J$ has no minimum, then $I+a+J$ is indiscernible by Lemma \ref{Lemma_indiscernible_wom}(ii). If both $a_0=\max I$ and $a_1=\min J$ exist, then the sequence $(I-a_0)+a_0+a_1+(J-a_1)$ is $B$-indiscernible and $a_0<a<a_1$, so by Lemma \ref{Lemma_indiscernible_wom}(iii) at least one of conditions (2) and (3) holds. If $I$ has no maximum and $a_1=\min J$ exists, then $I<a<a_1$ and $I+a_1$ is $B(J-a_1)$-indiscernible, so condition (3) holds by Lemma \ref{Lemma_indiscernible_wom}(i); similarly,  if $a_0=\max I$ and $J$ has no minimum, then condition (2) holds. 
\end{proof}

\begin{Proposition}\label{Proposition_replace_a_i_by_ai'}  Let $\mathbf p=(p,<)$ be a weakly-o-minimal pair over $A$. Suppose that  $I=(a_j\mid j\in J)$ is a (possibly finite)  Morley sequence in $\mathbf p_r$ over $A$.
\begin{enumerate}[(i)]
\item Suppose that $a\models p$ is in the $<$-convex hull of $I$.  Then we can remove at most one element from $I$ and insert $a$ so that the sequence remains Morley over $A$. 

\item If $I'=(a_j'\mid j\in J)$ is a sequence of realizations of $p$  such that $a_j'\dep_A a_j$ for all $j\in J$, then $I'$ is a Morley sequence in $\mathbf p_r$ over $A$.
\end{enumerate}
\end{Proposition} 
\begin{proof} (i) By extending $I$ if necessary, we may assume that $I$ is infinite; also, we may assume $a\notin I$.  Let $ I^{-}=\{x\in I\mid x< a\}$ and $ I^{+}=\{x\in I\mid a< x\}$. Then $\{I^-,I^+\}$ is a partition of $I$ and $I^-<a<I^+$ also holds; the assumptions of Corollary \ref{Cor_indiscernible_wom} are satisfied, so by removing at most one element from $I^-+a+I^+$ except $a$ we get an $A$-indiscernible sequence. The obtained sequence is Morley over $A$ as it is $A$-indiscernible and contains an infinite subsequence which is Morley over $A$.

(ii) It suffices to prove the statement for a finite $I$. So, assume that $ I=(a_0,\dots, a_n)$ is a Morley sequence in $\mathbf p_r$ over $A$. 
Extend $ I$ to a longer Morley sequence $I_0=(a_{-1},a_0,\dots a_n,a_{n+1})$; we will show that the sequence $I_0'=(a_{-1},a_0',\dots a_n',a_{n+1})$ has the same type over $A$ as $I_0$. 
Note that $a_{k-1}<\mathcal D_p(a_k)< a_{k+1}$ holds for all $k=0,1,\dots,n$, so $a_k\dep_A  a_k'$ implies $a_{k-1}<a_k'< a_{k+1}$. By part (i), we can replace some element of $I_0$, say $a_i$, by $a_k'$ so that the sequence remains Morley. In particular, $a_{i-1}<a_k'< a_{i+1}$ is satisfied, so $i=k$. Therefore, replacing $a_k$ by $a_k'$ in $I_0$ does not change the type $\tp(I_0/A)$. Successively, we can replace all $a_k$'s and conclude $\tp(I_0/A)=\tp(I_0'/A)$. 
\end{proof}

\end{document}